\documentclass[12pt, leqno]{amsart}
\usepackage{amssymb,upgreek}
\usepackage{tikz,tikz-cd}
\usetikzlibrary{decorations.pathreplacing,decorations.pathmorphing}
\usetikzlibrary{arrows,arrows.meta}
\usetikzlibrary{shapes,shapes.geometric,shapes.symbols}
\usetikzlibrary{matrix,calc}
\usepackage[poly,all]{xy}
\usepackage[colorinlistoftodos, textwidth = 2.3cm]{todonotes}
\usepackage{graphicx}
\usepackage{ytableau}
\ytableausetup{mathmode, boxsize=1em}

\usepackage{enumitem,mathrsfs}

\usepackage{hyperref}
\hypersetup{
    colorlinks=true,
    linkcolor = blue,
    citecolor=magenta,
    urlcolor=cyan,
}
\newtheorem{theorem}{Theorem}[section]
\newtheorem{lemma}[theorem]{Lemma}

\theoremstyle{definition}
\newtheorem{definition}[theorem]{Definition}
\newtheorem{example}[theorem]{Example}

\newtheorem{proposition}[theorem]{Proposition}
\newtheorem{corollary}[theorem]{Corollary}

\theoremstyle{remark}
\newtheorem{remark}[theorem]{Remark}

\numberwithin{equation}{section}

\begin{document}
\title[The projective cover of 0-Hecke modules]{The projective cover of tableau-cyclic indecomposable $H_n(0)$-modules}

\author[S.-I. Choi]{Seung-Il Choi}
\address{Center for Quantum Structures in Modules and Spaces, Seoul National University, Seoul 08826, Republic of Korea}
\email{ignatioschoi@snu.ac.kr}

\author[Y.-H. Kim]{Young-Hun Kim}
\address{Department of Mathematics, Sogang University, Seoul 04107, Republic of Korea \& Research Institute for Basic Science, Sogang University, Seoul 04107, Republic of Korea}
\email{ykim.math@gmail.com}

\author[S.-Y. Nam]{Sun-Young Nam}
\address{Department of Mathematics, Sogang University, Seoul 04107, Republic of Korea}
\email{synam.math@gmail.com}

\author[Y.-T. Oh]{Young-Tak Oh}
\address{Department of Mathematics, Sogang University, Seoul 04107, Republic of Korea}
\email{ytoh@sogang.ac.kr}

\newcommand{\nc}{\newcommand}
\nc{\SG}{\mathfrak{S}}
\nc{\PCT}{\mathrm{PCT}}
\nc{\SPCT}{\mathrm{SPCT}}
\nc{\RT}{\mathrm{RT}}
\nc{\SRT}{\mathrm{SRT}}
\nc{\RCT}{\mathrm{RCT}}
\nc{\SRCT}{\mathrm{SRCT}}
\nc{\SYCT}{\mathrm{SYCT}}
\nc{\SPYCT}{\mathrm{SPYCT}}
\nc{\stan}{\mathrm{stan}}
\nc{\Span}{\mathrm{span}}
\nc{\comp}{\mathrm{comp}}
\nc{\rmst}{\mathrm{st}}
\nc{\Des}{\mathrm{Des}}
\nc{\set}{\mathrm{set}}
\nc{\wt}{\mathrm{wt}}
\nc{\ch}{\mathrm{ch}}
\nc{\id}{\mathrm{id}}
\nc{\Sym}{\mathrm{Sym}}
\nc{\Qsym}{\mathrm{QSym}}
\nc{\Nsym}{\mathrm{NSym}}
\nc{\sh}{\mathrm{sh}}
\nc{\bfS}{\mathbf{S}}
\nc{\bfm}{\mathbf{m}}
\nc{\hbfS}{\widehat{\mathbf{S}}}
\nc{\bfF}{\mathbf{F}}
\nc{\calS}{\mathcal{S}}
\nc{\hcalS}{\widehat{\mathcal{S}}}
\nc{\alphamax}{\alpha_{\rm max}}
\nc{\brho}{\overline{\rho}}
\nc{\bphi}{\overline{\phi}}
\nc{\calV}{\mathcal{V}}
\nc{\calR}{\mathcal{R}}
\nc{\calG}{\mathcal{G}}
\nc{\tal}{\lambda(\alpha)}
\nc{\tbe}{\widetilde{\beta}}
\nc{\opi}{\overline{\pi}}
\nc{\calP}{\mathcal{P}}
\nc{\rmtop}{\mathrm{top}}
\nc{\rad}{\mathrm{rad}}
\nc{\bfP}{\mathbf{P}}
\nc{\SET}{\mathrm{SET}}
\nc{\SIT}{\mathrm{SIT}}
\nc{\rev}{\mathrm{r}}
\nc{\Th}{\theta}
\nc{\htau}{\widehat{\tau}}
\nc{\mPhi}{\Phi}
\nc{\mphi}{\phi}
\nc{\mPsi}{\Psi}
\nc{\hmPsi}{\widehat{\Psi}}
\nc{\mpsi}{\psi}
\nc{\mGam}{\Gamma}
\nc{\tcd}{\mathtt{cd}}
\nc{\trd}{\mathtt{rd}}
\nc{\trcd}{\mathtt{rcd}}
\nc{\rmr}{\mathrm{r}}
\nc{\rmc}{\mathrm{c}}
\nc{\rmt}{\mathrm{t}}
\nc{\bubact}{\,\scalebox{0.6}{$\bullet$}\,}
\nc{\hbubact}{\,\scalebox{0.6}{$\widehat{\bullet}$}\,}
\nc{\col}{\rm col}
\nc{\row}{\rm row}
\nc{\calE}{\mathcal{E}}
\nc{\calT}{\mathscr{T}}
\nc{\sfT}{\mathsf{T}}
\nc{\calEsa}{\mathcal{E}^\sigma(\alpha)}
\nc{\tauC}{\tau_{\scalebox{0.5}{$C$}}}
\nc{\sytabC}{\sytab_{\scalebox{0.5}{$C$}}}
\nc{\bbfP}{\overline{\bfP}}
\nc{\pr}{\mathbf{pr}}
\nc{\Ups}{\Upsilon}
\nc{\pact}{\diamond}
\nc{\tauE}{\tau_{\scalebox{0.5}{$E$}}}
\nc{\tauF}{\tau_{\scalebox{0.5}{$F$}}}
\nc{\tauG}{\tau_{\scalebox{0.5}{$G$}}}
\nc{\rtE}{T_{\scalebox{0.5}{$E$}}}
\nc{\rtF}{T_{\scalebox{0.5}{$F$}}}
\nc{\rtG}{T_{\scalebox{0.5}{$G$}}}
\nc{\oPaE}{\overline{\Phi}_{\alpha_E}}
\nc{\oPaF}{\overline{\Phi}_{\alpha_F}}
\nc{\oPaG}{\overline{\Phi}_{\alpha_G}}
\nc{\tab}{\tau}
\nc{\sytab}{\widehat{\tau}}
\nc{\hatE}{\widehat{E}}
\nc{\hcalE}{\widehat{\calE}}
\nc{\hatC}{\widehat{C}}
\nc{\bal}{{\boldsymbol{\upalpha}}}
\nc{\bbe}{{\boldsymbol{\upbeta}}}
\nc{\bgam}{{\boldsymbol{\upgamma}}}
\nc{\bdel}{{\boldsymbol{\updelta}}}
\nc{\weakcon}{\odot}
\nc{\calB}{\mathcal{B}}
\nc{\calM}{\mathcal{M}}
\nc{\calN}{\mathcal{N}}
\newcommand{\g}{\mathfrak{g}}
\newcommand{\la}{\lambda}
\newcommand{\h}{\mathfrak{h}}
\newcommand{\te}{\widetilde{e}}
\newcommand{\tf}{\widetilde{f}}
\newcommand{\mf}{\mathfrak}
\newcommand{\cP}{\mathscr{P}}
\newcommand{\gl}{\mathfrak{gl}}
\nc{\ldalpha}{\lambda(\alpha)}
\nc{\SRIT}{\mathrm{SRIT}}
\nc{\re}{\mathrm{rev}}
\nc{\otau}{\overline{\tau}}
\nc{\rtop}{{\rm top}}
\nc{\sfc}{\mathsf{c}}
\nc{\sfr}{\mathsf{r}}
\nc{\tH}{\mathtt{H}}
\nc{\tV}{\mathtt{V}}

\nc{\SPCTsa}{\SPCT^\sigma(\alpha)}
\nc{\bfSsa}{\bfS_\alpha^\sigma}
\nc{\bfSsaC}{{\bfS}^\sigma_{\alpha,C}}
\nc{\hbfSsa}{\widehat{\bfS}_\alpha^\sigma}
\nc{\upineq}{\rotatebox{90}{$<$}}
\nc{\downineq}{\rotatebox{270}{$<$}}
\nc{\diagineq}{\rotatebox{135}{$<$}}

\nc{\ra}{\rightarrow}

\def\C{\mathbb C}
\def\Z{\mathbb Z}

\newcommand{\precdot}{\prec\mathrel{\mkern-3mu}\mathrel{\cdot}}
\newcommand{\succdot}{\hskip3pt\cdot\mathrel{\mkern-10mu}\mathrel{\succ}}
\def\rddots#1{\cdot^{\cdot^{\cdot^{#1}}}}



\date{\today}

\begin{abstract}
Let $\alpha$ be a composition of $n$ and $\sigma$ a permutation in $\SG_{\ell(\alpha)}$. 
This paper concerns the projective covers of $H_n(0)$-modules $\calV_\alpha$, $X_\alpha$, and 
$\bfS^\sigma_{\alpha}$
whose images under the quasisymmetric characteristic are
the dual immaculate quasisymmetric function, the extended Schur function, and the quasisymmetric Schur function when $\sigma$ is the identity, respectively.
First, we show that the projective cover of $\calV_\alpha$ is the projective indecomposable module $\bfP_\alpha$ due to Norton, and
$X_\alpha$ and the $\phi$-twist of
the canonical submodule $\bfS^{\sigma}_{\beta,C}$ of $\bfS^\sigma_{\beta}$ for $(\beta,\sigma)$'s satisfying suitable conditions
appear as homomorphic images of $\calV_\alpha$.
Second, we introduce a combinatorial model for the $\phi$-twist of $\bfS^\sigma_{\alpha}$ and derive a series of surjections starting from $\bfP_\alpha$ to the $\phi$-twist of $\bfS^\id_{\alpha,C}$.
Finally, we construct the projective cover of 
every indecomposable direct summand $\bfS^\sigma_{\alpha, E}$ of $\bfS^\sigma_{\alpha}$.
As a byproduct, we give a characterization of triples $(\sigma,\alpha,E)$ such that
the projective cover of $\bfS^\sigma_{\alpha,E}$ is indecomposable.
\end{abstract}

\maketitle
\tableofcontents


\section{Introduction}
In 1979, Norton~\cite{79Norton} classified all projective indecomposable $H_n(0)$-modules $\bfP_\alpha~(\alpha \models n)$
up to equivalence, which were given by $2^{n-1}$ left ideals generated by mutually orthogonal primitive idempotents.
These modules were again studied intensively in the 2000s (for instance, see~\cite{11Denton, 06HNT, 16Huang}).
In particular, Huang~\cite{16Huang} described their induced modules $\bfP_{\bal}$ in terms of standard ribbon tableaux of shape $\bal$, where $\bal$ ranges over the set of generalized compositions.

In the middle of 1990s, there was a breakthrough in the representation theory of 0-Hecke algebras.
To be precise, letting $\calG_0(H_n(0))$ be the Grothendieck group of the category of finitely generated $H_n(0)$-modules, 
their direct sum for all $n\ge 0$ endowed with the induction product is isomorphic to the ring $\Qsym$ of quasisymmetric functions~\cite{96DKLT, 97KT}
via the {\em quasisymmetric characteristic}
\begin{align*}
\ch : \bigoplus_{n \ge 0} \calG_0(H_n(0)) \ra \Qsym. \end{align*}
From the viewpoint of this correspondence, 
when we have a family of notable quasisymmetric functions,
it would be of great importance to investigate if these elements appear as the image of the isomorphism classes of certain modules with nice properties under the quasisymmetric characteristic. 
The related studies have been done for the quasisymmetric Schur functions and their permuted version~\cite{15TW, 19TW}, the dual immaculate quasisymmetric functions~\cite{15BBSSZ}, and the extended Schur functions~\cite{19Searles}, all of which form a basis of $\Qsym$.

Let us explain these results in more detail.
From now on, we fix $\alpha \models n$ and $\sigma \in \SG_{\ell(\alpha)}$, where $\ell(\alpha)$ is the length of $\alpha$.

Let $\SRCT(\alpha)$ be the set of \emph{standard reverse composition tableaux of shape $\alpha$}.
In ~\cite{15TW}, Tewari and van Willigenburg defined an $H_n(0)$-action on $\SRCT(\alpha)$ and showed that 
the resulting $H_n(0)$-module, denoted by  $\bfS_\alpha$, has the \emph{quasisymmetric Schur function} $\calS_\alpha$
as the image under the quasisymmetric characteristic. 

As a far-reaching generalization, they \cite{19TW} introduced new 
combinatorial objects, called \emph{standard permuted composition tableaux of shape $\alpha$ and type $\sigma$}, by weakening the condition on the first column in standard reverse composition tableaux.
Let $\SPCT^\sigma(\alpha)$ be the set of standard permuted composition tableaux of shape $\alpha$ and type $\sigma$. 
It was shown in~\cite{19TW} that $\SPCT^\sigma(\alpha)$ has an $H_n(0)$-action and the resulting $H_n(0)$-module, denoted by  $\bfS_\alpha^\sigma$, shares many properties with $\bfS_\alpha$. 
For instance, in the same way as in \cite{15TW}, one can give a colored graph structure on $\SPCT^\sigma(\alpha)$.
Let $\calE^\sigma(\alpha)$ be the set of connected components and 
$\bfS^\sigma_{\alpha,E}$ be the $H_n(0)$-submodule of $\bfS^\sigma_\alpha$ whose underlying space is the $\C$-span of $E$ for $E \in \calE^\sigma(\alpha)$.
Then, as $H_n(0)$-modules, 
\begin{align*}
\bfS_\alpha^{\sigma} \cong \bigoplus_{E \in \mathcal{E}^\sigma(\alpha)} \bfS_{\alpha,E}^{\sigma}
\end{align*}
and each direct summand $\bfS_{\alpha,E}^{\sigma}$ is generated by a single tableau.
In this paper, we say that an $H_n(0)$-module is \emph{tableau-cyclic} if it is generated by a single tableau.

Very recently, in~\cite{20CKNO}, an extensive study has been done for these modules.
The authors, following the same manner as in~\cite{19Konig}, show that every direct summand $\bfS_{\alpha,E}^{\sigma}$ is indecomposable. 
Further, they characterize when $\bfS^\sigma_\alpha$ is indecomposable and show that $\bfP_{\alpha \cdot \sigma^{-1}}$ is the projective cover of $\bfS^\sigma_{\alpha,C}$ for the canonical class $C$ of $\mathcal{E}^\sigma(\alpha)$.
For the definition of $\alpha \cdot \sigma^{-1}$, see Subsection~\ref{subsec: comp and diag}.

Next, let us review the result related to the \emph{dual immaculate quasisymmetric functions}.
Noncommutative Bernstein operators were introduced by Berg {\it et al.}~\cite{14BBSSZ}.
Applied to the identity element, they yield the immaculate functions, which forms a basis of the the ring of noncommutative symmetric functions.
Soon after, using the combinatorial objects called standard immaculate tableaux, they constructed tableau-cyclic indecomposable $H_n(0)$-modules $\calV_\alpha$
whose quasisymmetric characteristics are the quasisymmetric functions which are dual to immaculate functions (see~\cite{15BBSSZ}).

Finally, let us review the result related to the \emph{extended Schur functions}. %
In~\cite{19AS}, Assaf and Searles defined the extended Schur functions as the stable limits of lock polynomials and showed that they form a basis of $\Qsym$.
And, using standard extended tableaux, Searles~\cite{19Searles} constructed tableau-cyclic indecomposable $H_n(0)$-modules $X_\alpha$ whose quasisymmetric characteristics are the extended Schur functions.

This paper mainly concerns projective covers of the modules mentioned in the above.   
As is well known, a projective cover of a finitely generated module $M$, which is the best approximation of $M$ by a projective module, plays an extremely important role in understanding the structure of $M$.
For instance, finding the projective cover is a key step to construct the minimal projective presentation (see~\cite{95ARS}).

The first objective of the present paper is to demonstrate the relationship among $\calV_\alpha$,
$X_\alpha$ and $\bfS^\sigma_{\alpha,C}$.
We construct a series of surjections
\begin{equation}\label{eq: surj series 1 intro}
\begin{tikzcd}
\bfP_{\alpha} \arrow[r, two heads] & 
\calV_\alpha \arrow[r, two heads] & 
X_\alpha \arrow[r, two heads] & 
\mphi[\bfS^\sigma_{\tal,C}],
\end{tikzcd}
\end{equation}
where $\sigma$ ranges over the set of permutations in $\SG_{\ell(\alpha)}$ satisfying that 
$\lambda(\alpha) = \alpha^\rmr \cdot \sigma$
and $\mphi[\bfS^\sigma_{\tal,C}]$ is the $\phi$-twist of $\bfS^\sigma_{\tal,C}$ 
(see~\eqref{eq: phi twist}).
The series~\eqref{eq: surj series 1 intro} enables us to
deal with these modules in a uniform way although they appear in the different contexts. 
Then we give a surjection from $\bfSsaC$ to $\bfS^{\sigma s_i}_{\alpha \cdot s_i,C}$ when $\ell(\sigma s_i) < \ell(\sigma)$ (see Theorem~\ref{thm: surj btw canonical}).
Using this result repeatedly, we finally arrive at a series of surjections starting from $\bfP_\alpha$ and ending at $\phi[\bfS^{\id}_{\alpha^\rmr,C}]$.

On the other hand, one can describe the $\phi$-twist of $\bfS^{\sigma}_{\alpha}$ in a combinatorial way using \emph{standard permuted Young composition tableaux} (see Definition~\ref{def: SPYCT}).
Giving a suitable $H_n(0)$-action on these tableaux, we obtain a new $H_n(0)$-module $\hbfS^\sigma_{\alpha}$, 
which is isomorphic to $\mphi[\bfS^{\sigma^{w_0}}_{\alpha^\rmr}]$ with $\sigma^{w_0} = w_0 \sigma w_0^{-1}$.
Using these modules, we derive a series of surjections starting from $\bfP_\alpha$ and ending at $\hbfS^{\id}_{\alpha,C}$.
More precisely, for any reduced expression $s_{i_1}\cdots s_{i_k}$ of $\sigma$, we have
\begin{equation}\label{eq: surj seq}
\begin{tikzcd}
\bfP_{\alpha}  \arrow[r, two heads] &
\calV_{\alpha} \arrow[r, two heads] &
X_{\alpha} \arrow[r,  two heads] &
\hbfS^{\sigma}_{\ldalpha^\rmr,C} \arrow[r, two heads] &
\hbfS^{\sigma s_{i_{k}}}_{\ldalpha^\rmr \cdot s_{i_k},C} \arrow[r, two heads] &
\cdots \arrow[r, two heads] &
\hbfS^{\id}_{\alpha,C}
\end{tikzcd}
\end{equation}
(Corollary~\ref{Coro:Seq_of_PVXS}).

The second objective of the present paper is to find the projective cover of 
$\bfS^\sigma_{\alpha,E}$ for all classes $E \in \mathcal{E}^\sigma(\alpha)$.
As an initial step, 
we decompose the composition diagram of $\alpha$ into horizontal strips completely determined by the descents of the source tableau $\tauE$ in $E$ and then 
construct a generalized ribbon diagram by placing them vertically in a suitable way. 
Let $\bal_E$ be the shape of this generalized ribbon diagram. 
Then we construct a surjective $H_n(0)$-module homomorphism
$\eta: \bbfP_{\bal_E} \to \bfS^\sigma_{\alpha,E}$
(Theorem~\ref{Thm:Surjective homo}).
Here $\bbfP_{\bal_E}$ is a slight modification of $\bfP_{\bal_E}$,
which is isomorphic to $\bigoplus_\beta \bfP_{\beta^\rmc}$ if $\bfP_{\bal_E} = \bigoplus_\beta\bfP_{\beta}$.
For details, see Subsection~\ref{subsec: PIM}.

We next show that $\eta$ is an essential epimorphism, which is the most nontrivial part in the present paper.
This can be achieved by showing that $\ker(\eta)$ is contained in $\rad(\bbfP_{\bal_E})$.
To do this, we first give an explicit description of $\ker(\eta)$ which is the $\C$-span of standard ribbon tableaux of shape $\bal_E$ satisfying the four conditions in~\eqref{Eq: Description of kernel}.
Then we give a sufficient condition for a standard ribbon tableau of shape $\bal_E$ to be contained the radical of $\bbfP_{\bal_E}$ (Lemma~\ref{Lem: To be contained in Rad}).
Finally, using the description of $\ker(\eta)$, we show that the tableaux generating $\ker(\eta)$ satisfies the sufficient condition above.
Consequently we conclude that $\bbfP_{\bal_E}$ is the projective cover of $\bfS^\sigma_{\alpha,E}$ (Theorem~\ref{Thm:projective cover S}).
As a byproduct of this theorem, we give a characterization of triples $(\sigma,\alpha,E)$ such that
the projective cover of $\bfS^\sigma_{\alpha,E}$ is indecomposable (Corollary~\ref{Cor: Classification covers}).

This paper is organized as follows.
In Section~\ref{Preliminaries}, we collect the materials which are required to develop our arguments. 
In Section~\ref{Section3}, 
we show that the projective cover of $\calV_\alpha$ is a projective indecomposable module, and $X_\alpha$ and the $\phi$-twist of the canonical submodule $\bfS^{\sigma}_{\beta,C}$ of $\bfS^\sigma_{\beta}$ for $(\beta,\sigma)$'s satisfying suitable conditions appear as homomorphic images of $\calV_\alpha$.
In Section~\ref{sec: surj from bfSsaC},
we introduce a combinatorial model for the $\phi$-twist of $\bfS^\sigma_{\alpha}$ and derive a series of surjections in~\eqref{eq: surj seq}.
The final section is devoted to the construction of the projective cover of $\bfS^\sigma_{\alpha,E}$ for all classes $E$.

\section{Preliminaries}
\label{Preliminaries}

In this section, $n$ denotes a nonnegative integer. 
Define $[n]$ to be $\{1,\ldots, n\}$ if $n > 0$ and $\emptyset$ else. 
In addition, we set $[-1]:=\emptyset$.
For positive integers $i\le j$, set $[i,j]:=\{i,i+1,\ldots, j\}$.

\subsection{Compositions and their diagrams}\label{subsec: comp and diag}
A \emph{composition} $\alpha$ of a nonnegative integer $n$, denoted by $\alpha \models n$, is a finite ordered list of positive integers $(\alpha_1,\ldots, \alpha_k)$ satisfying $\sum_{i=1}^k \alpha_i = n$.
For each $1 \le i \le k$, let us call $\alpha_i$ a \emph{part} of $\alpha$. And we call $k =: \ell(\alpha)$ the \emph{length} of $\alpha$ and $n =:|\alpha|$ the \emph{size} of $\alpha$. For convenience we define the empty composition $\emptyset$ to be the unique composition of size and length $0$.

Given $\alpha = (\alpha_1,\ldots,\alpha_k) \models n$ and $I = \{i_1 < \cdots < i_k\} \subset [n-1]$, 
let 
\begin{align*}
&\set(\alpha) := \{\alpha_1,\alpha_1+\alpha_2,\ldots, \alpha_1 + \alpha_2 + \cdots + \alpha_{k-1}\}, \\
&\comp(I) := (i_1,i_2 - i_1,\ldots,n-i_k).
\end{align*}
The set of compositions of $n$ is in bijection with the set of subsets of $[n-1]$ under the correspondence $\alpha \mapsto \set(\alpha)$ (or $I \mapsto \comp(I)$).

If $\alpha = (\alpha_1,\ldots, \alpha_k) \models n$ is such that $\alpha_1 \ge \cdots \ge \alpha_k$, then we say that $\alpha$ is \emph{partition} of $n$ and denote this by $\alpha \vdash n$. The partition obtained by sorting the parts of $\alpha$ in the weakly decreasing order is denoted by $\tal$.

Let $\alpha = (\alpha_1,\ldots, \alpha_k) \models n$. We define the \emph{composition diagram} $\tcd(\alpha)$ (resp. \emph{reverse composition diagram} $\trcd(\alpha)$) of $\alpha$ by a left-justified array of $n$ boxes where the $i$th row from the top (resp. from the bottom) has $\alpha_i$ boxes for $1 \le i \le k$.
We also define the \emph{ribbon diagram} $\trd(\alpha)$ of $\alpha$ by the connected skew diagram without $2 \times 2$ boxes, such that the $i$th row from the bottom has $\alpha_i$ boxes. 
For example, when $\alpha = (3,1,2)$, we have
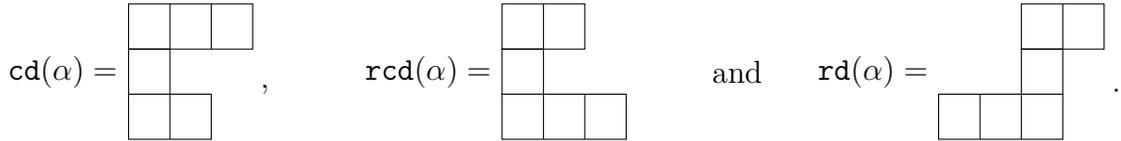
\begin{displaymath}
\begin{tikzpicture}
\def\hhh{4mm}
\def\vvv{5mm}
\def\www{0.20mm}
\draw[-] (\hhh*0,-\vvv*1) rectangle (\hhh*3,\vvv*0);
\draw[-] (\hhh*1,-\vvv*1) -- (\hhh*1,\vvv*0);
\draw[-] (\hhh*2,-\vvv*1) -- (\hhh*2,\vvv*0);
\node[left] at (\hhh*0,-\vvv*1.5) {$\tcd(\alpha)=$};
\draw[-] (\hhh*0,-\vvv*2) rectangle (\hhh*1,-\vvv*1);
\draw[-] (\hhh*0,-\vvv*3) rectangle (\hhh*2,-\vvv*2);
\draw[-] (\hhh*1,-\vvv*3) -- (\hhh*1,-\vvv*2);
\node[] at (\hhh*3.3,-\vvv*1.9) {,};
\end{tikzpicture}
\hskip 10mm
\begin{tikzpicture}
\def\hhh{4mm}
\def\vvv{5mm}
\def\www{0.20mm}
\node[left] at (\hhh*0,-\vvv*1.5) {$\trcd(\alpha)=$};
\draw[-] (\hhh*0,-\vvv*1) rectangle (\hhh*2,\vvv*0);
\draw[-] (\hhh*1,-\vvv*1) -- (\hhh*1,\vvv*0);
\draw[-] (\hhh*0,-\vvv*2) rectangle (\hhh*1,-\vvv*1);
\draw[-] (\hhh*0,-\vvv*3) rectangle (\hhh*3,-\vvv*2);
\draw[-] (\hhh*1,-\vvv*3) -- (\hhh*1,-\vvv*2);
\draw[-] (\hhh*2,-\vvv*3) -- (\hhh*2,-\vvv*2);
\end{tikzpicture}
\hskip 10mm
\begin{tikzpicture}
\def\hhh{4mm}
\def\vvv{5mm}
\def\www{0.20mm}
\node[left] at (-\hhh*4,-\vvv*1.5) {and};
\node[left] at (\hhh*0,-\vvv*1.5) {$\trd(\alpha)=$};
\draw[-] (\hhh*2,-\vvv*1) rectangle (\hhh*4,\vvv*0);
\draw[-] (\hhh*3,-\vvv*1) -- (\hhh*3,\vvv*0);
\draw[-] (\hhh*2,-\vvv*2) rectangle (\hhh*3,-\vvv*1);
\draw[-] (\hhh*0,-\vvv*3) rectangle (\hhh*3,-\vvv*2);
\draw[-] (\hhh*1,-\vvv*3) -- (\hhh*1,-\vvv*2);
\draw[-] (\hhh*2,-\vvv*3) -- (\hhh*2,-\vvv*2);
\node[] at (\hhh*4.3,-\vvv*1.9) {.};
\end{tikzpicture}
\end{displaymath}

A box is said to be in the $i$th row if it is in the $i$th row from the top, and in the $j$th column if it is in the $j$th column from the left.
We use $(i,j)$ to denote the box in the $i$th row and $j$th column.
For a filling $\tau$ of $\tcd(\alpha)$ or $\trcd(\alpha)$, we denote by $\tau_{i,j}$ the entry at $(i,j)$ in $\tau$.

For a composition $\alpha = (\alpha_1,\ldots,\alpha_k) \models n$, let $\alpha^\rmr$ denote the composition $(\alpha_k, \ldots, \alpha_1)$ and $\alpha^\rmc$ the unique composition satisfying that $\set(\alpha^c) = [n-1] \setminus \set(\alpha)$, and let $\alpha^\rmt := (\alpha^\rmr)^\rmc = (\alpha^\rmc)^\rmr$. Note that $\trd(\alpha^\rmt)$ is obtained by reflecting $\trd(\alpha)$ along the diagonal. 

The symmetric group $\SG_k$ acts (on the right) on the set of compositions of length $k$ by place permutation. 
In particular, $\alpha \cdot w_0 = \alpha^\rmr$, where $w_0$ is the longest element in $\SG_k$.

\subsection{The $0$-Hecke algebra and the quasisymmetric characteristic}
To begin with, we recall that the symmetric group $\SG_n$ is generated by simple transpositions $s_i := (i,i+1)$ with $1 \le i \le n-1$.
An expression for $\sigma \in \SG_n$ of the form $s_{i_1} \cdots s_{i_p}$ that uses the minimal number of simple transpositions is called a \emph{reduced expression} for $\sigma$. 
The number of simple transpositions in any reduced expression for $\sigma$, denoted by $\ell(\sigma)$, is called the \emph{length} of $\sigma$.

The $0$-Hecke algebra $H_n(0)$ is the $\C$-algebra generated by $\pi_1,\ldots,\pi_{n-1}$ subject to the following relations:
\begin{align*}
\pi_i^2 &= \pi_i \quad \text{for $1\le i \le n-1$},\\
\pi_i \pi_{i+1} \pi_i &= \pi_{i+1} \pi_i \pi_{i+1}  \quad \text{for $1\le i \le n-2$},\\
\pi_i \pi_j &=\pi_j \pi_i \quad \text{if $|i-j| \ge 2$}.
\end{align*}
For each $1 \le i \le n-1$, let $\opi_i := \pi_i -1$. Then $\{\opi_i \mid i = 1,\ldots,n-1\}$ is also a generating set of $H_n(0)$, which satisfies the following relations:
\begin{align*}
\opi_i^2 &= -\opi_i \quad \text{for $1\le i \le n-1$},\\
\opi_i \opi_{i+1} \opi_i &= \opi_{i+1} \opi_i \opi_{i+1}  \quad \text{for $1\le i \le n-2$},\\
\opi_i \opi_j &=\opi_j \opi_i \quad \text{if $|i-j| \ge 2$}.
\end{align*}

Pick up any reduced expression $s_{i_1}\cdots s_{i_p}$ for a permutation $\sigma \in \SG_n$. Then the elements $\pi_{\sigma}$ and $\opi_{\sigma}$ of $H_n(0)$ are defined by
\[
\pi_{\sigma} := \pi_{i_1} \cdots \pi_{i_p} \quad \text{and} \quad \opi_{\sigma} := \opi_{i_1} \cdots \opi_{i_p}.
\]
It is well known that these elements are independent of the choice of reduced expressions, and both $\{\pi_\sigma \mid \sigma \in \SG_n\}$ and $\{\opi_\sigma \mid \sigma \in \SG_n\}$ are bases for $H_n(0)$.

Let $\calR(H_n(0))$ denote the $\Z$-span of the isomorphism classes of finite dimensional representations of $H_n(0)$. The isomorphism class corresponding to an $H_n(0)$-module $M$ will be denoted by $[M]$. The \emph{Grothendieck group} $\calG_0(H_n(0))$ is the quotient of $\calR(H_n(0))$ modulo the relations $[M] = [M'] + [M'']$ whenever there exists a short exact sequence $0 \ra M' \ra M \ra M'' \ra 0$. The irreducible representations of $H_n(0)$ form a free $\Z$-basis for $\calG_0(H_n(0))$. Let
\[
\calG := \bigoplus_{n \ge 0} \calG_0(H_n(0)).
\]
According to \cite{79Norton}, there are $2^{n-1}$ distinct irreducible representations of $H_n(0)$. They are naturally indexed by compositions of $n$. Let $\bfF_{\alpha}$ denote the $1$-dimensional $\C$-vector space corresponding to the composition $\alpha \models n$, spanned by a vector $v_{\alpha}$. 
For each $1\le i \le n-1$, define an action of the generator $\pi_i$ of $H_n(0)$ as follows:
\[
\pi_i(v_\alpha) = \begin{cases}
0 & i \in \set(\alpha),\\
v_\alpha & i \notin \set(\alpha).
\end{cases}
\]
Then $\bfF_\alpha$ is an irreducible $1$-dimensional $H_n(0)$-representation.

In the following, let us review the connection between $\calG$ and quasisymmetric functions.
Quasisymmetric functions are power series of bounded degree in variables $x_{1},x_{2},x_{3},\ldots$  with coefficients in $\Z$, which are shift invariant in the sense that the coefficient of the monomial $x_{1}^{\alpha _{1}}x_{2}^{\alpha _{2}}\cdots x_{k}^{\alpha _{k}}$ is equal to the coefficient of the monomial $x_{i_{1}}^{\alpha _{1}}x_{i_{2}}^{\alpha _{2}}\cdots x_{i_{k}}^{\alpha _{k}}$ for any strictly increasing sequence of positive integers $i_{1}<i_{2}<\cdots <i_{k}$ indexing the variables and any positive integer sequence $(\alpha _{1},\alpha _{2},\ldots ,\alpha _{k})$ of exponents.

The ring $\Qsym$ of quasisymmetric functions is decomposed as
\[
\Qsym =\bigoplus _{n\geq 0} \Qsym_{n},
\]
where $\Qsym_{n}$ is the $\Z$-span of all quasisymmetric functions that are homogeneous of degree $n$.

Given a composition $\alpha$, the \emph{fundamental quasisymmetric function} $F_\alpha$ is defined by $F_\emptyset = 1$ and
\[
F_\alpha = \sum_{\substack{1 \le i_1 \le \cdots \le i_k \\ i_j < i_{j+1} \text{ if } j \in \set(\alpha)}} x_{i_1} \cdots x_{i_k}.
\]
For every nonnegative integer $n$, it is known that $\{F_\alpha \mid \alpha \models n\}$ is a basis for $\Qsym_n$.

It was shown in \cite{96DKLT} that, when $\calG$ is equipped with induction product, the linear map
\begin{align*}
\ch : \calG \ra \Qsym, \quad [\bfF_{\alpha}] \mapsto F_{\alpha},
\end{align*}
called \emph{quasisymmetric characteristic}, is a ring isomorphism.
One can see that, by the definition of Grothendieck group, if we have a short exact sequence $0 \ra M' \ra M \ra M'' \ra 0$ of finite dimensional $H_n(0)$-modules, then
\[
\ch([M]) = \ch([M']) + \ch([M'']).
\]

\subsection{The $H_n(0)$-action on standard ribbon tableaux}\label{subsec: PIM}

A \emph{generalized composition} $\bal$ of $n$, denoted by $\bal \models n$, is a formal composition $\alpha^{(1)} \oplus \cdots \oplus  \alpha^{(k)}$,
where $\alpha^{(i)} \models n_i$ for positive integers $n_i$'s with $n_1 + \cdots + n_k = n$.
Then the {\em generalized ribbon diagram} $\trd(\bal)$ of $\bal$ is 
a skew diagram whose connected components are $\trd(\alpha^{(1)}), \ldots, \trd(\alpha^{(k)})$ such that $\trd(\alpha^{(i+1)})$ is strictly to the northeast of $\trd(\alpha^{(i)})$ for $i = 1,\ldots, k-1$.
For instance, if $\bal = (2,1) \oplus (2)$, then
\[
\trd(\bal) = 
\begin{array}{l}
\begin{ytableau}
\none & \none & ~ & ~ \\
\none & ~  \\
~ & ~
\end{ytableau}
\end{array}.
\]

Let $\alpha = (\alpha_{1}, \ldots, \alpha_{k})$ and $\beta= (\beta_{1},\ldots, \beta_{l})$ be compositions.
Let $\alpha \cdot \beta$ be the \emph{concatenation}
and $\alpha \odot \beta$ the \emph{near concatenation} of $\alpha$ and $\beta$.
In other words,
$ \alpha \cdot \beta = (\alpha_1, \ldots, \alpha_k, \beta_1,\ldots, \beta_l)$ and 
$\alpha \odot \beta = (\alpha_1,\ldots, \alpha_{k-1},\alpha_k + 
\beta_1,\beta_2, \ldots, \beta_l)$.
For a generalized composition 
$\bal = \alpha^{(1)} \oplus \cdots \oplus  \alpha^{(m)}$,
we define $[\bal]$ to be the set consisting of $2^{m-1}$ compositions of the form
$$
\alpha^{(1)} \  \square \  \alpha^{(2)} \ \square \ \cdots \  \square  \ \alpha^{(m)},
\quad 
\text{where }
\square \in \{\odot, \cdot\}.
$$

\begin{definition}\label{def: SRT}
For a generalized composition $\bal \models n$, a \emph{standard ribbon tableau} (SRT) of shape $\bal$ is a filling of $\trd(\bal)$ by $1,2,\ldots,n$ without repetition such that every row increases from left to right and every column increases from top to bottom.
\end{definition}
Let $\SRT(\bal)$ denote the set of all standard ribbon tableaux of shape $\bal$.
Define an $H_n(0)$-action on the $\C$-span of $\SRT(\bal)$ by
\begin{align}\label{eq: action for ribbon}
\opi_i \cdot T = \begin{cases}
-T, & \text{if $i$ is strictly above $i+1$ in $T$},\\
0, & \text{if $i$ and $i+1$ are in the same row of $T$},\\
s_i \cdot T, & \text{if $i$ is strictly below $i+1$ in $T$}
\end{cases}
\end{align}
for $1\le i \le n-1$ and $T \in \SRT(\bal)$. 
Here $s_i \cdot T$ is obtained from $T$ by swapping $i$ and $i+1$. 
The resulting module is denoted by $\bfP_\bal$.

It is known that the $\bfP_\bal$ is projective for every generalized composition $\bal$.
To explain this, let us recall Norton's results on projective modules~\cite{79Norton}.
For $I \subseteq [n-1]$, let $I^\rmc := [n-1] \setminus I$. The \emph{parabolic subgroup} $\SG_{n,I}$ is the subgroup of $\SG_n$ generated by $\{s_i \mid i\in I\}$. We denote by $w_0(I)$ the longest element in $\SG_{n,I}$. 
Norton decomposed the regular representation of $H_n(0)$ into the direct sum of $2^{n-1}$ indecomposable submodules $\calP_I$ $(I \subset [n-1])$, which are defined by
\begin{align*}
\calP_I := H_n(0)\cdot \opi_{w_0(I)}\pi_{w_0(I^\rmc)}.
\end{align*}
Let $\rmtop(\calP_I)$ denote the top of $\calP_I$, that is, $\rmtop(\calP_I) := \calP_I / \rad \; \calP_I$.
It is known that, for each $I \subset [n-1]$, $\rmtop(\calP_I)$ is isomorphic to $\bfF_\alpha$ with $\set(\alpha) = I$.
The set $\{\calP_I \mid I \subseteq [n-1]\}$ is a complete list of non-isomorphic projective indecomposable $H_n(0)$-modules.

The following result is due to Huang~\cite{16Huang}.

\begin{theorem}{\rm (\cite[Theorem 3.3]{16Huang})}
\label{thm: bfP isom to calP}
\begin{enumerate}[label = {\rm (\alph*)}]
\item
Let $\alpha$ be a composition of $n$. 
Then $\bfP_\alpha$ is isomorphic to $\calP_{\set(\alpha)}$ as an $H_n(0)$-module.

\item
Let $\bal = \alpha^{(1)} \oplus \cdots \oplus \alpha^{(k)}$ be a generalized composition of $n$.
Then $\bfP_\bal$ is isomorphic to $\bigoplus_{\beta \in [\bal]} \bfP_\beta$ as an $H_n(0)$-module.
\end{enumerate}
\end{theorem}

Let us define another $H_n(0)$-action on the $\C$-span of $\SRT(\bal)$ as follows:
for $i \in [n-1]$ and $T \in \SRT(\bal)$,
\begin{align}\label{eq: def of bbfP action}
\pi_i \star T: = \begin{cases}
T, & \text{if $i$ is strictly above $i+1$ in $T$},\\
0, & \text{if $i$ and $i+1$ are in the same row of $T$},\\
s_i \cdot T, & \text{if $i$ is strictly below $i+1$ in $T$.}
\end{cases}
\end{align}
We denote the resulting $H_n(0)$-module by $\bbfP_\bal$.
One can easily see that for all $\alpha \models n$, 
$\bfP_\alpha$ is isomorphic to $\bbfP_{\alpha^\rmc}$ as an $H_n(0)$-module (for instance, see~\cite{20CKNO,16Huang,79Norton}).
The reason why we prefer $\bbfP_\alpha$ to $\bfP_\alpha$ is that the actions of $H_n(0)$-modules in our concern are described in terms of $\pi_i$'s. 

Let $T_0 \in \SRT(\bal)$ be the SRT obtained by filling $\trd(\bal)$
with entries $1, 2, \ldots, n$ from top to bottom and from left to right.
Then $\bfP_{\bal}$ and $\bbfP_{\bal}$ are cyclically generated by $T_0$.

\subsection{The $H_n(0)$-action on standard immaculate tableaux}
Noncommutative Bernstein operators were introduced by Berg {\it et al.}~\cite{14BBSSZ}.
Applied to the identity element, they yield the \emph{immaculate functions}, which form a basis of the the ring of noncommutative symmetric functions.
Soon after, using the combinatorial objects called standard immaculate tableaux, they constructed indecomposable $H_n(0)$-modules whose quasisymmetric characteristics are the quasisymmetric functions which are dual to immaculate functions (see~\cite{15BBSSZ}).

\begin{definition}\label{def: SIT}
Let $\alpha \models n$. A \emph{standard immaculate tableau} (SIT) of shape $\alpha$  is a filling $\calT$ of the composition diagram $\tcd(\alpha)$ with $\{1,2,\ldots,n\}$ such that the entries are all distinct, the entries in each row increase from left to right, and the entries in the first column increase from top to bottom.
\end{definition}

We denote the set of all standard immaculate tableaux of shape $\alpha$ by $\SIT(\alpha)$. 
Define an $H_n(0)$-action on the $\C$-span of $\SIT(\alpha)$ as follows: 
for each $i = 1,\ldots, n-1$ and $\calT \in \SIT(\alpha)$,  
\begin{align}\label{eq: action on SIT}
\pi_i \cdot \calT = \begin{cases}
\calT & \text{if $i$ is weakly below $i+1$ in $\calT$},\\
0 & \text{if $i$ and $i+1$ are in the first column of $\calT$},\\
s_i \cdot \calT & \text{otherwise}.
\end{cases}
\end{align}
The resulting module is denoted by $\calV_\alpha$.

Let $\calT_0 \in \SIT(\alpha)$ be the SIT obtained by filling $\tcd(\alpha)$
with entries $1, 2, \ldots, n$ from left to right and from top to bottom.

\begin{theorem}{\rm (\cite[Theorem 3.5]{15BBSSZ})}
For $\alpha \models n$, $\calV_\alpha$ is a cyclic indecomposable $H_n(0)$-module generated by $\calT_0$
whose quasisymmetric characteristic is the dual immaculate quasisymmetric function attached to $\alpha$.
\end{theorem}

\subsection{The $H_n(0)$-action on standard extended tableaux}\label{subsec: SET}

In~\cite{19AS}, Assaf and Searles defined the extended Schur functions as the stable limits of lock polynomials and showed that they form a basis of $\Qsym$.
Soon after, using standard extended tableaux, Searles~\cite{19Searles} constructed indecomposable $H_n(0)$-modules whose quasisymmetric characteristics are the extended Schur functions.

\begin{definition}
Given $\alpha \models n$, a \emph{standard extended tableau} (SET) of shape $\alpha$ is a filling $\sfT$ of the reverse composition diagram $\trcd(\alpha)$ with $\{1,2,\ldots,n\}$ such that the entries are all distinct, the entries in each row increase from left to right, and the entries in each column decrease from top to bottom.
\end{definition}

We denote the set of all standard extended tableaux of shape $\alpha$ by $\SET(\alpha)$. 
Define an $H_n(0)$-action on the $\C$-span of $\SET(\alpha)$ as follows: 
for each $i = 1,\ldots, n-1$ and $\sfT \in \SET(\alpha)$,
\begin{align}
\pi_i \cdot \sfT = \begin{cases}
\sfT & \text{if $i$ is strictly left of $i+1$ in $\sfT$},\\
0 & \text{if $i$ and $i+1$ are in the same column of $\sfT$},\\
s_i \cdot \sfT & \text{if $i$ is strictly right of $i+1$ in $\sfT$}.
\end{cases}
\end{align}
The resulting module is denoted by $X_\alpha$.

Let $\sfT_0 \in \SET(\alpha)$ be the SET obtained by filling $\trcd(\alpha)$
with entries $1, 2, \ldots, n$ from left to right and from bottom to top.

\begin{theorem}{\rm (\cite[Theorem 3.10]{19Searles})}
For $\alpha \models n$, $X_\alpha$ is a cyclic indecomposable $H_n(0)$-module generated by $\sfT_0$
whose quasisymmetric characteristic is the extended Schur function attached to $\alpha$.
\end{theorem}

\subsection{The $H_n(0)$-action on standard permuted composition tableaux}\label{subsec: SPCT}
Standard permuted composition tableaux were introduced and studied intensively by Tewari and van Willigenburg~\cite{15TW, 19TW}. 
\begin{definition}\label{def: PCT}
Given $\alpha \models n$ and $\sigma \in \SG_{\ell(\alpha)}$, a \emph{standard permuted composition tableau} ($\SPCT$) of shape $\alpha$ and type $\sigma$ is a filling $\tau$ of $\tcd(\alpha)$ with entries in $\{1,2,\ldots,n\}$ such that the following conditions hold:
\begin{enumerate}[label = {\rm (\arabic*)}]
\item The entries are all distinct.
\item The standardization of the word obtained by reading the first column from top to bottom is $\sigma$.
\item The entries along the rows decrease weakly when read from left to right.
\item If $i<j$ and $\tau_{i,k} > \tau_{j,k+1}$, then $(i,k+1) \in \tcd(\alpha)$ and $\tau_{i,k+1} > \tau_{j,k+1}$.
\end{enumerate}
The condition (4) is called the \emph{triple condition}.
\end{definition}

We denote by $\SPCT^\sigma(\alpha)$ the set of all standard permuted composition tableau of shape $\alpha$ and type $\sigma$. 
For $\alpha \models n$ and $\sigma \in \SG_{\ell(\alpha)}$, we say $\alpha$ is compatible with $\sigma$ if $\alpha_i \ge \alpha_j$ for all $i < j$ with $\sigma(i) > \sigma(j)$.

\begin{proposition} {\rm (\cite[Proposition 14]{13HMR})}\label{Prop: comparibility}
For a composition $\alpha$ and $\sigma \in \SG_{\ell(\alpha)}$, $\SPCTsa$ is nonempty if and only if $\alpha$ is compatible with $\sigma$.
\end{proposition}

Let $\tau \in \SPCTsa$. An integer $1 \le i \le n-1$ is a \emph{descent} of $\tau$ if $i+1$ lies weakly right of $i$ in $\tau$. Denote by $\Des(\tau)$ the set of all descents of $\tau$ and set $\comp(\tau) := \comp(\Des(\tau))$.
And, for $1 \leq i < j \leq n$, we say that $i$ and $j$ are \emph{attacking} (in $\tau$)
if either
\begin{enumerate}[label = {\rm (\roman*)}]
\item $i$ and $j$ are in the same column in $\tau$, or
\item $i$ and $j$ are in adjacent columns in $\tau$, with $j$ positioned lower-right of $i$.
\end{enumerate}
In case where $i$ and $i+1$ are attacking (resp. nonattacking) and $i$ is a descent of $\tau$, we simply say that $i$ is an \emph{attacking descent} (resp. \emph{nonattacking descent}).

Define an $H_n(0)$-action on the $\C$-span of $\SPCTsa$ as follows: 
for each $i = 1,\ldots, n-1$ and $\tau \in \SPCTsa$,
\begin{align}\label{eq: Hecke action on SPCT}
\pi_i \cdot \tau = \begin{cases}
\tau & \text{if $i$ is not a descent},\\
0 & \text{if $i$ is an attacking descent},\\
s_i \cdot \tau & \text{if $i$ is a nonattacking descent.}
\end{cases}
\end{align}
The resulting module is denoted by $\bfSsa$.

\begin{theorem}{\rm (\cite[Theorem 3.1]{19TW})}
\label{thm: Hecke action on SPCT}
Given $\alpha \models n$, \eqref{eq: Hecke action on SPCT} induces an $H_n(0)$-action on $\bfSsa$.
\end{theorem}

Let $\alpha$ be a composition whose largest part is $\alphamax$ and $\tau$ an SPCT of shape $\alpha$ and type $\sigma$. 
For $1 \le i \le \alphamax$, we define the \emph{$i$th column word} $w^i(\tau)$ of $\tau$ to be the word obtained from $\tau$ by reading the entries in the $i$th column from top to bottom. 
The \emph{standardized $i$th column word} of $\tau$, denoted by $\rmst(w^i(\tau))$, is the permutation $\rho \in \SG_{\ell(w^i(\tau))}$ uniquely determined by the condition: for $1 \le j < j' \le \ell(w^i(\tau))$,
\begin{align*}
\rho(j) > \rho(j') \quad \text{if and only if} \quad  w^i(\tau)_{j} > w^i(\tau)_{j'}.
\end{align*}
 Here $\ell(w^i(\tau))$ is the length of the word $w^i(\tau)$. The \emph{standardized column word} of $\tau$, denoted by $\rmst(\tau)$, is the word
\[
\rmst(\tau) = \rmst(w^1(\tau)) \  \rmst(w^2(\tau)) \cdots \rmst(w^{\alphamax}(\tau)).
\]

Recall that the equivalence relation $\sim_{\alpha}$ on $\SPCTsa$ defined by
\[
\tau_1 \sim_{\alpha} \tau_2 \quad \text{if and only if} \quad \rmst(\tau_1) = \rmst(\tau_2)\quad \text{for $\tau_1,\tau_2 \in \SPCTsa$}
\]
was introduced in~\cite[Section 3]{19TW}.
Let $\mathcal{E}^\sigma(\alpha)$ be the set of all equivalence classes under $\sim_{\alpha}$.
For each $E \in \calEsa$, let $\bfS^{\sigma}_{\alpha,E}$ be the $H_n(0)$-submodule of $\bfSsa$ whose underlying space is the $\C$-span of $E$. As $H_n(0)$ modules,
\begin{align*}
\bfSsa \cong \bigoplus_{E \in \calEsa} \bfS^{\sigma}_{\alpha,E}.
\end{align*}
An $\SPCT$ $\tau$ is said to be a \emph{source tableau} if, for every $i \notin \Des(\tau)$ where $i \neq n$, we have that $i+1$ lies to the immediate left of $i$.
In particular, for every composition $\alpha = (\alpha_1,\ldots, \alpha_{k})$ and $\sigma \in \SG_k$,
one can construct a source tableau in $\SPCTsa$ in the following way:
For $1 \leq i \leq k$,
fill the $\alpha_{\sigma^{-1}(i)}$th row with
\[
\sum_{j = 0}^{i-1} \alpha_{\sigma^{-1}(j)} + 1 \, ,  \sum_{j = 0}^{i-1} \alpha_{\sigma^{-1}(j)} + 2 \, , \ldots \, , \sum_{j = 0}^{i-1} \alpha_{\sigma^{-1}(j)} + \alpha_{\sigma^{-1}(i)}
\]
in the decreasing order from left to right.
Here $\alpha_{\sigma^{-1}(0)}$ is set to be 0.
The tableau defined as above is called \emph{the canonical source tableau of shape $\alpha$ and type $\sigma$}, and is denoted by $\tauC$.
We call the class containing $\tauC$, denoted by $C$, the \emph{canonical class} in $\mathcal{E}^\sigma(\alpha)$, and ${\bfS}^\sigma_{\alpha,C}$ the \emph{canonical submodule} of $\bfSsa$.

\begin{theorem}{\rm (\cite{20CKNO,19TW})}\label{lem: unique source and sink}
For $E \in \mathcal{E}^{\sigma}(\alpha)$, there is a unique source tableau $\tauE$ in $E$ and $\bfS^\sigma_{\alpha,E}$ is a cyclic indecomposable $H_n(0)$-module generated by $\tauE$.
\end{theorem}

It was shown in~\cite{15TW} that the image of the quasisymmetric characteristic of $\bfS^{\id}_\alpha$ is the \emph{quasisymmetric Schur function} $\calS_\alpha$.
Very recently, the quasi-Schur expansion of the image of the quasisymmetric characteristic of $\bfS^{\sigma}_\alpha$ was given in~\cite{20CKNO}.

\section{A series of surjections starting from $\bfP_{\alpha}$}
\label{Section3}

To begin with, we recall the notion of a projective cover.
Let $R$ be a left artin ring and $A,B$ be finitely generated $R$-modules.
An epimorphism $f:A\to B$ is called an \emph{essential epimorphism} if a morphism $g: X\to A$ is an epimorphism 
whenever $f \circ g:X\to B$ is an epimorphism. 
A \emph{projective cover} of $A$ is an essential epimorphism $f:P\to A$ with $P$ a projective $R$-module, 
which always exists and is unique up to isomorphism. 
It plays an extremely important role in understanding the structure of $A$ (see~\cite{95ARS}).

The purpose of this section is to demonstrate the relationship among $\calV_\alpha$,
$X_\alpha$, and $\bfS^\sigma_{\alpha,C}$. 
To be precise, we construct a series of surjections
\begin{equation}\label{eq: surj series 1}
\begin{tikzcd}
\bfP_{\alpha} \arrow[r, two heads] & 
\calV_\alpha \arrow[r, two heads] & 
X_\alpha \arrow[r, two heads] & 
\mphi[\bfS^\sigma_{\tal,C}],
\end{tikzcd}
\end{equation}
which implies that $\calV_\alpha$, $X_\alpha$, and $\mphi[\bfS^\sigma_{\tal,C}]$ have the same projective cover.
Here $\mphi[\bfS^\sigma_{\tal,C}]$ denotes the $\phi$-twist of $\bfS^\sigma_{\tal,C}$ and $\sigma$ ranges over the set of permutations in $\SG_{\ell(\alpha)}$ satisfying that 
$\lambda(\alpha) = \alpha^\rmr \cdot \sigma$. 
For the definition of $\phi$-twist, see Subsection~\ref{subsec: X to bfS}.

\subsection{A surjection from $\bfP_{\alpha}$ to $\calV_\alpha$}
\label{Sec3.1}

For $\calT \in \SIT(\alpha)$, we denote by $\calT_{i,j}$ the entry at the box in row $i$ from top to bottom and column $j$ from left to right in $\calT$.
For $T \in \SRT(\alpha)$, we denote by $T^{q}_{p}$ the entry at the $q$th box from the top of the $p$th column from the left in $T$.

For each $T \in \SRT(\alpha^\rmc)$, define $\calT_T$ to be the filling of $\tcd(\alpha)$ given by
\[
(\calT_T)_{i,j} = T^{j}_{i}.
\]
\newpage\noindent
In other words, $\calT_T$ is defined via the following process:
\begin{enumerate}[label = {\rm (\arabic*)}]
\item Lift each column of $T$ to the topmost row.
\item Transpose the resulting filling along the main diagonal.
\end{enumerate}

Define a $\C$-linear map $\mPhi : \bbfP_{\alpha^c} \ra \calV_\alpha$ by 
\begin{align}\label{eq: def of phi}
\mPhi(T) = 
\begin{cases}
\calT_T & \text{if $\calT_T$ is an SIT,}\\
0 & \text{otherwise.}
\end{cases}
\end{align}

\begin{example}
For $\alpha = (1,2,2)$, let
\[
T_1 = 
\begin{array}{l}
\begin{ytableau}
\none & \none & 4 \\
\none & 2 & 5 \\
1 & 3
\end{ytableau}
\end{array}
\in \SRT(\alpha^\rmc)
\quad \text{and} \quad
T_2 = 
\begin{array}{l}
\begin{ytableau}
\none & \none & 4 \\
\none & 1 & 5 \\
2 & 3
\end{ytableau}
\end{array}
\in \SRT(\alpha^\rmc).
\]
Then
\[
\calT_{T_1} =
\begin{array}{l}
\begin{ytableau}
1 & \none \\
2 & 3 \\
4 & 5
\end{ytableau}
\end{array}
\in \SIT(\alpha)
\quad \text{and} \quad
\calT_{T_2} =
\begin{array}{l}
\begin{ytableau}
2 & \none \\
1 & 3 \\
4 & 5
\end{ytableau}
\end{array}
\notin \SIT(\alpha).
\]
Thus, $\mPhi(T_1) = \calT_{T_1}$ and  $\mPhi(T_2) = 0$.
\end{example}

\begin{theorem}\label{thm: V hom ima of P}
For $\alpha \models n$, $\mPhi : \bbfP_{\alpha^c} \ra \calV_\alpha$ is a surjective $H_n(0)$-module homomorphism,
thus $\calV_\alpha$ is a homomorphic image of $\bfP_{\alpha}$.
\end{theorem}

\begin{proof}
First, let us show the surjectivity. Let $\calT$ be an SIT of shape $\alpha$.
Define the filling $T_{\calT}$ of $\trd(\alpha^c)$ by
\begin{align}\label{eq: T_calT}
(T_{\calT})^{j}_{i} = \calT_{i,j}. 
\end{align}
We claim that $T_{\calT} \in \SRT(\alpha^\rmc)$ and $\mPhi(T_{\calT}) = \calT$.
Since the entries in each row of $\calT$ increase from left to right, those in each column of $T_{\calT}$ increase from top to bottom.
Moreover, since the entries in the first column of $\calT$ increase from top to bottom, we have
\[
(T_{\calT})^{1}_{i} < (T_{\calT})^{1}_{i+1} < (T_{\calT})^{\alpha_{i+1}}_{i+1}.
\]
This implies that the entries in each row of $T_{\calT}$ increase from left to right, thus $T_{\calT} \in \SRT(\alpha^\rmc)$.
It follows from the definition of $\mPhi$ that $\mPhi(T_{\calT}) = \calT$, which verifies the claim.

Next, let us show that $\mPhi$ is an $H_n(0)$-module homomorphism. 
Let $T$ be an SRT of shape $\alpha^c$ and $1 \le i \le n-1$. We have three cases.
\smallskip

{\it Case 1: $\pi_i \star T = T$.}
There is nothing to prove in case where $\mPhi(T) = 0$.
Assume that $\mPhi(T) = \calT_T \neq 0$.
Since $i$ lies strictly above $i+1$ in $T$ and $T$ is of ribbon shape, $i$ lies weakly right of $i+1$ in $T$.
If $i$ and $i+1$ are in the same column in $T$, then they are in the same row in $\calT_T$. 
Otherwise, $i$ lies strictly below $i+1$ in $\calT_T$. 
In both cases, it is immediate from~\eqref{eq: action on SIT} that $\pi_i \cdot \calT_T = \calT_T$.
\smallskip

{\it Case 2: $\pi_i \star T = 0$.}
There is nothing to prove in case where $\mPhi(T) = 0$.
Assume that $\mPhi(T) = \calT_T \neq 0$.
Since $i$ and $i+1$ are in the same row of $T$, we have 
\[
(\calT_T)_{j,1} = i \quad \text{and} \quad (\calT_T)_{j+1,\alpha_{j+1}} = i+1 \quad \text{for some $1\le j \le \ell(\alpha)-1$.}
\]
This implies that $\alpha_{j+1} = 1$, and thus $\pi_i \cdot \calT_T = 0$.
\smallskip

{\it Case 3: $\pi_i \star T = s_i \cdot T$.}
We first deal with the case where $\mPhi(T) = 0$. 
Since $T$ is increasing down each column, $\calT_T$ is increasing across each row from left to right.
The assumption $\mPhi(T) = 0$ implies that the entries in the first column of $\calT_T$ fail to increase from top to bottom.
Pick up $1\le j \le n-1$ such that $(\calT_T)_{j,1} > (\calT_T)_{j+1,1}$.
If either $(\calT_T)_{j+1,1} \neq i$ or $(\calT_T)_{j,1} \neq i+1$, then $(\calT_{s_i \cdot T})_{j,1} > (\calT_{s_i \cdot T})_{j+1,1}$. 
This implies that $\calT_{s_i \cdot T} \notin \SIT(\alpha)$, thus $\mPhi(s_i \cdot T) = 0$. Otherwise, $T^1_{j+1} = i$ and $T^1_{j} = i+1$. This forces $i$ to be strictly above $i+1$ in $T$. 
It, however, contradicts the assumption that $\pi_i \star T = s_i \cdot T$.

Next we deal with the case where $\mPhi(T) = \calT_T \neq 0$. 
We claim that $\mPhi (s_i \cdot T) = \pi_i \cdot \calT_T$.
We first consider the case where $\mPhi (s_i \cdot T) = 0$. 
Then the entries in the first column of $\calT_{s_i \cdot T}$ fail to increase from top to bottom. 
On the other hand, the entries in the first column of $\calT_T$ increase from top to bottom. 
Thus there exists an index $1 \le j \le \ell(\alpha) - 1$ such that $(\calT_{s_i \cdot T})_{j+1,1} = i$ and $(\calT_{s_i \cdot T})_{j,1} = i+1$. 
This implies that $(\calT_T)_{j,1} = i$ and $(\calT_T)_{j+1,1} = i+1$, thus $\pi_i \cdot \calT_T = 0$. 
We next consider the case where $\mPhi (s_i \cdot T) = \calT_{s_i \cdot T} \neq 0$.
Since $\pi_i \star T = s_i \cdot T$, $i$ is strictly below $i+1$ in $T$. 
Since $T$ is an SRT, $i$ lies strictly left of $i+1$ in $T$. 
This forces $i$ to be strictly above $i+1$ in $\calT_T$. 
If $i$ and $i+1$ appear in the first column of $\calT_T$, then the entries of the first column of $\calT_{s_i \cdot T}$ fail to increase. 
It contradicts the assumption that $\mPhi (s_i \cdot T) \neq 0$. 
As a consequence, $\pi_i \cdot \calT_T = s_i \cdot \calT_T = \calT_{s_i \cdot T}$.
\smallskip

The second assertion follows from the fact that $\bfP_\alpha \cong \bbfP_{\alpha^\rmc}$.
\end{proof}

Note that $\Phi(T) \neq 0$ if and only if the topmost entries of the columns of $T$ are increasing from left to right.
Therefore the kernel of $\Phi$ is given by the $\C$-span of
\[
\{T \in \SRT(\alpha^\rmc) \mid T^1_p > T^1_{p+1} \quad \text{for some $1 \leq p < \ell(\alpha)$}\}.
\]
We describe a sufficient and necessary condition for $\overline{\bfP}_{\alpha^\rmc}$ to be isomorphic to $\calV_\alpha$ as an $H_n(0)$-module.

\begin{corollary}\label{Cor: P isom to X}
Let $\alpha \models n$.
Then $\overline{\bfP}_{\alpha^\rmc}$ is isomorphic to $\calV_\alpha$ as an $H_n(0)$-module if and only if $\alpha$ is of hook shape, that is, $\alpha = (n-k, 1^k)$ for some $k \geq 0$.
\end{corollary}
\begin{proof}
Let $T$ be an SRT of shape $\alpha^\rmc$.
Note that $\ker(\Phi) = \{0\}$ since the topmost entries of the columns of $T$ are increasing from left to right.
Thus the ``if'' part is verified.

To show the ``only if'' part, we assume that $\alpha$ is not of hook shape.
Then there is an index $i\geq 2$ such that $\alpha_i \geq 2$.
Let $T \in \SRT(\alpha^\rmc)$ be the SRT obtained by filling $\trd(\alpha^\rmc)$
with entries $1, 2, \ldots, n$ from left to right and from top to bottom.
Since $T^1_{i-1} > T^1_{i}$, $T \in \ker(\Phi)$, which implies that
$\dim(\bbfP_{\alpha^\rmc}) > \dim(\calV_\alpha)$.
This says that $\overline{\bfP}_{\alpha^\rmc}$ is not isomorphic to $\calV_\alpha$, as required.
\end{proof}

\subsection{A surjection from $\calV_\alpha$ to $X_\alpha$}

Let $\alpha$ be a composition. 
For $\sfT \in \SET(\alpha)$, we denote by $\sfT_{i,j}$ the entry at the box in row $i$ and column $j$ in $\sfT$.
For $\calT \in \SIT(\alpha)$,
define $\sfT_{\calT}$ to be the filling of $\trcd(\alpha)$ given by
\begin{align*}
(\sfT_{\calT})_{i,j} = \calT_{\ell(\alpha) +1 - i, j},
\end{align*}
equivalently, by flipping $\calT$ horizontally.
Define a $\C$-linear map $\mGam : \calV_\alpha \ra X_\alpha$ by 
\begin{align}\label{eq: def of psi}
\mGam(\calT) = 
\begin{cases}
\sfT_{\calT} & \text{if $\sfT_{\calT}$ is an SET,}\\
0 & \text{otherwise.}
\end{cases}
\end{align}

\begin{example}
For $\alpha = (1,2,2)$, let
\[
\calT_1 = 
\begin{array}{l}
\begin{ytableau}
1 & \none \\
2 & 3 \\
4 & 5
\end{ytableau}
\end{array}
\in \SIT(\alpha)
\quad \text{and} \quad
\calT_2 = 
\begin{array}{l}
\begin{ytableau}
1 & \none \\
2 & 5 \\
3 & 4
\end{ytableau}
\end{array}
\in \SIT(\alpha).
\]
Then
\[
\sfT_{\calT_1} =
\begin{array}{l}
\begin{ytableau}
4 & 5 \\
2 & 3 \\
1 & \none
\end{ytableau}
\end{array}
\in \SET(\alpha)
\quad \text{but} \quad
\sfT_{\calT_2} =
\begin{array}{l}
\begin{ytableau}
3 & 4 \\
2 & 5 \\
1
\end{ytableau}
\end{array}
\notin \SET(\alpha).
\]
Thus, $\mGam(\calT_1) = \sfT_{\calT_1}$ and  $\mGam(\calT_2) = 0$.
\end{example}

\begin{theorem}\label{thm: X hom ima of V}
Given $\alpha \models n$, the map $\mGam : \calV_\alpha \ra X_\alpha$ is a surjective $H_n(0)$-module homomorphism, thus $X_\alpha$ is a homomorphic image of $\bfP_\alpha$.
\end{theorem}

\begin{proof}
The verification can be done with a slight modification of the proof of Theorem~\ref{thm: V hom ima of P}.
\end{proof}

Note that
$\Gamma \circ \Phi(T) \neq 0$  if and only if 
there exist $1 \leq p < \ell(\alpha)$ and $1 \leq q \leq \alpha_p$ such that $T^q_p > T^q_{p+1}$.
Therefore the kernel of $\Gamma \circ \Phi$ is given by the $\C$-span of
\[
\{T \in \SRT(\alpha^\rmc) \mid T^q_p > T^q_{p+1} \quad \text{for some $1 \leq p < \ell(\alpha)$ and $1 \leq q \leq \alpha_p$}\}.
\]
\begin{remark}
Theorem~\ref{thm: X hom ima of V} can also be derived from the construction of $\calV_\alpha$ and $X_\alpha$.
In~\cite{19Searles}, Searles introduced \emph{standard row-increasing tableaux} of shape $\alpha$ and an $H_n(0)$-module $V_\alpha$ whose underlying vector space is the $\C$-span of the set $\SRIT(\alpha)$ of standard row-increasing tableaux of shape $\alpha$.
Letting $\mathrm{NSET}(\alpha) := \SRIT(\alpha) \setminus \SET(\alpha)$ and $Y_\alpha$ be the $\C$-span of $\mathrm{NSET}(\alpha)$, 
he showed that $Y_\alpha$ is closed under the $H_n(0)$-action, thus it is an $H_n(0)$-submodule of $V_\alpha$.
In this setting, $X_\alpha$ is defined by $V_\alpha / Y_\alpha$.

It should be remarked that the $H_n(0)$-module $\calM_\alpha$ introduced by Berg {\it et al.}~\cite[Subsection 3.3]{15BBSSZ} can be naturally identified with $V_\alpha$.
Under this identification, let $\re(\SIT(\alpha)) := \{\sfT_\calT \mid \calT \in \SIT(\alpha)\}$, $\mathrm{NSIT}(\alpha) := \SRIT(\alpha) \setminus \re(\SIT(\alpha))$, and $\calN_\alpha$ be the $\C$-span of $\mathrm{NSIT}(\alpha)$.
They showed that $\calN_\alpha$ is closed under the $H_n(0)$-action, thus it is an $H_n(0)$-submodule of $V_\alpha$.
In this setting, $\calV_\alpha$ is defined by $V_\alpha / \calN_\alpha$.
Since $\mathrm{NSIT}(\alpha) \subset \mathrm{NSET}(\alpha)$, there exists a natural $H_n(0)$-module projection from $V_\alpha / \calN_\alpha$ onto $V_\alpha / Y_\alpha$, which is essentially same to $\mGam: \calV_\alpha \ra X_\alpha$.
\end{remark}

\subsection{A surjection from $X_\alpha$ to the $\phi$-twist of $\bfS^\sigma_{\tal,C}$}
\label{subsec: X to bfS}

Let us recall the automorphism $\mphi: H_n(0) \ra H_n(0)$ defined by $\phi(\pi_{i}) = \pi_{w_0 s_i w_0} = \pi_{n-i}$ for all $1 \le i \le n-1$ (see~\cite{05Fayers}). 
Given an $H_n(0)$-module $M$, it induces another $H_n(0)$-action $\pact$ on the vector space $M$ given by
\begin{align}\label{eq: phi twist}
\pi_i \pact v := \mphi(\pi_i) \cdot v \quad \text{for $1 \le i \le n-1$}.
\end{align}
We denote the resulting $H_n(0)$-module by $\mphi[M]$ and call it \emph{the $\phi$-twist of $M$}.

Let $\alpha$ be a composition and $\sigma \in \SG_{\ell(\alpha)}$ satisfying that $\lambda(\alpha) = \alpha^\rmr \cdot \sigma$.
Indeed the set of such $\sigma$'s appears as a Bruhat interval (see Remark~\ref{rem:Bruhat}).
For $\sfT \in \SET(\alpha)$,
define the filling $\tau_{\sfT}$ of $\tcd(\ldalpha)$ by
\begin{align*}
(\tau_{\sfT})_{i,j} = n+1 - \sfT_{\sigma(i),j}.
\end{align*}
Roughly speaking,
$\tau_{\sfT}$ is the filling obtained from $\sfT$ by rearranging the rows of $\sfT$ according to $\sigma$ 
and by taking the complement of each entry in $\sfT$. 

\begin{remark}\label{rem:Bruhat}
Assume that $\lambda(\alpha) = \alpha^\rmr \cdot \sigma$.
Letting $\alpha^\rmr = ((\alpha^\rmr)_1, (\alpha^\rmr)_2, \ldots, (\alpha^\rmr)_{\ell(\alpha)})$,
we define $\sigma_{\rm max}$ to be the permutation in $\SG_{\ell(\alpha)}$ such that  $\lambda(\alpha) = \alpha^\rmr \cdot \sigma_{\rm max}$ and
\[
\sigma_{\rm max}^{-1}(i) > \sigma_{\rm max}^{-1}(j) \text{ whenever $(\alpha^\rmr)_i = (\alpha^\rmr)_j$ and $i <j$.}
\]
We also define $\sigma_{\rm min}$ to be the permutation in $\SG_{\ell(\alpha)}$ such that $\lambda(\alpha) = \alpha^\rmr \cdot \sigma_{\rm min}$ and
\[
\sigma_{\rm min}^{-1}(i) < \sigma_{\rm min}^{-1}(j) \text{ whenever $(\alpha^\rmr)_i = (\alpha^\rmr)_j$ and $i <j$.}
\]
Roughly, $\sigma_{\rm max}$ (resp. $\sigma_{\rm min})$ is the permutation sorting the entries of $\alpha^{\rmr}$ in the descending order in such a way that the order of same entries in $\alpha^\rmr$ is reversed in (resp. as in) $\lambda(\alpha)$.
Write $\lambda(\alpha) = 1^{m_1}2^{m_2} \cdots k^{m_k}$ and set $m_{k+1} := 0$.
Let $\rho$ be the longest permutation in the subgroup
$\prod_{1 \le i \le k} \SG_{[ \sum_{j = i+1}^{k+1} m_j+1, \sum_{j = i}^{k+1} m_j ]}$ of $\SG_{\ell(\alpha)}$.
Then $\sigma_{\rm max} = \sigma_{\rm min} \cdot \rho$. 
Hence we have
\[
\{ \sigma \in \SG_{\ell(\alpha)} \mid \lambda(\alpha) = \alpha^\rmr \cdot \sigma \} = 
[\sigma_{\rm min}, \sigma_{\rm max}] \, ,
\]
and the number of such $\sigma$'s
equals to $\prod_{i=1}^k (m_i!)$.
\end{remark}

\begin{example}\label{eg: tau_sft}
Let $\alpha = (1,2,1,1,2)$. Then
$\sigma_{\rm max} = 41 \, 532$,  $\sigma_{\rm min} = 14 \, 235$, and
$\sigma_{\rm max} = \sigma_{\rm min} \cdot (s_1)  (s_3s_4s_3)$,
which are illustrated as follows:
\[
\begin{tikzpicture}
\node at (-0.5,4) {$\alpha^\rmr$};
\node at (0.2,4) {$=$};
\node at (1,4) {$2$};
\node at (2,4) {$1$};
\node at (3,4) {$1$};
\node at (4,4) {$2$};
\node at (5,4) {$1$};

\node at (-1,2) {$\alpha^\rmr \cdot \sigma_{\min}$};
\node at (0.2,2) {$=$};
\node at (1,2) {$2$};
\node at (2,2) {$2$};
\node at (3,2) {$1$};
\node at (4,2) {$1$};
\node at (5,2) {$1$};

\node at (-1,0) {$\alpha^\rmr \cdot \sigma_{\max}$};
\node at (0.2,0) {$=$};
\node at (1,0) {$2$};
\node at (2,0) {$2$};
\node at (3,0) {$1$};
\node at (4,0) {$1$};
\node at (5,0) {$1$};

\draw[->] (1,3.5) -- (1,2.5);
\draw[->] (2,3.5) -- (3,2.5);
\draw[->] (3,3.5) -- (4,2.5);
\draw[->] (4,3.5) -- (2,2.5);
\draw[->] (5,3.5) -- (5,2.5);

\draw[->] (1,1.5) -- (2,0.5);
\draw[->] (2,1.5) -- (1,0.5);
\draw[->] (3,1.5) -- (5,0.5);
\draw[->] (4,1.5) -- (4,0.5);
\draw[->] (5,1.5) -- (3,0.5);

\end{tikzpicture}
\]
Let
\[
\sfT_1 =
\begin{array}{l}
\begin{ytableau}
6 & 7\\
5 \\
4  \\
2 & 3 \\
1 & \none
\end{ytableau}
\end{array}
\quad \text{and} \quad
\sfT_2 =
\begin{array}{l}
\begin{ytableau}
5 & 7\\
4 \\
3 \\
2 & 6 \\
1
\end{ytableau}
\end{array}
\in \SET(\alpha).
\]
In case of $\sigma = \sigma_{\rm min}$, we have
\[
\tau_{\sfT_1} =
\begin{array}{l}
\begin{ytableau}
2 & 1 \\
6 & 5 \\
3 \\
4  \\
7  
\end{ytableau}
\end{array}
\quad \text{and} \quad
\tau_{\sfT_2} =
\begin{array}{l}
\begin{ytableau}
*(black!10) 3 & *(black!10) 1 \\
6 & *(black!10) 2 \\
4 \\
5  \\
7  
\end{ytableau}
\end{array}.
\]
Note that $\tau_{\sfT_1}$ is an $\SPCT$, but $\tau_{\sfT_2}$ is not (see the shaded boxes).
\end{example}

Consider a $\mathbb{C}$-linear map
$\Ups_{\alpha, \sigma} : X_\alpha \ra \mphi[\bfS^{{\sigma}}_{\lambda(\alpha),C}]$ defined by
\begin{align*}
 \quad \sfT \mapsto
\begin{cases}
\tau_\sfT & \text{if it is an SPCT,}\\
0 & \text{otherwise}
\end{cases}
\end{align*}
for $\sfT \in \SET(\alpha)$.
We claim that  if $\tau_T$ is an SPCT, then $\tau_T$ is automatically contained in $\bfS^{{\sigma  }}_{\lambda(\alpha),C}$.
Let $\sfT, \sfT' \in \SET(\alpha)$.
Since the entries in the $i$th column of $\sfT$ and $\sfT'$ are decreasing from top to bottom for each $i \ge 1$, the standardization of the $i$th column of $\tau_{\sfT}$ is equal to that of $\tau_{\sfT'}$.
Moreover, under this construction, 
it is easy to see that $\sfT_0 \in \SET(\alpha)$, which is a generator of $X_\alpha$,
is mapped to the canonical source tableau $\tauC$ in $\SPCT^{\sigma}(\ldalpha)$
(for the definition of $\sfT_0$ and $\tauC$, see Subsection~\ref{subsec: SET} and Subsection~\ref{subsec: SPCT}).
Therefore $\Ups_{\alpha, \sigma}$ is a well-defined $\mathbb{C}$-linear map.
In Example~\ref{eg: tau_sft}, $\Ups_{(1,2,2), \id}(\sfT_1) = \tau_{\sfT_1}$ but $\Ups_{(1,2,2), \id}(\sfT_2) = 0$
since $\tau_{\sfT_2}$ does not satisfy the triple condition (see the shaded boxes).

From now on,
if $\alpha$ and $\sigma$ are clear in the context, we will drop the subscript from $\Ups_{\alpha, \sigma}$.

\begin{theorem}\label{thm: X isom to bfS}
Let $\alpha \models n$ and $\sigma \in \SG_{\ell(\alpha)}$ satisfying that $\lambda(\alpha) = \alpha^\rmr \cdot \sigma$.
The map $\Ups : X_\alpha \rightarrow \mphi[\bfS^{{\sigma}}_{{\ldalpha},C}]$ is a surjective $H_n(0)$-module homomorphism.
\end{theorem}
\begin{proof}
We first show that $\Ups$ is surjective.
For $\tau \in C$,
let $\sfT_\tau$ be the filling of $\trcd(\alpha)$ defined by
\[
(\sfT_\tau)_{i,j} = n+1-\tau_{\sigma^{-1}(i),j} \, .
\]
Since $\tau$ is decreasing across each row from left to right, $\sfT_\tau$ is increasing across each row from left to right.
Since $\tau \in C$, $\tau_{\sigma^{-1}(i), j} < \tau_{\sigma^{-1}(k),j}$ for $i < k$ and $(i, j), (k,j) \in \trcd(\alpha)$.
This shows that $\sfT_\tau$ is decreasing down each column.
Consequently, $\sfT_\tau \in \SET(\alpha)$.
It is straightforward from the construction of $\sfT_\tau$ that $\Ups(\sfT_\tau) = \tau$.

Next, we show that $\Ups$ is an $H_n(0)$-module homomorphism, that is,
$$
\Ups(\pi_i \cdot \sfT) = \pi_i  \pact \Ups(\sfT)  = \pi_{n-i} \cdot \Ups(\sfT)  \quad \text{for $1 \le i \le n-1$ and $\sfT \in \SET(\alpha)$.}
$$

{\it Case 1: $\pi_i \cdot \sfT = \sfT$.}
From the assumption, $i$ lies strictly left of $i+1$ in $\sfT$.
Hence, by the construction of $\Ups$,
$n+1-i$ lies strictly left of $n-i$ in $\tau_\sfT$.
If $\Ups(\sfT) = 0$, then $\Ups(\pi_i \cdot \sfT) = \pi_i  \pact \Ups(\sfT) =0$.
Otherwise, $n-i \notin \Des(\tau_T)$, therefore $\pi_{n-i} \cdot \Ups(\sfT) = \Ups(\sfT) $ as required.

{\it Case 2: $\pi_i \cdot \sfT = 0$.}
If $\Ups(\sfT) = 0$, then there is nothing to prove.
Otherwise,
$n-i$ is an attacking descent in $\tau_\sfT$ since $i$ and $i+1$ are in the column of $\sfT$.
Thus $\pi_{n-i}\cdot \Ups(\sfT) = 0$.

{\it Case 3: $\pi_i \cdot \sfT = s_i \cdot \sfT$.}
From the assumption, $i$ lies strictly right of $i+1$ in $\sfT$, thus $n-i$ lies strictly left of $n-i+1$ in $\tau_\sfT$.
Moreover, by construction of $\Ups$ we have that $ \tau_\sfT = s_{n-i}\cdot \tau_{s_{i}\cdot \sfT}$.
Suppose that $\Ups(\sfT) = 0$ but $\Ups(\pi_i \cdot \sfT) \neq 0$.
If $n-i$ and $n-i+1$ are attacking in $\tau_\sfT$, then the triple condition breaks at $n-i$, $n-i+1$, and the entry immediate right of $n-i+1$ in $\tau_{s_{i}\cdot \sfT}$. 
Therefore $n-i$ and $n-i+1$ are nonattacking in $\tau_\sfT$.
This implies that $\tau_\sfT \in C$, which is a contradiction to $\Ups(\sfT) = 0$.
Hence we can conclude that $\Ups(\pi_i \cdot \sfT) = 0$ whenever $\Ups(\sfT) = 0$.
Next we deal with the case where $\Ups(\sfT) \neq 0$.
Since $n-i \in \Des(\tau_\sfT)$, we have two cases:

\begin{enumerate}[label = {\rm (\roman*)}]
\item
If $n-i$ is an attacking descent in $\tau_\sfT$,
then $\pi_{n-i} \cdot \Ups(\sfT) = 0$.
Let $(r_{n-i},c_{n-i})$ and $(r_{n-i+1},c_{n-i+1})$ be the box of $n-i$ and $n-i+1$ in $\tau_\sfT$, respectively.
Then $r_{n-i} < r_{n-i+1}$ and $c_{n-i} + 1 = c_{n-i+1}$.
Since the shape of $\tau_\sfT$ is the partition $\lambda(\alpha)$,
the box $(r_{n-i}, c_{n-i+1})$ should be filled with an entry $< n-i$ in $\tau_{\sfT}$, so in $\tau_{s_i \cdot \sfT}$.
The triple condition breaks in $\tau_{s_i \cdot \sfT}$ since the entry at $(r_{n-i+1}, c_{n-i+1})$ in $\tau_{s_i \cdot \sfT}$ is $n-i$.
This shows that $\Ups(\pi_i \cdot \sfT) = 0$, and therefore $\Ups(\pi_i \cdot \sfT) = \pi_{n-i} \cdot \Ups(\sfT) = 0$.

\item
If $n-i$ is a nonattacking descent in $\tau_\sfT$, then
$$
\pi_{n-i} \cdot \Ups(\sfT) = \pi_{n-i} \cdot \tau_\sfT = s_{n-i} \cdot \tau_\sfT = \tau_{s_{n-i} \cdot \sfT} = \Ups(s_i \cdot \sfT) = \Ups(\pi_i \cdot \sfT)$$
as desired.
\end{enumerate}
\end{proof}

\begin{example}
Let $\alpha = (2,1,2)$ and $\beta  = (1,2,2)$.
Letting $\sigma = 312$ and $\sigma' = 123$, then $\lambda(\alpha) = \alpha^\rmr \cdot \sigma$ and $\lambda(\beta) = \beta^\rmr \cdot \sigma'$.
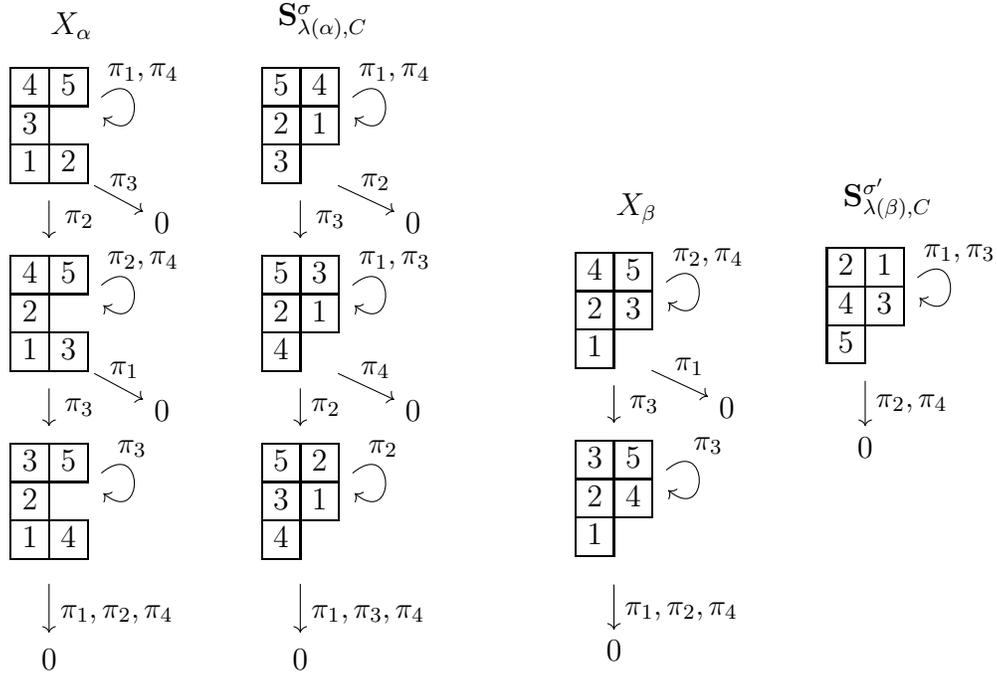
\[
\def \hp{0.15}
\def \vp{0.25}
\begin{tikzpicture}[baseline = 0mm]
\node[above] at (0.3,4) {$X_{\alpha}$};
\node at (0,3) {
\begin{ytableau}
4 & 5 \\
3  \\
1 & 2
\end{ytableau}} ;
\node at (0.7-\hp,3.25) {} edge [out=40,in=320, loop] ();
\node at (1.4-\hp,3.7) {\small $\pi_1,\pi_4$};

\draw [->] (0,1.75+\vp) -- (0,1.25+\vp);
\node at (0.4,1.5+\vp) {\small $\pi_2$};

\node at (1.5,1.7) {$0$};
\draw [->] (0.6,2.2) -- (1.25,1.85);
\node at (1,2.25) {\small $\pi_3$};

\node at (0,0+2*\vp) {
\begin{ytableau}
4 & 5 \\
2    \\
1 & 3
\end{ytableau}};
\node at (0.7-\hp,0.25+2*\vp) {} edge [out=40,in=320, loop] ();
\node at (1.4-\hp,0.7+2*\vp) {\small $\pi_2,\pi_4$};

\node at (1.5,-1.3+2*\vp) {$0$};
\draw [->] (0.6,-0.8+2*\vp) -- (1.25,-1.15+2*\vp);
\node at (1,-0.75+2*\vp) {\small $\pi_1$};

\draw [->] (0,-1.25+3*\vp) -- (0,-1.75+3*\vp);
\node at (0.4,-1.5+3*\vp) {\small $\pi_3$};

\node at (0,-3+4*\vp) {
\begin{ytableau}
3 & 5 \\
2  \\
1 & 4
\end{ytableau}};
\node at (0.7-\hp,-2.75+4*\vp) {} edge [out=40,in=320, loop] ();
\node at (1.4-2*\hp,-2.3+4*\vp) {\small $\pi_3$};

\draw [->] (0,-4.1+4*\vp) -- (0,-4.75+4*\vp);
\node at (0.9,-4.5+4*\vp) {\small $\pi_1,\pi_2,\pi_4$};
\node at (0,-5.1+4*\vp) {$0$};
\end{tikzpicture}
\qquad
\begin{tikzpicture}[baseline = 0mm]
\node[above] at (0.3,4) {$\bfS^{\sigma}_{\ldalpha,C}$};
\node at (0,3) {
\begin{ytableau}
5 & 4 \\
2 & 1 \\
3
\end{ytableau}} ;
\node at (0.7-\hp,3.25) {} edge [out=40,in=320, loop] ();
\node at (1.4-\hp,3.7) {\small $\pi_1,\pi_4$};

\draw [->] (0,1.75+\vp) -- (0,1.25+\vp);
\node at (0.4,1.5+\vp) {\small $\pi_3$};

\node at (1.5,1.7) {$0$};
\draw [->] (0.5,2.2) -- (1.25,1.85);
\node at (1,2.25) {\small $\pi_2$};

\node at (0,0+2*\vp) {
\begin{ytableau}
5 & 3 \\
2 & 1 \\
4
\end{ytableau}};
\node at (0.7-\hp,0.25+2*\vp) {} edge [out=40,in=320, loop] ();
\node at (1.4-\hp,0.7+2*\vp) {\small $\pi_1,\pi_3$};

\node at (1.5,-1.3+2*\vp) {$0$};
\draw [->] (0.5,-0.8+2*\vp) -- (1.25,-1.15+2*\vp);
\node at (1,-0.75+2*\vp) {\small $\pi_4$};

\draw [->] (0,-1.25+3*\vp) -- (0,-1.75+3*\vp);
\node[right] at (0.,-1.5+3*\vp) {\small $\pi_2$};

\node at (0,-3+4*\vp) {
\begin{ytableau}
5 & 2 \\
3  & 1 \\
4
\end{ytableau}};
\node at (0.7-\hp,-2.75+4*\vp) {} edge [out=40,in=320, loop] ();
\node at (1.4-2*\hp,-2.3+4*\vp) {\small $\pi_2$};

\draw [->] (0,-4.1+4*\vp) -- (0,-4.75+4*\vp);
\node at (0.9,-4.5+4*\vp) {\small $\pi_1,\pi_3,\pi_4$};
\node at (0,-5.1+4*\vp) {$0$};
\end{tikzpicture}
\qquad \qquad 
\begin{tikzpicture}[baseline = 70]
\node[above] at (0.3,4) {$X_{\beta}$};
\node at (0,3) {
\begin{ytableau}
4 & 5 \\
2 & 3  \\
1
\end{ytableau}} ;
\node at (0.7-\hp,3.25) {} edge [out=40,in=320, loop] ();
\node at (1.4-\hp,3.7) {\small $\pi_2,\pi_4$};

\draw [->] (0,1.75+\vp) -- (0,1.25+\vp);
\node at (0.4,1.5+\vp) {\small $\pi_3$};

\node at (1.5,1.7) {$0$};
\draw [->] (0.5,2.2) -- (1.25,1.85);
\node at (1,2.25) {\small $\pi_1$};

\node at (0,0+2*\vp) {
\begin{ytableau}
3 & 5 \\
2 & 4   \\
1
\end{ytableau}};
\node at (0.7-\hp,0.25+2*\vp) {} edge [out=40,in=320, loop] ();
\node at (1.4-\hp,0.7+2*\vp) {\small $\pi_3$};

\draw [->] (0,-1.4+3*\vp) -- (0,-2+3*\vp) node[below] {$0$};
\node[right] at (0.0,-1.) {\small $\pi_1,\pi_2,\pi_4$};
\end{tikzpicture}
\qquad
\begin{tikzpicture}[baseline = 24mm]
\node[above] at (0.3,4) {$\bfS^{\sigma'}_{\lambda(\beta),C}$};
\node at (0,3) {
\begin{ytableau}
2 & 1 \\
4 & 3 \\
5
\end{ytableau}} ;
\node at (0.7-\hp,3.25) {} edge [out=40,in=320, loop] ();
\node at (1.4-\hp,3.7) {\small $\pi_1,\pi_3$};

\draw [->] (0,1.75+\vp) -- (0,1.15+\vp) node[below]  {$0$};
;
\node[right] at (0.,1.7) {\small $\pi_2,\pi_4$};
\end{tikzpicture}
\]
The above figure shows that $\Ups: X_{(2,1,2)} \ra \mphi[\bfS^{{312}}_{{(2,2,1)},C}]$ is an $H_5(0)$-module isomorphism, and 
the map $\Ups: X_{(1,2,2)} \ra \mphi[\bfS^{{123}}_{{(2,2,1)},C}]$ is a surjective $H_5(0)$-module homomorphism with a nonzero kernel.
\end{example}

\begin{remark}\label{rem: ker of Ups}
We can give an explicit description of the kernel of $\Ups : X_\alpha \rightarrow \mphi[\bfS^{{\sigma}}_{{\ldalpha},C}]$, which is the $\C$-span of $\{\sfT \in \SET(\alpha) \mid \Ups(\sfT) = 0\}$.
Note that $\Ups(\sfT) = 0$ if and only if $\tau_\sfT$ does not satisfy the triple condition.

Define $\Xi_\alpha$ to be the subset of $\SET(\alpha)$ consisting of $\sfT$'s with a triple $(i,j,k)$ such that
\begin{enumerate}[label = {\bf K\arabic*.}]
  \item $i < j$ and $\sigma^{-1}(i) < \sigma^{-1}(j)$, and
  \item $\sfT_{i, k} < \sfT_{j, k+1}$ and $\sfT_{i,k+1} > \sfT_{j,k+1}$.
\end{enumerate}
For any $\sfT \in \Xi_\alpha$,
$\tau_\sfT$ is not an SPCT since
the boxes $(\sigma^{-1}(i),k+1)$, $(\sigma^{-1}(j),k)$ and $(\sigma^{-1}(j),k+1)$ in $\tau_\sfT$ does not satisfy the triple condition.
So $\Xi_\alpha \subseteq \ker(\Ups)$.
On the other hand,
if $\sfT \in \ker(\Ups)$, then there exist three boxes $(i, k), (i, k+1)$ and $(j, k+1)$ in $\tau_\sfT$ such that 
\[
i < j, \quad (\tau_\sfT)_{i,k} > (\tau_\sfT)_{j,k+1}, \quad \text{and}  \quad (\tau_\sfT)_{i,k+1} < (\tau_\sfT)_{j,k+1} \, .
\]
Note that $(i,k+1)$ is in $\tcd(\ldalpha)$ since $\tau_\sfT$ is of partition shape.
Moreover, we can see that $\sigma(i) < \sigma(j)$. 
Otherwise, from the construction of $\Ups$ it follows that $\sfT_{\sigma(i),k+1} > \sfT_{\sigma(j),k+1}$, which contradicts the column condition for $\SET$.
Consequently the triple $(\sigma(i), \sigma(j), k)$ satisfies both \textbf{K1} and \textbf{K2}.
So $\sfT \in \Xi_\alpha$.
\end{remark}

We close this subsection with a sufficient and necessary condition for $\Ups : X_\alpha \rightarrow \mphi[\bfS^{{\sigma}}_{{\ldalpha},C}]$ to be an isomorphism. 
To begin with, we recall the notion of $\sigma$-simple compositions, which was introduced in~\cite{20CKNO} to characterize indecomposable $\bfS^\sigma_\alpha$'s.

\begin{definition}\label{def: simple}
  Let $\alpha = (\alpha_1, \alpha_2, \ldots, \alpha_\ell)$ be compatible with $\sigma$ for $\sigma \in \SG_{\ell}$.
\begin{enumerate}
  \item[(a)] For $i < j$,
  if $\sigma(i) < \sigma(j)$ and $\alpha_i \geq \alpha_j \geq 2$,
  then $(i, j)$ is called {\it a permutation-ascending composition-descending $($PACD$)$ pair attached to the pair $(\alpha \, ; \sigma)$}.
  \item[(b)] We say that $\alpha$ is \emph{$\sigma$-simple} if every PACD pair $(i,j)$ attached to $(\alpha \, ; \sigma)$ satisfies one of the following conditions:
      \begin{enumerate}
        \item[(i)] There exists $i < k < j$ such that
        $\sigma(i) < \sigma(k) < \sigma(j)$ and $\alpha_k = \alpha_j -1$.
        \item[(ii)] There exists $k > j$ such that
         $\sigma(i) < \sigma(k) < \sigma(j)$ and $\alpha_k = \alpha_j$.
      \end{enumerate}
\end{enumerate}
\end{definition}

\begin{corollary}\label{cor: X and S isomorphic}
Let $\alpha \models n$ and $\sigma \in \SG_{\ell(\alpha)}$ satisfying that $\lambda(\alpha) = \alpha^\rmr \cdot \sigma$.
Then $X_\alpha$ is isomorphic to $\mphi[\bfS^{{\sigma}}_{{\ldalpha},C}]$ as an $H_n(0)$-module if and only if ${\ldalpha}$ is ${\sigma}$-simple.
In particular, if $\alpha$ is a partition, then $X_\alpha$ is isomorphic to $\phi[\bfS^{w_0}_\alpha]$.
\end{corollary}
\begin{proof}
We first observe that a partition $\mu$ is $\rho$-simple for some permutation $\rho \in \SG_{\ell(\mu)}$
if and only if there does not exist a PACD pair attached to $(\mu \, ; \rho)$.

In order to prove the ``if" part of the assertion, we suppose that ${\ldalpha}$ is ${\sigma}$-simple.
If there exists $\sfT \in \SET(\alpha)$ contained in $\ker(\Ups)$,
then there is a triple $(i,j,k)$ satisfying both \textbf{K1} and \textbf{K2} in Remark~\ref{rem: ker of Ups}.
In this case one can see that $(\sigma^{-1}(i) , \sigma^{-1}(j))$ is a PACD pair attached to $({\ldalpha} \, ; \sigma)$, which is absurd.
Thus $\ker(\Ups) = 0$.

For the ``only if" part of the assertion,
we assume that ${\ldalpha}$ is not ${\sigma}$-simple.
Then we can choose a PACD pair $(i,j)$ attached to $({\ldalpha} \, ; \sigma)$, that is, $i<j$, $\sigma(i) < \sigma(j)$ and ${\ldalpha}_i \geq {\ldalpha}_j \geq 2$.
Let $\sfT_\alpha$ be the tableau in $\SET(\alpha)$ filled with entries $1, 2, \ldots, n$ from bottom to top and from left to right.
Then the triple $(\sigma(i), \sigma(j), 1)$ in $\sfT_\alpha$ satisfies both \textbf{K1} and \textbf{K2},
and hence $\sfT_\alpha \in \ker(\Ups)$.
This proves the contraposition of the ``only if'' part.
\end{proof}

\section{A series of surjections starting from $\bfSsaC$}
\label{sec: surj from bfSsaC}

In this section, we investigate surjections starting from $\bfSsaC$. We give a surjection from $\bfSsaC$ to $\bfS^{\sigma s_i}_{\alpha \cdot s_i,C}$ when $\ell(\sigma s_i) < \ell(\sigma)$.
Then we define standard permuted Young composition tableaux ($\SPYCT$'s).
Using these combinatorial objects, we introduce a new $H_n(0)$-module $\hbfS^\sigma_{\alpha}$, which is isomorphic to $\mphi[\bfS^{\sigma^{w_0}}_{\alpha^\rmr}]$ with $\sigma^{w_0} = w_0 \sigma w_0^{-1}$.
As in the previous section, $\mphi[\bfS^{\sigma^{w_0}}_{\alpha^\rmr}]$ denotes the $\phi$-twist of $\bfS^{\sigma^{w_0}}_{\alpha^\rmr}$.
Combining the results of Section~\ref{Section3} and Section~\ref{sec: surj from bfSsaC}, we derive a series of surjections starting from $\bfP_\alpha$ and ending at $\hbfS^{\id}_{\alpha,C}$ (see Corollary~\ref{Coro:Seq_of_PVXS}).

\subsection{A surjection from $\bfSsaC$ to $\bfS^{\sigma s_i}_{\alpha \cdot s_i,C}$ when $\ell(\sigma s_i) < \ell(\sigma)$}\label{subsec: surj from SPCTs}

Note that $H_k(0)$ acts on the right on $\C[\Z^k]$, the $\C$-span of $\Z^k$, by
\begin{align}\label{eq: bubble sorting operator}
\bfm \bubact \pi_i = \begin{cases}
\bfm \cdot s_i & \text{if $m_i < m_{i+1}$,}\\
\bfm & \text{otherwise}
\end{cases}
\end{align}
for $1\le i \le k-1$ and $\bfm = (m_1,\ldots,m_k) \in \Z^k$. 
Here $\bfm \cdot s_i$ is obtained from $\bfm$ by swapping $m_i$ and $m_{i+1}$.
Obviously, this action restricts to the $\C$-span of compositions of length $k$.

In this subsection, we will assume that $\alpha \models n$ is compatible with $\sigma \in \SG_{\ell(\alpha)}$. 
It was shown in~\cite{20CKNO} that 
for an index $1 \le i \le \ell(\alpha) - 1$ such that $\ell(\sigma s_i) < \ell(\sigma)$, we have 
\begin{align*}
\ch([\bfSsa]) = \sum_{\alpha = \beta \bubact \pi_{i}} \ch([\bfS^{\sigma s_i}_\beta]).
\end{align*}
It is currently unknown whether this characteristic relation can be lifted to the level of $H_n(0)$-modules.
In the following we provide an interesting relation between $\bfSsaC$ and $\bfS^{\sigma s_i}_{\alpha \cdot s_i,C}$ that can be regarded as a first step towards solving this problem.

Suppose that there is an index $1 \le i \le \ell(\alpha) - 1$, which will be fixed in the rest of this subsection, such that $\ell(\sigma s_i) < \ell(\sigma)$, that is, $\sigma(i) > \sigma(i+1)$.
By the $\sigma$-compatibility of $\alpha$, we have that $\alpha_{i} \ge \alpha_{i+1}$.

Define a map 
\[
\mpsi^{(i)}: \SPCTsa \ra \coprod_{\alpha = \beta \bubact \pi_i} \SPCT^{\sigma s_i}(\beta),
\]
where $\mpsi^{(i)}(\tau)$ is defined as follows:
\begin{enumerate}[label = {$\bullet$}]
\item In case where $\tau_{i+1,j} < \tau_{i,j+1}$ for all $1 \le j \le \alpha_i$, define $\mpsi^{(i)}(\tau)$ to be the SPCT of shape $\alpha \cdot s_i$ and type $\sigma s_i$ obtained from $\tau$ by swapping the $i$th row and the $(i+1)$st row.
\item 
Otherwise, let $j_0$ be the smallest integer such that $\tau_{i+1,j_0} > \tau_{i,j_0+1}$. 
In this case, define $\mpsi^{(i)}(\tau)$ to be the SPCT of shape $\alpha$ and type $\sigma s_i$ obtained from $\tau$ by swapping $\tau_{i,j}$ and $\tau_{i+1,j}$ for all $1 \le j \le j_0$.
\end{enumerate}

Let $\lambda$ be a partition and $\sigma_1, \sigma_2 \in \SG_{\ell(\lambda)}$. In~\cite{20CKNO}, the authors constructed a map $\mpsi_{\sigma_1}^{\sigma_2}$ from $\coprod_{\tal = \lambda} \SPCT^{\sigma_1}(\alpha)$ to $\coprod_{\tal = \lambda} \SPCT^{\sigma_2}(\alpha)$ in a combinatorial manner.
Using~\cite[Lemma 4.5]{20CKNO}, it is not difficult to show that $\mpsi^{(i)}(\tau)$ is identical to $\mpsi_{\sigma}^{\sigma s_i}(\tau)$.
Combining this with~\cite[Proposition 4.6]{20CKNO}, we deduce that $\mpsi^{(i)}$ is well-defined and bijective.

\begin{example}
Given $\alpha = (3,4,2)$ and $\sigma = 231$, let
\[
\tau = \begin{array}{l}
\begin{ytableau}
6 & 3 & 2 \\
9 & 8 & 7 & 5 \\
4 & 1
\end{ytableau} 
\end{array} \in \SPCT^\sigma(\alpha).
\]
Since $\tau_{3,j} < \tau_{2,j+1}$ for all $1 \le j \le \alpha_2$, 
\[
\psi^{(2)}(\tau) = \begin{array}{l}
\begin{ytableau}
6 & 3 & 2 \\
4 & 1 \\
9 & 8 & 7 & 5
\end{ytableau} 
\end{array} \in \SPCT^{\sigma s_2}({\alpha\cdot s_2}).
\]
On the other hand, since $\psi^{(2)}(\tau)_{2,1} > \psi^{(2)}(\tau)_{1,2}$, $j_0 = 1$ and thus
\[
\psi^{(1)}\circ\psi^{(2)}(\tau) = \begin{array}{l}
\begin{ytableau}
4 & 3 & 2 \\
6 & 1 \\
9 & 8 & 7 & 5
\end{ytableau} 
\end{array}\in \SPCT^{\sigma s_2 s_1}({\alpha\cdot s_2s_1}).
\]
\end{example}

We define a linear map $\mPsi^{(i)}: \bfSsa \ra \bfS^{\sigma s_i}_{\alpha \cdot s_i}$ by
\[
\mPsi^{(i)}(\tau) := \begin{cases}
\mpsi^{(i)}(\tau) & \text{if $\mpsi^{(i)}(\tau) \in \SPCT^{\sigma s_i}(\alpha \cdot s_i)$},\\
0 & \text{otherwise.}
\end{cases}
\]
Let $\iota_C: \bfSsaC \ra \bfSsa$ be the inclusion map and $\pr_C: \bfS^{\sigma s_i}_{\alpha \cdot s_i} \ra \bfS^{\sigma s_i}_{\alpha \cdot s_i,C}$ the projection. Finally, we define a linear map
\[
\mPsi^{(i)}_C: \bfSsaC \ra \bfS^{\sigma s_i}_{\alpha \cdot s_i,C}, \quad 
v \mapsto (\pr_C \circ \mPsi^{(i)} \circ \iota_C) (v).
\]

\begin{theorem}\label{thm: surj btw canonical}
For $1 \le i \le \ell(\alpha) - 1$ such that $\ell(\sigma s_i) < \ell(\sigma)$,
the map $\mPsi^{(i)}_C: \bfSsaC \ra \bfS^{\sigma s_i}_{\alpha \cdot s_i,C}$ is a surjective $H_n(0)$-module homomorphism.
\end{theorem}

\begin{proof}
First, we show that the map $\mPsi^{(i)}_C$ is surjective. 
Let $\tau$ be an SPCT in $\bfS^{\sigma s_i}_{\alpha \cdot s_i,C}$. 
Let $\otau$ be the filling obtained from $\tau$ by swapping the $i$th row and the $(i+1)$st row. 
The only nontrivial part in showing $\otau \in \bfSsaC$ is to verify that $\otau$ satisfies the triple condition at the boxes in the $i$th and $(i+1)$st row.
Suppose that there is $k \ge 1$ such that $\otau_{i,k} > \otau_{i+1,k+1}$ and $\otau_{i,k+1} < \otau_{i+1,k+1}$, equivalently, $\tau_{i+1,k} > \tau_{i,k+1}$ and $\tau_{i+1,k+1} < \tau_{i,k+1}$. 
However, the latter inequality cannot occur. 
This is because the conditions $\sigma(i) > \sigma(i+1)$ and $\tau \in \bfS^{\sigma s_i}_{\alpha \cdot s_i,C}$ imply that $\tau_{i+1,j} > \tau_{i,j}$ for every $1 \le j \le \min\{\alpha_i, \alpha_{i+1}\}$.
Hence, $\otau$ is an SPCT in $\bfS^{\sigma}_{\alpha,C}$.
By the triple condition, we have $\tau_{i,j} < \tau_{i+1,j+1}$ for $1\le j \le \min\{\alpha_i, \alpha_{i+1}\}$, equivalently, $\otau_{i+1,j} < \otau_{i,j+1}$ for $1\le j \le \min\{\alpha_i, \alpha_{i+1}\}$, thus $\psi^{(i)}(\otau) = \tau$.
As a consequence, $\mPsi^{(i)}_C(\otau) = \tau$.

Next, let us prove that $\mPsi^{(i)}_C$ is an $H_n(0)$-module homomorphism.
Let $1 \le j < n$ and $\tau \in \bfS^{\sigma}_{\alpha,C}$.
We have two cases.

{\it Case 1: $j$ is in the left-adjacent column of $j+1$ in $\tau$.}
If $j$ or $j+1$ is neither in the $i$th row nor in the $(i+1)$st row, then the relative position of $j$ and $j+1$ does not change. 
In this case $\pi_j \cdot \mPsi^{(i)}_C(\tau) = \mPsi^{(i)}_C(\pi_j \cdot \tau)$, so we are done. 

Let us deal with the remaining case where $j, j+1$ are in row $i$ or row $(i+1)$, but they are positioned in distinct rows.
Since $\sigma(i) > \sigma(i+1)$, we have $\tau_{i,1} > \tau_{i+1,1}$.
By the triple condition, $\tau_{i,k} > \tau_{i,k+1} > \tau_{i+1,k+1}$ for all $1 \le k < \min\{\alpha_i,\alpha_{i+1}\}$. 
Therefore $j$ appears in the $(i+1)$st row and $j+1$ in the $i$th row of $\tau$.
Let $j$ be in the $k$th column, that is, $j = \tau_{i+1,k}$ and $j+1 = \tau_{i,k+1}$.
If the positions of $j$ and $j+1$ are invariant under $\mpsi^{(i)}$, then there is nothing to prove.
Otherwise, $j$ appears at $(i,k)$ and $j+1$ at $(i+1,k+1)$ in $\mpsi^{(i)}(\tau)$. 
This means that $j$ is an attacking descent of $\mpsi^{(i)}(\tau)$, thus $\pi_j \cdot \mPsi^{(i)}_C(\tau) = 0$. 
On the other hand, in $\pi_j \cdot \tau$, $j$ appears at $(i,k+1)$ and $j+1$ at $(i+1,k)$.
By the definition of $\mpsi^{(i)}$, we deduce that
the $(k+1)$st column of $\mpsi^{(i)}(\pi_j \cdot \tau)$ is equal to that of $\pi_j \cdot \tau$.
Combining this with the conditions $\sigma(i) > \sigma(i+1)$ and $\pi_j \cdot \tau$ is in the canonical class yields
\[
(\mpsi^{(i)}(\pi_j \cdot \tau))_{i,k+1} > (\mpsi^{(i)}(\pi_j \cdot \tau))_{i+1,k+1}.
\]
One can see that $(\mpsi^{(i)}(\pi_j \cdot \tau))_{i,1} < (\mpsi^{(i)}(\pi_j \cdot \tau))_{i+1,1}$, thus $\mpsi^{(i)}(\pi_j \cdot \tau)$ is not in the canonical class.
Hence $\mPsi^{(i)}_C(\pi_j \cdot \tau) = 0$.

{\it Case 2: $j$ is not in the left-adjacent column of $j+1$ in $\tau$.}
Note that for $\tau \in \SPCTsa$, the set of entries in the $p$th column of $\tau$ is equal to that in the $p$th column of $\psi^{(i)}(\tau)$ for $p \ge 1$.
This implies that $j$ is also not in the left adjacent column of $j+1$ in $\mpsi^{(i)}(\tau)$. 
Therefore $\pi_j \cdot \mPsi^{(i)}_C(\tau) = \mPsi^{(i)}_C(\pi_j \cdot \tau)$.
\end{proof}

As an immediate consequence of Theorem~\ref{thm: surj btw canonical}, we derive a series of surjections starting from a $\bfSsaC$.
\begin{corollary}\label{cor: seq of surj}
For any reduced expression $s_{i_1}\cdots s_{i_k}$ of $\sigma$, 
we have the following series of surjections:
\[
\begin{tikzcd}
\bfS^{\sigma}_{\alpha,C} \arrow[r, two heads] &
\bfS^{\sigma s_{i_k}}_{\alpha \cdot s_{i_k},C} \arrow[r, two heads] &
\bfS^{\sigma s_{i_k} s_{i_{k-1}}}_{\alpha \cdot s_{i_k}s_{i_{k-1}},C} \arrow[r, two heads] &
\cdots \arrow[r, two heads] &
\bfS^{\id}_{\alpha \cdot \sigma^{-1},C}
\end{tikzcd}
\]
\end{corollary}

\subsection{The description of the $\phi$-twist of $\bfS^{\sigma}_{\alpha}$ in terms of standard permuted Young composition tableaux}
\label{subsec: SPYCT}

Standard Young composition tableaux were introduced in~\cite{13LMvW} as a
combinatorial model for the image of the quasisymmetric Schur functions under the automorphism $\rho: \Qsym \ra \Qsym$, $F_\alpha \mapsto F_{\alpha^\rmr}$.
We here introduce their permuted version, standard permuted Young composition tableaux, which turn out to be very suitable to describe the $H_n(0)$-module $\phi[\bfSsa]$.

\begin{definition}\label{def: SPYCT}
Given $\alpha \models n$ and $\sigma \in \SG_{\ell(\alpha)}$, a \emph{standard permuted Young composition tableau} $\sytab$ ($\SPYCT$) of shape $\alpha$ and type $\sigma$ is a filling of $\trcd(\alpha)$ with entries in $\{1,2,\ldots,n\}$ such that the following conditions hold:
\begin{enumerate}[label = {\rm (\arabic*)}]
    \item The entries are all distinct.
    \item The standardization of the word obtained by reading the first column from bottom to top is $\sigma$.
    \item The entries in each row are increasing from left to right.
   \item If $i<j$ and $\sytab_{i,k} < \sytab_{j,k+1}$, then $(i,k+1) \in \trcd(\alpha)$ and $\sytab_{i,k+1} < \sytab_{j,k+1}$.
\end{enumerate}
The condition (4) is called the \emph{Young triple condition}.
\end{definition}

In this subsection, let $\alpha \models n$ and $\sigma \in \SG_{\ell(\alpha)}$. 
We denote by $\SPYCT^\sigma(\alpha)$ the set of all SPYCTs of shape $\alpha$ and type $\sigma$. 
Let $\sytab \in \SPYCT^\sigma(\alpha)$. An integer $1\le i \le n-1$ is a \emph{descent} of $\sytab$ if $i+1$ lies weakly left of $i$ in $\sytab$. 
Denote by $\Des(\sytab)$ the set of all descents of $\sytab$ and set $\comp(\sytab) := \comp(\Des(\sytab))$.
For $1\le i < j \le n$, we say that $i$ and $j$ are \emph{attacking} (in $\sytab$) if either
\begin{enumerate}[label = {\rm (\roman*)}]
    \item $i$ and $j$ are in the same column, or
    \item $i$ and $j$ are in adjacent columns, with $j$ positioned upper-left of $i$.
\end{enumerate}
In case where $i$ and $i+1$ are attacking (resp. nonattacking) and $i$ is a descent of $\sytab$, we simply say that $i$ is an \emph{attacking descent} (resp. \emph{nonattacking descent}).

Let $\hbfS^\sigma_\alpha$ be the $\C$-span of $\SPYCT^\sigma(\alpha)$ and
$\sigma^{w_0}$ the conjugation of $\sigma$ by $w_0$, that is, $\sigma^{w_0} = w_0 \sigma w_0^{-1}$.
Define a $\C$-linear isomorphism $\upphi: \hbfS^\sigma_\alpha \ra \phi[\bfS^{\sigma^{w_0}}_{\alpha^\rmr}]$ by letting
\[
\upphi(\sytab)_{i,j} = n + 1 -\sytab_{i,j},
\]
then extending it by linearity.

Define an $H_n(0)$-action on $\hbfS^\sigma_\alpha$ by transporting that of $\phi[\bfS^{\sigma^{w_0}}_{\alpha^\rmr}]$
via $\upphi$, precisely,
\begin{align*}
h \cdot v = \upphi^{-1}\left( h \cdot \upphi(v) \right) \quad \text{for $h \in H_n(0)$ and $v \in \hbfS^\sigma_\alpha$}.
\end{align*}
It follows that 
\begin{align}\label{eq: isom SYCT and SPCT module}
\hbfS^\sigma_\alpha \cong \phi[\bfS^{\sigma^{w_0}}_{\alpha^\rmr}]
\end{align}
as an $H_n(0)$-module.
One can see that the generators $\pi_i$ act on $\SPYCT^\sigma(\alpha)$ in the following way:
\begin{align*}
\pi_i \cdot \sytab := \begin{cases}
\sytab & \text{if $i$ is not a descent},\\
0 & \text{if $i$ is an attacking descent},\\
s_i \cdot \sytab & \text{if $i$ is a nonattacking descent}
\end{cases}
\end{align*}
for $1\le i \le n-1$ and $\htau \in \SPYCT^\sigma(\alpha)$.

\begin{example}
The $H_5(0)$-actions on $\SPCT^{132}((2,2,1))$ and $\SPYCT^{213}((1,2,2))$ are illustrated as follows:
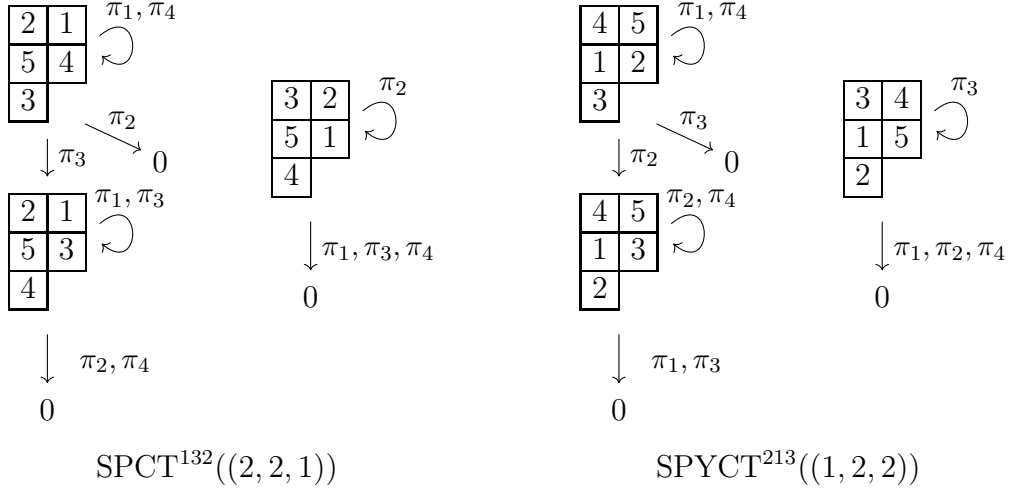
\[
\def \hp{0.25}
\def \vp{0.25}
\begin{tikzpicture}[baseline = 0mm]
\node[below] at (2.25,-4.5) {$\SPCT^{132}((2,2,1))$};

\node at (0,0+2*\vp) {
\begin{ytableau}
2 & 1 \\
5 & 4 \\
3
\end{ytableau}};
\node at (0.7-\hp,0.25+2*\vp) {} edge [out=40,in=320, loop] ();
\node at (1.4-\hp,0.7+2*\vp) {\small $\pi_1,\pi_4$};

\node at (1.5,-1.3+2*\vp) {$0$};
\draw [->] (0.5,-0.8+2*\vp) -- (1.25,-1.15+2*\vp);
\node at (1,-0.75+2*\vp) {\small $\pi_2$};

\draw [->] (0,-1.25+3*\vp) -- (0,-1.75+3*\vp);
\node[right] at (0.,-1.5+3*\vp) {\small $\pi_3$};

\node at (0,-3+4*\vp) {
\begin{ytableau}
2 & 1 \\
5 & 3 \\
4
\end{ytableau}};
\node at (0.7-\hp,-2.75+4*\vp) {} edge [out=40,in=320, loop] ();
\node at (1.4-2*\hp,-2.3+4*\vp) {\small $\pi_1,\pi_3$};

\draw [->] (0,-4.1+4*\vp) -- (0,-4.75+4*\vp);
\node at (0.9,-4.5+4*\vp) {\small $\pi_2,\pi_4$};
\node at (0,-5.1+4*\vp) {$0$};


\node at (4-0.5,-3+4*\vp+1.5) {
\begin{ytableau}
3 & 2 \\
5 & 1 \\
4
\end{ytableau}};
\node at (4.7-\hp-0.5,-2.75+4*\vp+1.5) {} edge [out=40,in=320, loop] ();
\node at (5.4-2*\hp-0.5,-2.3+4*\vp+1.5) {\small $\pi_2$};

\draw [->] (4-0.5,-4.1+4*\vp+1.5) -- (4-0.5,-4.75+4*\vp+1.5);
\node at (4.9-0.5,-4.5+4*\vp+1.5) {\small $\pi_1,\pi_3,\pi_4$};
\node at (4-0.5,-5.1+4*\vp+1.5) {$0$};
\end{tikzpicture}
\qquad \qquad 
\begin{tikzpicture}[baseline = 0mm]
\node[below] at (2.25,-4.5) {$\SPYCT^{213}((1,2,2))$};

\node at (0,0+2*\vp) {
\begin{ytableau}
4 & 5 \\
1 & 2 \\
3
\end{ytableau}};
\node at (0.7-\hp,0.25+2*\vp) {} edge [out=40,in=320, loop] ();
\node at (1.4-\hp,0.7+2*\vp) {\small $\pi_1,\pi_4$};

\node at (1.5,-1.3+2*\vp) {$0$};
\draw [->] (0.5,-0.8+2*\vp) -- (1.25,-1.15+2*\vp);
\node at (1,-0.75+2*\vp) {\small $\pi_3$};

\draw [->] (0,-1.25+3*\vp) -- (0,-1.75+3*\vp);
\node[right] at (0.,-1.5+3*\vp) {\small $\pi_2$};

\node at (0,-3+4*\vp) {
\begin{ytableau}
4 & 5 \\
1 & 3 \\
2
\end{ytableau}};
\node at (0.7-\hp,-2.75+4*\vp) {} edge [out=40,in=320, loop] ();
\node at (1.4-2*\hp,-2.3+4*\vp) {\small $\pi_2,\pi_4$};

\draw [->] (0,-4.1+4*\vp) -- (0,-4.75+4*\vp);
\node at (0.9,-4.5+4*\vp) {\small $\pi_1,\pi_3$};
\node at (0,-5.1+4*\vp) {$0$};

\node at (4-0.5,-3+4*\vp+1.5) {
\begin{ytableau}
3 & 4 \\
1 & 5 \\
2
\end{ytableau}};
\node at (4.7-\hp-0.5,-2.75+4*\vp+1.5) {} edge [out=40,in=320, loop] ();
\node at (5.4-2*\hp-0.5,-2.3+4*\vp+1.5) {\small $\pi_3$};

\draw [->] (4-0.5,-4.1+4*\vp+1.5) -- (4-0.5,-4.75+4*\vp+1.5);
\node at (4.9-0.5,-4.5+4*\vp+1.5) {\small $\pi_1,\pi_2,\pi_4$};
\node at (4-0.5,-5.1+4*\vp+1.5) {$0$};
\end{tikzpicture}
\]
\end{example}

For each $E \in \calE^{\sigma^{w_0}}(\alpha^\rmr)$, denote $\upphi^{-1}(\phi[\bfS^{\sigma^{w_0}}_{\alpha^\rmr,E}])$ by $\hbfS^\sigma_{\alpha,E}$. 
Then $\hbfS^\sigma_\alpha$ is decomposed into indecomposable modules as follows:
\begin{align*}
\hbfS^\sigma_\alpha 
= \upphi^{-1}(\phi[\bfS^{\sigma^{w_0}}_{\alpha^\rmr}]) 
\cong \bigoplus_{E \in \calE^{\sigma^{w_0}}(\alpha^\rmr)} \upphi^{-1}\left(\phi[\bfS^{\sigma}_{\alpha^\rmr,E}]\right) 
\cong \bigoplus_{E \in \calE^{\sigma^{w_0}}(\alpha^\rmr)} \hbfS^\sigma_{\alpha,E}
\end{align*}
For the canonical submodule $\bfS^{\sigma^{w_0}}_{\alpha^\rmr,C}$ of $\bfS^{\sigma^{w_0}}_{\alpha^\rmr}$, denote $\upphi^{-1}(\phi[\bfS^{\sigma^{w_0}}_{\alpha^\rmr,C}])$ by $\hbfS^\sigma_{\alpha,C}$.

When $\ell(\sigma s_{i}) < \ell(\sigma)$, the surjective homomorphism $\mPsi^{(i)}_C: \bfS^{\sigma}_{\alpha,C} \ra \bfS^{\sigma s_{i}}_{\alpha\cdot s_{i} ,C}$ in Theorem~\ref{thm: surj btw canonical} induces a surjective $H_n(0)$-module homomorphism from 
$\hmPsi^{(i)}_C: \phi[\bfS^{\sigma}_{\alpha,C}] \ra \phi[\bfS^{\sigma s_{i}}_{\alpha \cdot s_{i} ,C}]$ sending $x$ to $\mPsi^{(i)}_C(x)$.
With this notation, we derive the following noteworthy series of surjections.

\begin{corollary}
\label{Coro:Seq_of_PVXS}
Let $\sigma \in \SG_{\ell(\alpha)}$ be a permutation satisfying $\lambda(\alpha)^\rmr = \alpha \cdot \sigma$. For any reduced expression $s_{i_1}\cdots s_{i_k}$ of $\sigma$, we have the following series of surjections:
\[
\begin{tikzcd}
\bfP_{\alpha}  \arrow[r, two heads] &
\calV_{\alpha} \arrow[r, two heads] &
X_{\alpha} \arrow[r,  two heads] &
\hbfS^{\sigma}_{\ldalpha^\rmr,C} \arrow[r, two heads] &
\hbfS^{\sigma s_{i_{k}}}_{\ldalpha^\rmr \cdot s_{i_k},C} \arrow[r, two heads] &
\cdots \arrow[r, two heads] &
\hbfS^{\id}_{\alpha,C}
\end{tikzcd}
\]
In particular, if $\alpha$ is a partition, then we have
\[
\begin{tikzcd}
\bfP_{\alpha}  \arrow[r, two heads] &
\calV_{\alpha} \arrow[r, two heads] &
X_{\alpha} \cong
\hbfS^{w_0}_{\alpha^\rmr} \arrow[r, two heads] &
\hbfS^{w_0 s_{i_{k}}}_{\alpha^\rmr \cdot s_{i_k},C} \arrow[r, two heads] &
\cdots \arrow[r, two heads] &
\hbfS^{\id}_{\alpha,C}.
\end{tikzcd}
\]
\end{corollary}
\begin{proof}
Combining Theorem~\ref{thm: X isom to bfS} with~\eqref{eq: isom SYCT and SPCT module} yields that 
$\upphi^{-1} \circ \Ups : X_\alpha \rightarrow \hbfS^{\sigma}_{\ldalpha^\rmr,C}$ is a surjective $H_n(0)$-module homomorphism.
And, from Theorem~\ref{thm: surj btw canonical} it follows that 
$\upphi^{-1} \circ \hmPsi^{(i)}_C \circ \upphi: \hbfS^{\sigma}_{\alpha,C} \ra \hbfS^{\sigma s_i}_{\alpha \cdot s_i,C}$ is a surjective $H_n(0)$-module homomorphism.
Now, the desired result is obtained by~\eqref{eq: surj series 1}. 
\end{proof}

\begin{remark}
Let $\alpha \models n$. 
Choose a permutation $\sigma \in \SG_{\ell(\alpha)}$ satisfying that $\lambda(\alpha)^\rmr = \alpha \cdot \sigma$ and $\lambda(\alpha)^\rmr = \alpha \cdot (\sigma s_i)$ for some $1 \le i \le \ell(\alpha) - 1$.
From the definitions of the maps $\Ups$, $\upphi$, and $\hmPsi^{(i)}_C$ we have the following commuting diagram:
\[
\begin{tikzcd}
X_\alpha \arrow[rr, "\Ups"] \arrow[rrd, "\Ups"'] &  & {\phi[\bfS^{\sigma^{w_0}}_{\ldalpha,C}]} \arrow[rr, "\upphi^{-1}"] \arrow[d, " \hmPsi^{(i)}_C"]   &  & {\hbfS^{\sigma}_{\ldalpha^\rmr,C}} \arrow[d, "\upphi^{-1} \circ \hmPsi^{(i)}_C \circ \upphi"] \\
&  & {\phi[\bfS^{(\sigma s_{i})^{w_0}}_{\ldalpha,C}]} \arrow[rr, "\upphi^{-1}"] &  & {\hbfS^{\sigma s_{i}}_{\ldalpha^\rmr, C}}                 
\end{tikzcd}
\]
\end{remark}

Let us consider the quasisymmetric characteristic image of $\hbfS^\sigma_{\alpha}$.
Recall that $\rho: \Qsym \ra \Qsym$ is the involution defined by $\rho(F_\alpha) = F_{\alpha^\rmr}$
and $\phi: H_n(0) \ra H_n(0)$ is the involution satisfying that 
$\phi [\bfF_\alpha] \cong \bfF_{\alpha^{\rmr}}$.
This implies that for any finite dimensional $H_n(0)$-module $M$,
we have
\begin{align}\label{eq : mphi and rho}
\ch([\mphi[M]]) = \rho (\ch[M]).
\end{align}

In \cite{13LMvW}, the {\em Young quasisymmetric Schur function} $\hcalS_\alpha$ is defined by 
$\rho(\calS_{\alpha^\rmr})$, where $\calS_{\alpha^\rmr}$ is the quasisymmetric Schur function.
Combining~\eqref{eq: isom SYCT and SPCT module} with~\eqref{eq : mphi and rho},
we derive that
$$
\ch([\hbfSsa]) = \rho (\ch([\bfS^{\sigma^{w_0}}_{\alpha^\rmr}])).
$$
In particular, 
\[
\ch([\hbfS^\id_\alpha]) = \hcalS_\alpha. 
\]
We obtain a recursive relation for $\ch([\hbfSsa])$'s
from~\cite[Theorem 4.7]{20CKNO}.

\begin{theorem}
Let $\alpha \models n$ and $\sigma \in \SG_{\ell(\alpha)}$.
For $1 \le i \le \ell(\alpha)-1$ such that $\ell(\sigma s_i) < \ell(\sigma)$, 
we have
\begin{align*}
\ch([\hbfSsa]) = \sum_{\alpha^\rmr = \beta^\rmr \bubact \pi_{n-i}} \ch([\hbfS^{ \sigma s_{i}}_\beta]).
\end{align*}
\end{theorem}

Applying $\rho: \Qsym \ra \Qsym$ to~\cite[Corollary 4.8]{20CKNO}, 
we obtain the expansion of $\ch([\hbfSsa])$ in terms of Young quasisymmetric Schur functions.
\begin{corollary}\label{cor: qschur exp}
Let $\alpha$ be a composition. 
For any $\sigma \in \SG_{\ell(\alpha)}$, we have
\begin{align*}
\ch([\hbfSsa]) &= \sum_{\alpha^\rmr = \beta^\rmr \bubact \pi_{\sigma^{w_0}}} \hcalS_{\beta}.
\end{align*}
Furthermore, if $\alpha^\rmr$ is a partition and $w_0$ is the longest element in $\SG_{\ell(\alpha)}$, 
then
\begin{align*}
\ch([\hbfS_\alpha^{w_0}]) &= s_{\alpha^\rmr}, 
\end{align*}
where $s_{\alpha^\rmr}$ is the Schur function.
\end{corollary}

\section{The projective cover of $\bfS^{\sigma}_{\alpha,E}$ for an arbitrary class $E$}
\label{Sec: Projective cover of E}

In this section, we assume that $\alpha \models n$ is compatible with $\sigma \in \SG_{\ell(\alpha)}$. 
It was shown in~\cite{20CKNO} that 
$\bfP_{\alpha \cdot \sigma^{-1}}$ is the projective cover of $\bfS^\sigma_{\alpha,C}$
for the canonical class $C$ of $\mathcal{E}^\sigma(\alpha)$.
The purpose of this section is to find the projective cover of 
$\bfS^\sigma_{\alpha,E}$ for all classes $E \in \mathcal{E}^\sigma(\alpha)$
and 
to classify $\bfS^\sigma_{\alpha,E}$'s whose projective covers are indecomposable.
To begin with,
let us collect necessary  definitions and notation.
\begin{definition}
Suppose that 
$H$ and $H'$ are two disjoint connected horizontal strips in $\tcd(\alpha)$.
\begin{enumerate}[label = {\rm (\alph*)}]
\item 
We say that {\it $H$ and $H'$  share a column} (in $\tcd(\alpha)$) 
if there exists a column of $\tcd(\alpha)$ which intersects with both $H$ and $H'$.

\item 
We say that {\it $H$ and $H'$ are attacking} (in $\tcd(\alpha)$) if 
the rightmost box of $H$ and the leftmost box of $H'$
are in adjacent columns,
with the latter positioned lower-right of the former.
\end{enumerate}
\end{definition}

For example, let us consider the following composition diagram:
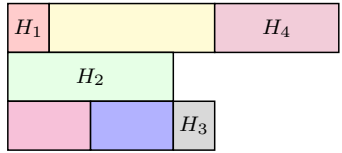
\begin{displaymath}
\begin{tikzpicture}
\def\hhh{4mm}
\def\vvv{5mm}
\def\www{0.20mm}
\draw[line width=\www, fill=red!20] (\hhh*0,\vvv*0) rectangle (\hhh*1,\vvv*1);
\draw[line width=\www, fill=yellow!20] (\hhh*1,\vvv*0) rectangle (\hhh*5,\vvv*1);
\draw[line width=\www, fill=purple!20] (\hhh*5,\vvv*0) rectangle (\hhh*8,\vvv*1);
\draw[line width=\www, fill=green!10] (\hhh*0,-\vvv*1) rectangle (\hhh*4,\vvv*0);
\draw[line width=\www, fill=magenta!30] (\hhh*0,-\vvv*2) rectangle (\hhh*2,-\vvv*1);
\draw[line width=\www, fill=blue!30] (\hhh*2,-\vvv*2) rectangle (\hhh*4,-\vvv*1);
\draw[line width=\www, fill=gray!30] (\hhh*4,-\vvv*2) rectangle (\hhh*5,-\vvv*1);
\node[] at (\hhh*0.5,\vvv*0.5) {\tiny $H_1$};
\node[] at (\hhh*6.5,\vvv*0.5) {\tiny $H_4$};
\node[] at (\hhh*2,-\vvv*0.5) {\tiny $H_2$};
\node[] at (\hhh*4.5,-\vvv*1.5) {\tiny $H_3$};
\end{tikzpicture}
\end{displaymath}
The horizontal strips $H_1, H_2$ share the first column,
but $H_2, H_3$ and $H_3,H_4$ do not share a column.
The rightmost box of $H_2$ and the leftmost box of $H_3$
are in adjacent columns and the latter is lower-right of the former,
so $H_2$ and $H_3$ are attacking.

Given an SPCT $\tau$ and an entry $i$ in $\tau$, we denote by ${\sf r}_\tau(i)$ and ${\sf c}_\tau(i)$ the row index and column index of $\tau$, respectively.
\smallskip

\noindent {\bf Convention.}
Since we are working inside $E \in \calE^\sigma(\alpha)$,
when $\tau$ is the source tableau $\tauE$ in $E$, we omit the subscript $\tauE$ from ${\sf r}_{\tauE}(i)$ and $\sfc_{\tauE}(i)$ for simplicity.
\smallskip

Set $\Des(\tauE) = \{d_1 <  d_2 <  \cdots < d_m \}$, $d_0:=0$, and $d_{m+1}: = n$.
Let $\bal^{(1)} = (1^{d_1})$.
For $1 \leq j \le m$, define 
\[
\bal^{(j+1)} = 
\begin{cases}
\bal^{(j)} \odot (1^{d_{j+1} - d_{j}}) & 
\begin{array}{l}
\text{if $d_{j-1} + 1$ is weakly right of $d_{j+1}$ or}\\ 
\text{$d_{j-1}+1$ and $d_{j+1}$ are attacking in $\tauE$,}
\end{array}\\
\bal^{(j)} \oplus (1^{d_{j+1} - d_{j}}) & 
\text{~otherwise,}
\end{cases} 
\]
where $\bal^{(j)} \odot (1^{d_{j+1} - d_{j}})$ is the near concatenation of $\bal^{(j)}$ and $(1^{d_{j+1} - d_{j}})$ (see Subsection~\ref{subsec: PIM}).
Denote $\bal_E := \bal^{(m+1)}$.

Let us depict $\bal_E$ in a more pictorial manner.
For $1 \le j \le m+1$, define $\tH_j$ to be  the horizontal strip consisting of boxes with entries from $d_{j-1} +1$ to $d_j$ in $\tauE$.
For each $2 \le j \le m+1$, the ribbon diagram $\trd(\bal^{(j)})$ can be constructed recursively in the following steps:
\begin{enumerate}[label = {\rm (\roman*)}] 
\item Rotate $\tH_j$ to be a vertical strip. 
\item Attach this vertical strip to the right $\trd(\bal^{(j-1)})$ so that column $j$ and column $j+1$ of the resulting generalized ribbon diagram are connected to each other if $\tH_j$ and $\tH_{j+1}$ 
share a column or are attacking in $\tcd(\alpha)$, disconnected otherwise.
\end{enumerate}

Let $\rtE \in \SRT(\bal_E)$ be SRT obtained by filling $\trd(\bal_E)$
with entries $1, 2, \ldots, n$ from top to bottom and from left to right. 
It should be noticed that this tableau is nothing but $T_0 \in \SRT(\bal_E)$ introduced in Subsection~\ref{subsec: PIM}.

We can construct $\rtE$ directly from the source tableau $\tauE$.
Given $\tau \in E$, let us use the notation $\tau(\tH_j)$ to denote the subfilling of $\tau$ occupied by $\tH_j$ in $\tcd(\alpha)$.
When $i$ is an entry of $\tau(\tH_j)$ in $\tau$, we simply write $i \in \tau(\tH_j)$.
With this notation, $\rtE$ can be obtained by placing $\tauE(\tH_j)$ vertically in the order $j = 1,2,\ldots,m+1$. 

For $T \in \SRT(\bal_E)$, let $\tau_T$ be the filling of $\tcd(\alpha)$ such that $\tH_j$ is filled with the entries of the $j$th column of $T$ in the decreasing order beginning from $j = 1$ and ending at $j = m+1$.
Define a $\mathbb{C}$-linear map $\eta: \bbfP_{\bal_E} \to \bfS^\sigma_{\alpha,E}$ by letting
\begin{align*}
 \quad T \mapsto 
\begin{cases}
\tau_T & \text{if it is contained in $E$,}\\
0 & \text{otherwise}
\end{cases}
\end{align*}
for $T \in \SRT(\bal_E)$ and extending it by linearity.
For the definition of $\bbfP_{\bal_E}$, see~\eqref{eq: def of bbfP action}.

\begin{example}\label{Two sources E and F}
Let
\begin{equation*}
\begin{array}{l}
\tauE =
\end{array}
\begin{array}{l}
\begin{ytableau}
{\color{red} 1}  \\
{\color{red} 7} & 6 & 5 & 4 \\
9 & 8 & {\color{red} 3} & 2 
\end{ytableau}
\end{array}
\begin{array}{l}
\quad
\text{and}
\quad
\tau =
\end{array}
\begin{array}{l}
\begin{ytableau}
4 \\
7 & 6 & 5 & 2 \\
9 & 8 & 3 & 1
\end{ytableau}
\end{array}.
\end{equation*}
In the left tableau, the entries in red denote descents.
Note that 
$d_0+1 = {\color{red} 1}$ and $d_2={\color{red} 3}$ are nonattacking,
$d_1+1=2$ is right of $d_3={\color{red} 7}$,
and $d_2+1=4$ is right of $d_4 = 9$.
So we have $\bal_E = (1) \oplus (1,2,1,1,2,1)$.
The following figures illustrate $\tH_j$, $\tauE(\tH_j)$, and $\tau(\tH_j)$ for $j = 1,2,3,4$.
\begin{displaymath}
\begin{tikzpicture}
\def\hhh{4mm}
\def\vvv{5mm}
\def\www{0.20mm}
\draw[line width=\www, fill=red!20] (\hhh*0,\vvv*0) rectangle (\hhh*1,\vvv*1);
\draw[line width=\www, fill=gray!20] (\hhh*0,-\vvv*1) rectangle (\hhh*4,\vvv*0);
\draw[line width=\www,dotted] (\hhh*1,-\vvv*1) -- (\hhh*1,\vvv*0);
\draw[line width=\www,dotted] (\hhh*2,-\vvv*1) -- (\hhh*2,\vvv*0);
\draw[line width=\www,dotted] (\hhh*3,-\vvv*1) -- (\hhh*3,\vvv*0);
\draw[line width=\www, fill=magenta!30] (\hhh*0,-\vvv*2) rectangle (\hhh*2,-\vvv*1);
\draw[line width=\www,dotted] (\hhh*1,-\vvv*2) -- (\hhh*1,-\vvv*1);
\draw[line width=\www, fill=blue!30] (\hhh*2,-\vvv*2) rectangle (\hhh*4,-\vvv*1);
\draw[line width=\www,dotted] (\hhh*3,-\vvv*2) -- (\hhh*3,-\vvv*1);
\node[] at (\hhh*0.6,\vvv*0.5) {\small $\tH_1$};
\node[] at (\hhh*2,-\vvv*0.5) {\small $\tH_3$};
\node[] at (\hhh*1,-\vvv*1.5) {\small $\tH_4$};
\node[] at (\hhh*3,-\vvv*1.5) {\small $\tH_2$};
\end{tikzpicture}
\hskip 10mm
\begin{tikzpicture}
\def\hhh{4mm}
\def\vvv{5mm}
\def\www{0.20mm}
\draw[line width=\www, fill=red!20] (\hhh*0,\vvv*0) rectangle (\hhh*1,\vvv*1);
\draw[line width=\www, fill=gray!20] (\hhh*0,-\vvv*1) rectangle (\hhh*4,\vvv*0);
\draw[line width=\www,dotted] (\hhh*1,-\vvv*1) -- (\hhh*1,\vvv*0);
\draw[line width=\www,dotted] (\hhh*2,-\vvv*1) -- (\hhh*2,\vvv*0);
\draw[line width=\www,dotted] (\hhh*3,-\vvv*1) -- (\hhh*3,\vvv*0);
\draw[line width=\www, fill=magenta!30] (\hhh*0,-\vvv*2) rectangle (\hhh*2,-\vvv*1);
\draw[line width=\www,dotted] (\hhh*1,-\vvv*2) -- (\hhh*1,-\vvv*1);
\draw[line width=\www, fill=blue!30] (\hhh*2,-\vvv*2) rectangle (\hhh*4,-\vvv*1);
\draw[line width=\www,dotted] (\hhh*3,-\vvv*2) -- (\hhh*3,-\vvv*1);
\node[left] at (-\hhh*0.3,-\vvv*0.5) {$\tauE=$};
\node[] at (\hhh*0.5,\vvv*0.5) {$1$};
\node[] at (\hhh*0.5,-\vvv*0.5) {$7$};
\node[] at (\hhh*1.5,-\vvv*0.5) {$6$};
\node[] at (\hhh*2.5,-\vvv*0.5) {$5$};
\node[] at (\hhh*3.5,-\vvv*0.5) {$4$};
\node[] at (\hhh*0.5,-\vvv*1.5) {$9$};
\node[] at (\hhh*1.5,-\vvv*1.5) {$8$};
\node[] at (\hhh*2.5,-\vvv*1.5) {$3$};
\node[] at (\hhh*3.5,-\vvv*1.5) {$2$};
\end{tikzpicture}
\hskip 10mm
\begin{tikzpicture}
\def\hhh{4mm}
\def\vvv{5mm}
\def\www{0.20mm}
\draw[line width=\www, fill=red!20] (\hhh*0,\vvv*0) rectangle (\hhh*1,\vvv*1);
\draw[line width=\www, fill=gray!20] (\hhh*0,-\vvv*1) rectangle (\hhh*4,\vvv*0);
\draw[line width=\www,dotted] (\hhh*1,-\vvv*1) -- (\hhh*1,\vvv*0);
\draw[line width=\www,dotted] (\hhh*2,-\vvv*1) -- (\hhh*2,\vvv*0);
\draw[line width=\www,dotted] (\hhh*3,-\vvv*1) -- (\hhh*3,\vvv*0);
\draw[line width=\www, fill=magenta!30] (\hhh*0,-\vvv*2) rectangle (\hhh*2,-\vvv*1);
\draw[line width=\www,dotted] (\hhh*1,-\vvv*2) -- (\hhh*1,-\vvv*1);
\draw[line width=\www, fill=blue!30] (\hhh*2,-\vvv*2) rectangle (\hhh*4,-\vvv*1);
\draw[line width=\www,dotted] (\hhh*3,-\vvv*2) -- (\hhh*3,-\vvv*1);
\node[left] at (-\hhh*0.3,-\vvv*0.5) {$\tau=$};
\node[] at (\hhh*0.5,\vvv*0.5) {$4$};
\node[] at (\hhh*0.5,-\vvv*0.5) {$7$};
\node[] at (\hhh*1.5,-\vvv*0.5) {$6$};
\node[] at (\hhh*2.5,-\vvv*0.5) {$5$};
\node[] at (\hhh*3.5,-\vvv*0.5) {$2$};
\node[] at (\hhh*0.5,-\vvv*1.5) {$9$};
\node[] at (\hhh*1.5,-\vvv*1.5) {$8$};
\node[] at (\hhh*2.5,-\vvv*1.5) {$3$};
\node[] at (\hhh*3.5,-\vvv*1.5) {$1$};
\end{tikzpicture}
\end{displaymath}
For instance, 
$\tauE(\tH_3) = 
\begin{tikzpicture}[baseline=-4mm]
\def\hhh{4mm}
\def\vvv{5mm}
\def\www{0.20mm}
\draw[line width=\www, fill=gray!20] (\hhh*0,-\vvv*1) rectangle (\hhh*4,\vvv*0);
\draw[line width=\www,dotted] (\hhh*1,-\vvv*1) -- (\hhh*1,\vvv*0);
\draw[line width=\www,dotted] (\hhh*2,-\vvv*1) -- (\hhh*2,\vvv*0);
\draw[line width=\www,dotted] (\hhh*3,-\vvv*1) -- (\hhh*3,\vvv*0);
\node[] at (\hhh*0.5,-\vvv*0.5) {$7$};
\node[] at (\hhh*1.5,-\vvv*0.5) {$6$};
\node[] at (\hhh*2.5,-\vvv*0.5) {$5$};
\node[] at (\hhh*3.5,-\vvv*0.5) {$4$};
\end{tikzpicture}$ and \
$\tau(\tH_3) = 
\begin{tikzpicture}[baseline=-4mm]
\def\hhh{4mm}
\def\vvv{5mm}
\def\www{0.20mm}
\draw[line width=\www, fill=gray!20] (\hhh*0,-\vvv*1) rectangle (\hhh*4,\vvv*0);
\draw[line width=\www,dotted] (\hhh*1,-\vvv*1) -- (\hhh*1,\vvv*0);
\draw[line width=\www,dotted] (\hhh*2,-\vvv*1) -- (\hhh*2,\vvv*0);
\draw[line width=\www,dotted] (\hhh*3,-\vvv*1) -- (\hhh*3,\vvv*0);
\node[] at (\hhh*0.5,-\vvv*0.5) {$7$};
\node[] at (\hhh*1.5,-\vvv*0.5) {$6$};
\node[] at (\hhh*2.5,-\vvv*0.5) {$5$};
\node[] at (\hhh*3.5,-\vvv*0.5) {$2$};
\end{tikzpicture}$.
Moreover, the ribbon diagram of $\bal_E$ is drawn as follows:
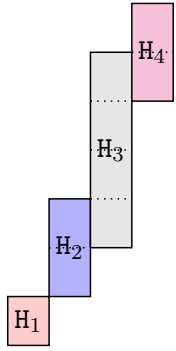
\begin{displaymath}
\begin{tikzpicture}
\def\hhh{4mm}
\def\vvv{5mm}
\def\www{0.20mm}
\draw[line width=\www,fill=red!20] (\hhh*0,\vvv*0) rectangle (\hhh*1,\vvv*1);
\draw[line width=\www,fill=blue!30] (\hhh*1,\vvv*1) rectangle (\hhh*2,\vvv*3);
\draw[line width=\www,dotted] (\hhh*1,\vvv*2) -- (\hhh*2,\vvv*2);
\draw[line width=\www,fill=gray!20] (\hhh*2,\vvv*2) rectangle (\hhh*3,\vvv*6);
\draw[line width=\www,dotted] (\hhh*2,\vvv*3) -- (\hhh*3,\vvv*3);
\draw[line width=\www,dotted] (\hhh*2,\vvv*4) -- (\hhh*3,\vvv*4);
\draw[line width=\www,dotted] (\hhh*2,\vvv*5) -- (\hhh*3,\vvv*5);
\draw[line width=\www,fill=magenta!30] (\hhh*3,\vvv*5) rectangle (\hhh*4,\vvv*7);
\draw[line width=\www,dotted] (\hhh*3,\vvv*6) -- (\hhh*4,\vvv*6);
\node[] at (\hhh*0.5,\vvv*0.5) {\small $\tH_1$};
\node[] at (\hhh*1.5,\vvv*2) {\small $\tH_2$};
\node[] at (\hhh*2.5,\vvv*4) {\small $\tH_3$};
\node[] at (\hhh*3.5,\vvv*6) {\small $\tH_4$};
\end{tikzpicture}
\end{displaymath}
Consider the following SRTs of shape $\bal_E$: 
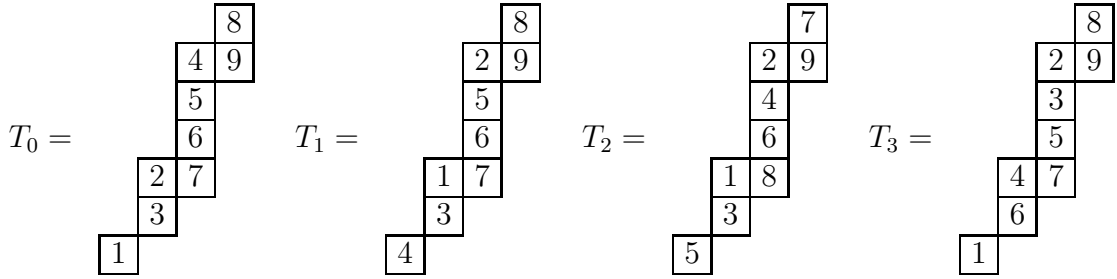
\begin{displaymath}
\begin{array}{l}
T_0 = 
\end{array}
\begin{array}{l}
\begin{ytableau}
\none & \none & \none & 8 \\
\none & \none & 4 & 9 \\
\none & \none & 5 \\
\none & \none & 6 \\
\none & 2 & 7 \\
\none & 3  \\
1
\end{ytableau}
\end{array}
\begin{array}{l}
\hskip 1mm
T_1 =
\end{array}
\begin{array}{l}
\begin{ytableau}
\none & \none & \none & 8 \\
\none & \none & 2 & 9 \\
\none & \none & 5 \\
\none & \none & 6 \\
\none & 1 & 7 \\
\none & 3  \\
4
\end{ytableau}
\end{array}
\begin{array}{l}
\hskip 1mm
T_2 =
\end{array}
\begin{array}{l}
\begin{ytableau}
\none & \none & \none & 7 \\
\none & \none & 2 & 9 \\
\none & \none & 4 \\
\none & \none & 6 \\
\none & 1 & 8 \\
\none & 3  \\
5
\end{ytableau}
\end{array}
\begin{array}{l}
\hskip 1mm
T_3 =
\end{array}
\begin{array}{l}
\begin{ytableau}
\none & \none & \none & 8 \\
\none & \none & 2 & 9 \\
\none & \none & 3 \\
\none & \none & 5 \\
\none & 4 & 7 \\
\none & 6  \\
1
\end{ytableau}
\end{array}
\end{displaymath}
Observe that $\tau_{T_0} = \tauE$, $\tau_{T_1} = \tau$, 
\begin{displaymath}
\tau_{T_2} =
\begin{array}{l}
\begin{ytableau}
5 \\
*(black!10) 8 & *(black!10) 6 & 4 & 2 \\
9 & *(black!10) 7 & 3 & 1 
\end{ytableau}
\end{array}
\quad
\text{and}
\quad
\tau_{T_3} =
\begin{array}{l}
\begin{ytableau}
1 \\
7 & 5 & 3 & 2 \\
9 & 8 & 6 & 4 
\end{ytableau}
\end{array}.
\hspace*{10mm}
\end{displaymath}
Due to $\tauE, \tau \in E$, we have that $\eta(T_0) = \tauE$ and $\eta(T_1) = \tau$.
On the contrary, we see that $\eta(T_2)=\eta(T_3)=0$ since 
$\tau_{T_2}$ is not an SPCT and $\tau_{T_2}$ is an SPCT not in $E$.
\end{example}
\vskip 3mm

With the above notation and hypothesis, we can state the following theorem.

\begin{theorem}\label{Thm:Surjective homo}
The map $\eta: \overline{\bfP}_{\bal_E} \to \bfS^\sigma_{\alpha,E}$
is a surjective $H_n(0)$-module homomorphism.
\end{theorem}

We now collect lemmas which are required to prove Theorem~\ref{Thm:Surjective homo}.

\begin{lemma}\label{Lemma:Surjective}
The map 
$\eta : \overline{\bfP}_{\bal_E} \to \bfS^\sigma_{\alpha,E}$
is a surjective $\mathbb{C}$-linear map.
\end{lemma}
\begin{proof}
Recall that $\Des(\tauE) = \{ d_1 < d_2 < \cdots < d_m \}$ and $d_0=0$, $d_{m+1}=n$.
For $\tau \in E$,
we fill the $j$th column of $\trd(\bal_E)$ with the entries of $\tau(\tH_j)$ so that they are increasing from top to bottom
for all $1\leq j \leq m+1$.
Denote by $T_\tau$ the resulting filling.
Let us choose an arbitrary $1 \leq j \leq m$, which will be fixed.
Let 
\begin{align*}
x &= \text{the entry at the uppermost box in column $j$ of $T_\tau$,} \\
y &= \text{the entry at the lowermost box in column $j+1$ of $T_\tau$.}
\end{align*}
We claim that $T_\tau$ is an SRT.
To do this, it suffices to show that $x<y$ only in the case where column $j$ and column $j+1$ are connected.
We have two cases.

{\it Case 1: $\sfc(d_{j-1}+1) \geq \sfc(d_{j+1})$}.
Note that $x$ appears at $(\sfr(d_j),\sfc(d_{j-1})+1)$ and 
$y$ at $(\sfr(d_{j+1}),\sfc(d_{j+1}))$ in $\tau$.
Since $\tH_j$ and $\tH_{j+1}$ share a column,
more precisely, there exists $k$ such that 
$(\sfr(d_j),k) \in \tH_j$ and $(\sfr(d_{j+1}),k) \in \tH_{j+1}$,
our claim follows from the inequalities
$$y \geq \tau_{\sfr(d_{j+1}),k} > \tau_{\sfr(d_{j}),k} \geq x.$$

{\it Case 2: $\sfc(d_{j-1}+1)+1 = \sfc(d_{j+1})$ and $\sfr(d_j) < \sfr(d_{j+1})$}.
Note that $x$ appears at $(\sfr(d_j),\sfc(d_{j+1})-1)$ and 
$y$ at $(\sfr(d_{j+1}),\sfc(d_{j+1}))$ in $\tau$.
If $(\sfr(d_j),\sfc(d_{j+1})) \in \tcd(\alpha)$, then 
the entry at this box in $\tauE$ is less than the entry at $(\sfr(d_{j+1}),\sfc(d_{j+1}))$ in $\tauE$.
Since $\tauE, \tau$ are in the same class $E$,
the entry at $(\sfr(d_j),\sfc(d_{j+1}))$ in $\tau$ is less than the entry at $(\sfr(d_{j+1}),\sfc(d_{j+1}))$ in $\tau$.
From the triple condition it follows that $x < y$
even in the case where $(\sfr(d_j),\sfc(d_{j+1})) \notin \tcd(\alpha)$.

The above discussion shows that $T_\tau \in \SRT(\bal_E)$.
It is straightforward from the definition of $\eta : \overline{\bfP}_{\bal_E} \to \bfS^\sigma_{\alpha,E}$ that $\eta(T_\tau) = \tau$.
\end{proof}

The following lemma plays an essential role in proving that $\eta: \overline{\bfP}_{\bal_E} \to \bfS^\sigma_{\alpha,E}$ is an $H_n(0)$-module homomorphism.

\begin{lemma}\label{lem: des to des}
Given $\tau \in E$, 
suppose that $i$ appears in $\tau(\tH_j)$ and $i+1$ in $\tau(\tH_k)$.
Then $i$ is a descent of $\tau$ if and only if $j<k$.
\end{lemma}
\begin{proof}
We first prove the `only if' part.
Note that $\tau = \pi_\sigma \cdot \tauE$ for some $\sigma \in \SG_n$.

We use induction on $\ell(\sigma)$ to show the assertion.
In case where $\ell(\sigma) = 0$, we have $\tau= \tauE$.
By the definition of $\tauE$, 
if $i$ is a descent of $\tauE$, then $i+1$ appears in $\tauE(\tH_{j+1})$ as desired.
In case where $\ell(\sigma) = 1$, $\sigma = s_p$ for some $1 \leq p \leq n-1$.
Observe that $\tau$ and $\tauE$ are same except for $p$ and $p+1$, so we have only to check the cases where $i = p-1,p,p+1$. 
It is clear that $p$ is not a descent of $\tau$.
Suppose that $i = p-1$ is a descent of $\tau$.
Note that $p+1$ is in $\tauE(\tH_{k})$ since $p$ is in $\tau(\tH_{k})$.
Moreover, since $p$ is a descent in $\tauE$, it follows that $p$ is in $\tauE(\tH_{k-1})$.
Combining $p-1 \in \tauE(\tH_{j})$ with $p \in \tauE(\tH_{k-1})$, we deduce that $j \le k-1$, thus $j < k$.
We can deal with the case where $i = p+1$ in the same manner as above.

We now assume that the assertion is true when $\ell(\sigma)$ is less than $l$ ($l \ge 2$).
Let $\tau \in E$ with $\pi_{\sigma} \cdot \tauE$ with $\ell(\sigma) = l$
and $s_{p_l} s_{p_{l-1}} \cdots s_{p_1}$ a reduced expression for $\sigma$.
Let $\tau' = \pi_{\sigma'} \cdot \tauE$ with $\sigma' = s_{p_l} \sigma  =  s_{p_{l-1}} \cdots s_{p_1}$.
Considering that $\tau$ is identical to $\tau'$ except for $p_l, p_l+1$ and $p_l$ is not a descent of $\tau$,
we need to check the cases where $i = p_l-1, p_l+1$.
Since the latter case can be verified in the same manner as in the former case,
we only deal with the case where $i = p_l-1$ is a descent of $\tau$.

Let $p_l+1 \in \tau(\tH_q)$ for some $1 \le q \le m+1$.
Then 
\[
p_l-1  \in \tau'(\tH_j), \quad p_l  \in \tau'(\tH_q), \quad \text{and} \quad p_l+1  \in \tau'(\tH_k).
\]
Since $p_l$ is a descent of $\tau'$, the induction hypothesis implies that 
$q < k$. 
When $p_l-1$ is a descent of $\tau'$, the induction hypothesis again implies that $j < q$, thus $j < k$.
We next deal with the case where $p_l-1$ is not a descent of $\tau'$.
Note that 
$p_l-1$ and $p_l$ are not in the same column in $\tau'$.
Moreover, due to the triple condition, it does not occur that 
$p_l-1$ and $p_l$ are in adjacent columns in $\tau'$,
with $p_l-1$ positioned strictly lower-right of $p_l$.
Considering $\tau'':= s_{p_l-1} \cdot \tau'$,
one sees that $\tau'' \in E$ and $\tau'= \pi_{p_l-1} \cdot \tau''$. 
By virtue of~\cite[Lemma A.2]{20CKNO}, 
one can write $\tau''$ as $\pi_{\sigma''} \cdot \tauE$ for some $\sigma''\in \SG_n$ of length $l-2$. 
Then 
\[
p_l-1  \in \tau''(\tH_q),
\quad
p_l  \in \tau''(\tH_j),
\quad\text{and}\quad
p_l+1  \in \tau''(\tH_k).
\]
Note that 
$\sfc_{\tau''}(p_l) = \sfc_{\tau'}(p_l-1) = \sfc_{\tau}(p_l-1)$
and 
$\sfc_{\tau''}(p_l+1) = \sfc_{\tau'}(p_l+1) = \sfc_{\tau}(p_l)$.
From the assumption that $p_l-1$ is a descent of $\tau$
it follows that $\sfc_{\tau''}(p_l) \leq \sfc_{\tau''}(p_l+1)$,
which again implies that $p_l$ is a descent of $\tau''$.
Hence, the inequality $j<k$ is straightforward from the induction hypothesis.

Next, we prove the `if' part. 
Suppose that $j<k$, but $i$ is not a descent of $\tau$. 
If $i$ and $i+1$ are in the same row, then $i+1$ lies to the immediate left of $i$, which says that
$i$ is in the leftmost box of $\tH_{j}$ and $i+1$ in the rightmost box of $\tH_{k}$.
For each $1 \le t \le m$
the leftmost box of $\tH_t$ is located left of the rightmost box of $\tH_{t+1}$,
thus there exists $\tH_{t_0}$ with $j < t_0 < k$ 
which shares a column with $\tH_{j}$ and also a column with $\tH_{k}$.
Note that the largest entry in $\tH_{j}(\tauE)$ is less than the smallest entry in $\tH_{t_0}(\tauE)$ and 
the largest entry in $\tH_{t_0}(\tauE)$ is less than the smallest entry in $\tH_{k}(\tauE)$.
Since $\tau$ and $\tauE$ are in the same class $E$, there exists an entry $s$ in $\tH_{t_0}(\tau)$ such that $i < s < i+1$, which is absurd.
Thus, we can deduce that $i+1$ is strictly above $i$ or strictly below $i$.
By the triple condition, $i+1$ cannot appear strictly above and immediate left of $i$.
Therefore $\tau' := s_i \cdot \tau$ is an SPCT. It is clear that $i$ is a descent of $\tau'$, $i \in \tau'(\tH_k)$, and $i+1 \in \tau'(\tH_j)$. Thus by the ``only if'' part, we have $k < j$, which contradicts the assumption $j < k$.
Hence, we conclude that $i$ is a descent of $\tau$.
\end{proof}

\begin{proof}[Proof of Theorem~\ref{Thm:Surjective homo}]

The surjectivity of $\eta$ is verified in Lemma~\ref{Lemma:Surjective}.
For the completion of the proof, it suffices to show that $$
\eta(\pi_i \star T) = \pi_i \cdot \eta(T) 
$$ 
for $1 \leq i \leq n-1$ and $T \in \SRT(\bal_E)$.
Let $i \in \tau_T(\tH_j)$ and $i+1 \in \tau_T(\tH_k)$.

{\it Case 1: $\pi_i \star T = T$.}
In this case $i$ is strictly above $i+1$ in $T$, thus $j \geq k$.
When $\eta(T) = 0$, there is nothing to prove.
When $\eta(T) \neq 0$, the equality $\pi_i \cdot \eta(T) = \eta(T)$ follows from Lemma~\ref{lem: des to des}.

{\it Case 2: $\pi_i \star T = 0$.}
In this case $i$ and $i+1$ in the same row of $T$, thus $k = j+1$.
When $\eta(T) = 0$, there is nothing to prove.
Suppose that $\eta(T) \neq 0$.
Notice that $i$ appears at the rightmost box of $\tau_T(\tH_j)$, $i+1$ appears at the leftmost box of $\tau_T(\tH_{j+1})$, and column $j$ is connected to column $j+1$ in $T$.
If $\tH_{j}$ and $\tH_{j+1}$ share a column, then 
$i$ and $i+1$ are in the same column in $\tau_T$.
It means that $i$ is an attacking descent of $\tau_T$.
Otherwise, $\tH_{j}$ is strictly above $\tH_{j+1}$ 
and $i, i+1$ are in the adjacent columns in $\tau_T$.
It means that $i$ is an attacking descent of $\tau_T$.
Consequently $\pi_i \cdot \eta(T) = 0$.

{\it Case 3: $\pi_i \star T = s_i \cdot T$.}
In this case $i$ is strictly below $i+1$ in $T$, thus $j <k$.
This tells us that $i$ is strictly right of $i+1$ in $s_i \cdot T$.

We first deal with the case where $\eta(T) = 0$.
Suppose that $\eta(s_i \cdot T) \neq 0$.
Due to Lemma~\ref{lem: des to des}, we have that $i \not\in \Des(\tau_{s_i \cdot T})$.
If $i$ and $i+1$ are in the same row of $\tau_{s_i \cdot T}$, 
then it follows from the definition of $\eta$ that
$i$ lies weakly left of $i+1$ in $s_i \cdot T$, which is absurd.
This says that $i$ and $i+1$ are not in the same row of $\tau_{s_i \cdot T}$.
Therefore the filling $s_i \cdot \tau_{s_i \cdot T}$, which is obtained from $\tau_{s_i \cdot T}$ by swapping $i$ and $i+1$, still remains inside $E$.
This contradicts $\eta(T) = 0$ since $s_i \cdot \tau_{s_i \cdot T} = \tau_T$.
As a consequence, we can conclude that $\eta(s_i \cdot T) = 0$.

Next, we deal with the case where $\eta(T) \neq 0$.
Due to Lemma~\ref{lem: des to des}, we have that $i \in \Des(\tau_T)$.
When $i$ is a nonattacking descent, it holds that $\pi_i \cdot \eta(T) = s_i \cdot \tau_T = \tau_{s_i \cdot T} = \eta(s_i \cdot T)$.
When $i$ is an attacking descent, $\pi_i \cdot \eta(T) = 0$.
If $i$ and $i+1$ are in the same column of $\tau_T$,
then $\tau_{s_i\cdot T} \not\in E$ since $\tau_{s_i \cdot T} = s_i \cdot \tau_T$.
Otherwise, 
$\tau_{s_i\cdot T}$ is not an SPCT since the triple condition breaks at $i$, $i+1$ and the entry immediate right of $i+1$.
In both cases, we have $\eta(s_i\cdot T) = 0$.
\end{proof}

Theorem~\ref{Thm:Surjective homo} shows that $\bbfP_{\bal_E}$ is a candidate for the projective cover of $\bfS^\sigma_{\alpha,E}$.
To achieve the goal of this section, we have to verify that $\eta: \bbfP_{\bal_E} \to \bfS^\sigma_{\alpha,E}$ is an essential epimorphism.

\begin{lemma}{\rm (\cite[Proposition 3.6]{95ARS})}\label{lem: ess epi}
The following are equivalent for an epimorphism $f:A \ra B$, where $A$ and $B$ are finitely generated modules over a left artin ring.
\begin{enumerate}[label = {\rm (\alph*)}]
\item $f$ is an essential epimorphism.
\item $\ker(f) \subset \rad(A)$.
\end{enumerate}
\end{lemma}

One can describe $\ker(\eta)$ explicitly.
Let $T \in \SRT(\bal_E)$.
Recall that $T^{q}_{p}$ denotes the entry at the $q$th box \emph{from the top} of the $p$th column from the left in $T$ (see Subsection~\ref{Sec3.1}).
Similarly, we denote by $T^{-q}_{p}$ the entry at the $q$th box \emph{from the bottom} of the $p$th column from the left in $T$.
Letting $c_p(T)$ be the size of the $p$th column of $T$,
it is clear that $T^{-q}_{p} = T^{c_p(T) - q +1}_{p}$.
Let $(x_{p}, y_{p,q})$ be the box in $\tau_T$ corresponding to $T^{-q}_p$ for $1 \leq q \leq c_p(T)$ via $\eta$.
From the definition of $\eta$ we have the following equality: 
\begin{align}\label{eq: x and y def}
y_{p,q} = q +  \sum_{\substack{p < p' \\  x_p = x_{p'}}} c_{p'}(T).
\end{align}
Define $\Theta_{\alpha,E}$ to be the subset of $\SRT(\bal_E)$ 
consisting of $T$'s such that there exists a quadruple $(i,j,s,t)$ satisfying
\begin{equation}\label{Eq: Description of kernel}
\left\{
\begin{aligned}
\text{\bf L1. } & 
i<j,~T_i^{-s} > T_j^{-t},~ \text{and}~y_{i,s}  = y_{j,t},~\text{or}\\
\text{\bf L2. } &
s=c_i(T),~
T_i^{-s} > T_j^{-t}, ~
x_i < x_j, \text{ and } \alpha_{x_i}
= y_{i,s} = y_{j,t} -1 ~\text{or}
\\
\text{\bf L3. } &
s=c_i(T),~
T_{i}^{-s} >  T_j^{-t} > T_{i^-}^{-1},~ x_i < x_j, \text{ and } \alpha_{x_i} > y_{i,s} = y_{j,t}-1, 
\text{ or }\\
\text{\bf L4. } &
s < c_i(T),~ 
T_i^{-s} > T_j^{-t} > T_i^{-(s+1)},~ x_i < x_j, \text{ and }  y_{i,s} = y_{j,t}-1.
\end{aligned}
\right.
\end{equation}
In {\bf L3}, the subscript $i^-$ denotes the column index $<i$ of $T$
uniquely determined by the conditions $x_{i} = x_{i^-}$ and $y_{i^-,1} =y_{i,s}+1$.

\begin{lemma}\label{lem: kernel of eta}
$\ker(\eta)$ is equal to the $\C$-span of $\Theta_{\alpha,E}$.
\end{lemma}

\begin{proof}
Let $T \in \Theta_{\alpha,E}$.
Assume that there exists a quadruple $(i,j,s,t)$ satisfying the condition ${\bf L1}$.
Then the entry at $(x_i,y_{i,s})$ is greater than that at $(x_j,y_{j,t})$ in $\tau_T$.
One can also see that the entry at $(x_i,y_{i,s})$ is less than that at $(x_j,y_{j,t})$ in $\tau_{\rtE}$
since $(\rtE)_i^{-s} < (\rtE)_j^{-t}$.
This implies that $\tau_T$ does not belong to $E$.
For the condition {\bf L1} in $\tau_T$, 
see below.
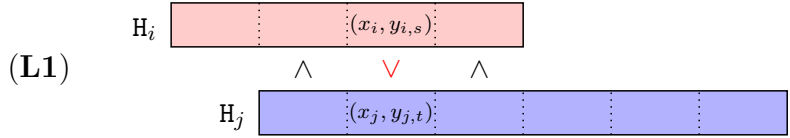
\begin{displaymath}
\begin{tikzpicture}[scale=0.8]
\def\hhh{13mm}
\def\vvv{6mm}
\def\www{0.20mm}
\draw[line width=\www, fill=red!20] (-\hhh*3,\vvv*0) rectangle (\hhh*1,\vvv*1);
\draw[line width=\www,dotted] (-\hhh*2,\vvv*1) -- (-\hhh*2,\vvv*0);
\draw[line width=\www,dotted] (-\hhh*1,\vvv*1) -- (-\hhh*1,\vvv*0);
\draw[line width=\www,dotted] (\hhh*0,\vvv*1) -- (\hhh*0,\vvv*0);
\draw[line width=\www, fill=blue!30] (-\hhh*2,-\vvv*2) rectangle (\hhh*4,-\vvv*1);
\draw[line width=\www,dotted] (-\hhh*1,-\vvv*2) -- (-\hhh*1,-\vvv*1);
\draw[line width=\www,dotted] (-\hhh*0,-\vvv*2) -- (-\hhh*0,-\vvv*1);
\draw[line width=\www,dotted] (\hhh*1,-\vvv*2) -- (\hhh*1,-\vvv*1);
\draw[line width=\www,dotted] (\hhh*2,-\vvv*2) -- (\hhh*2,-\vvv*1);
\draw[line width=\www,dotted] (\hhh*3,-\vvv*2) -- (\hhh*3,-\vvv*1);
\node[] at (-\hhh*3.3,\vvv*0.4) {\small $\tH_i$};
\node[] at (-\hhh*2.3,-\vvv*1.6) {\small $\tH_j$};
\node[] at (-\hhh*0.5,\vvv*0.5) {\tiny $(x_i,y_{i,s})$};
\node[] at (-\hhh*0.5,-\vvv*1.5) {\tiny $(x_j,y_{j,t})$};
\node[] at (-\hhh*1.5,-\vvv*0.5) {\small $\downineq$};
\node[] at (-\hhh*0.5,-\vvv*0.5) {\small \color{red} $\upineq$};
\node[] at (\hhh*0.5,-\vvv*0.5) {\small  $\downineq$};
\node[] at (-\hhh*4.5,-\vvv*0.5) {({\bf L1})};
\end{tikzpicture}
\end{displaymath}
Next, 
assume that there exists a quadruple $(i,j,s,t)$ satisfying the condition ${\bf L2}$.
Then $\tau_T$ is not an SPCT since
$s=c_i(T)$
and 
the entry at $(x_i,y_{i,s})$ is greater than that at $(x_j,y_{j,t})$ in $\tau_T$ but $(x_i,y_{i,s}+1) \notin \tcd(\alpha)$.
For the condition {\bf L2} in $\tau_T$, 
see below.
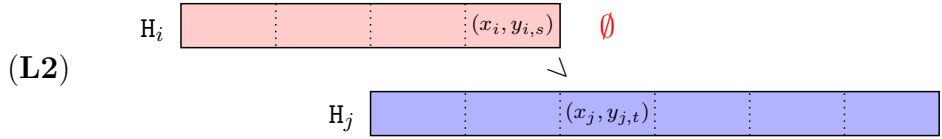
\begin{displaymath}
\begin{tikzpicture}[scale=0.8]
\def\hhh{14mm}
\def\vvv{6mm}
\def\www{0.20mm}
\def\hhhh{80mm}
\draw[line width=\www, fill=red!20] (-\hhh*3,\vvv*0) rectangle (\hhh*1,\vvv*1);
\draw[line width=\www,dotted] (-\hhh*2,\vvv*1) -- (-\hhh*2,\vvv*0);
\draw[line width=\www,dotted] (-\hhh*1,\vvv*1) -- (-\hhh*1,\vvv*0);
\draw[line width=\www,dotted] (\hhh*0,\vvv*1) -- (\hhh*0,\vvv*0);
\draw[line width=\www, fill=blue!30] (-\hhh*1,-\vvv*2) rectangle (\hhh*5,-\vvv*1);
\draw[line width=\www,dotted] (-\hhh*0,-\vvv*2) -- (-\hhh*0,-\vvv*1);
\draw[line width=\www,dotted] (\hhh*1,-\vvv*2) -- (\hhh*1,-\vvv*1);
\draw[line width=\www,dotted] (\hhh*2,-\vvv*2) -- (\hhh*2,-\vvv*1);
\draw[line width=\www,dotted] (\hhh*3,-\vvv*2) -- (\hhh*3,-\vvv*1);
\draw[line width=\www,dotted] (\hhh*4,-\vvv*2) -- (\hhh*4,-\vvv*1);
\node[] at (-\hhh*3.3,\vvv*0.4) {\small $\tH_i$};
\node[] at (-\hhh*1.3,-\vvv*1.6) {\small $\tH_j$};
\node[] at (\hhh*1.5,\vvv*0.5) {\color{red} $\emptyset$};
\node[] at (\hhh*0.5,\vvv*0.5) {\tiny $(x_i,y_{i,s})$};
\node[] at (\hhh*1.5,-\vvv*1.5) {\tiny $(x_j,y_{j,t})$};
\node[] at (\hhh*1,-\vvv*0.5) { $\diagineq$};
\node[] at (-\hhh*4.5,-\vvv*0.5) {({\bf L2})};
\end{tikzpicture}
\end{displaymath}
Finally, assume that 
there exists a quadruple $(i,j,s,t)$ satisfying the condition ${\bf L3}$ or ${\bf L4}$,
then the entries at $(x_{i},y_{i,s}), (x_{i^-},y_{i^-,1})$ and $(x_j,y_{j,t})$,
or 
$(x_i,y_{i,s}), (x_i,y_{i,s-1})$ and $(x_j,y_{j,t})$ in $\tau_T$ 
do not satisfy the triple condition.
For the condition {\bf L3} or {\bf L4} in $\tau_T$, 
see below.
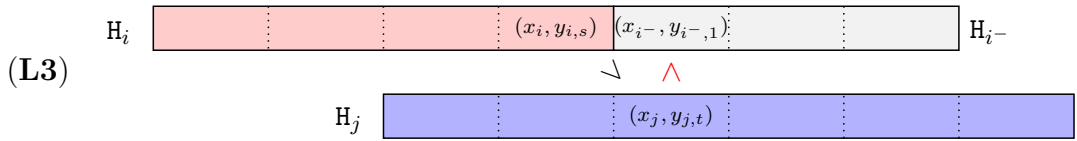
\begin{displaymath}
\begin{tikzpicture}[scale=0.8]
\def\hhh{17mm}
\def\vvv{6mm}
\def\www{0.20mm}
\def\hhhh{90mm}
\draw[line width=\www, fill=red!20] (-\hhh*3,\vvv*0) rectangle (\hhh*1,\vvv*1);
\draw[line width=\www,dotted] (-\hhh*2,\vvv*1) -- (-\hhh*2,\vvv*0);
\draw[line width=\www,dotted] (-\hhh*1,\vvv*1) -- (-\hhh*1,\vvv*0);
\draw[line width=\www,dotted] (\hhh*0,\vvv*1) -- (\hhh*0,\vvv*0);
\draw[line width=\www, fill=gray!10] (\hhh*1,\vvv*0) rectangle (\hhh*4,\vvv*1);
\draw[line width=\www,dotted] (\hhh*2,\vvv*1) -- (\hhh*2,\vvv*0);
\draw[line width=\www,dotted] (\hhh*3,\vvv*1) -- (\hhh*3,\vvv*0);
\draw[line width=\www, fill=blue!30] (-\hhh*1,-\vvv*2) rectangle (\hhh*5,-\vvv*1);
\draw[line width=\www,dotted] (-\hhh*0,-\vvv*2) -- (-\hhh*0,-\vvv*1);
\draw[line width=\www,dotted] (\hhh*1,-\vvv*2) -- (\hhh*1,-\vvv*1);
\draw[line width=\www,dotted] (\hhh*2,-\vvv*2) -- (\hhh*2,-\vvv*1);
\draw[line width=\www,dotted] (\hhh*3,-\vvv*2) -- (\hhh*3,-\vvv*1);
\draw[line width=\www,dotted] (\hhh*4,-\vvv*2) -- (\hhh*4,-\vvv*1);
\node[] at (-\hhh*3.3,\vvv*0.4) {\small $\tH_i$};
\node[right] at (\hhh*4,\vvv*0.4) {\small $\tH_{i^-}$};
\node[] at (-\hhh*1.3,-\vvv*1.6) {\small $\tH_j$};
\node[] at (\hhh*1.5,-\vvv*0.5) {\color{red} $\downineq$};
\node[] at (\hhh*0.5,\vvv*0.5) {\tiny $(x_i,y_{i,s})$};
\node[] at (\hhh*1.5,\vvv*0.5) {\tiny $(x_{i^-},y_{i^-,1})$};

\node[] at (\hhh*1.5,-\vvv*1.5) {\tiny $(x_j,y_{j,t})$};
\node[] at (\hhh*1,-\vvv*0.5) {$\diagineq$};
\node[] at (-\hhh*4,-\vvv*0.5) {({\bf L3})};
\end{tikzpicture}
\end{displaymath}
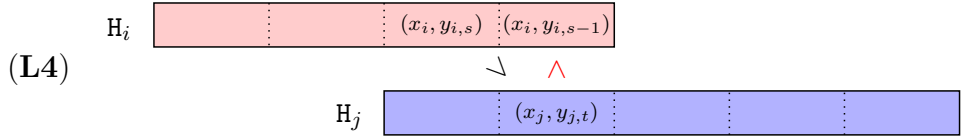
\begin{displaymath}
\begin{tikzpicture}[scale=0.8]
\def\hhh{17mm}
\def\vvv{6mm}
\def\www{0.20mm}
\draw[line width=\www, fill=red!20] (-\hhh*3,\vvv*0) rectangle (\hhh*1,\vvv*1);
\draw[line width=\www,dotted] (-\hhh*2,\vvv*1) -- (-\hhh*2,\vvv*0);
\draw[line width=\www,dotted] (-\hhh*1,\vvv*1) -- (-\hhh*1,\vvv*0);
\draw[line width=\www,dotted] (\hhh*0,\vvv*1) -- (\hhh*0,\vvv*0);
\draw[line width=\www, fill=blue!30] (-\hhh*1,-\vvv*2) rectangle (\hhh*4,-\vvv*1);
\draw[line width=\www,dotted] (-\hhh*0,-\vvv*2) -- (-\hhh*0,-\vvv*1);
\draw[line width=\www,dotted] (\hhh*1,-\vvv*2) -- (\hhh*1,-\vvv*1);
\draw[line width=\www,dotted] (\hhh*2,-\vvv*2) -- (\hhh*2,-\vvv*1);
\draw[line width=\www,dotted] (\hhh*3,-\vvv*2) -- (\hhh*3,-\vvv*1);
\node[] at (-\hhh*3.3,\vvv*0.4) {\small $\tH_i$};
\node[] at (-\hhh*1.3,-\vvv*1.6) {\small $\tH_j$};
\node[] at (-\hhh*0.5,\vvv*0.5) {\tiny $(x_i,y_{i,s})$};
\node[] at (\hhh*0.5,\vvv*0.5) {\tiny $(x_i,y_{i,s-1})$};
\node[] at (\hhh*0.5,-\vvv*1.5) {\tiny $(x_j,y_{j,t})$};
%
\node[] at (\hhh*0,-\vvv*0.5) { $\diagineq$};
\node[] at (\hhh*0.5,-\vvv*0.5) { \small \color{red}
$\downineq$};
\node[] at (-\hhh*4,-\vvv*0.5) {({\bf L4})};
\end{tikzpicture}
\end{displaymath}
In any cases, $\tau_T$ cannot be  an SPCT.
So $\Theta_{\alpha,E} \subseteq \ker(\eta)$.

On the other hand, if $T \in \SRT(\bal_E) \cap \ker(\eta)$, then
$\tau_T$ is either
an SPCT not in $E$ or a non-SPCT. 
One can see in a routine way that there exists a quadruple $(i,j,s,t)$ satisfying ${\bf L1}$ in the former case 
and one of conditions ${\bf L2}, {\bf L3}$ or ${\bf L4}$ in the latter case.
\end{proof}

Let $T$ be an SRT of shape $\bal_E$. 
In the following, We will give a sufficient condition for $T$ to be in $\rad(\bbfP_{\bal_E})$.
To do this, we introduce necessary definitions and notation.

\begin{definition}\hfill
\begin{enumerate}[label = {\rm (\arabic*)}]
\item Given a filling $T$ of a generalized ribbon diagram of size $n$, we define $T[i]$ ($1 \le i \le n$) to be the $i$th entry when we read $T$ from bottom to top and from left to right.
\item Let $\bal$ be a generalized composition of $n$. For each $\beta \in [\bal]$, let $T \in \SRT(\beta)$. We define $T^\bal$ to be the filling of $\trd(\bal)$ such that
$T^\bal[i] := T[i]$ for each $1 \le i \le n$.
\end{enumerate}
\end{definition}

For instance, let $\bal = (1) \oplus (1) \oplus (2,1,2,2)$
and $\beta = (1) \cdot (1) \odot (2,1,2,2) = (1,3,1,2,2)$.
Then  
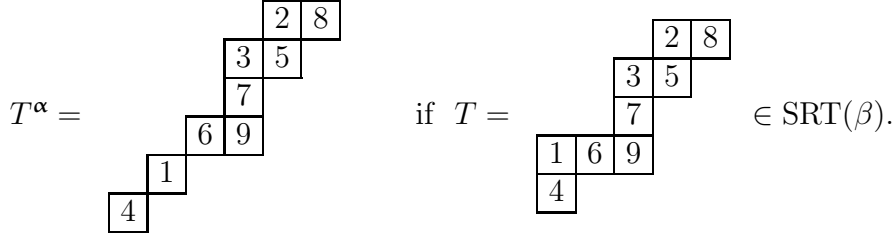
\begin{displaymath}
\begin{array}{l}
T^\bal = 
\end{array}
\begin{array}{l}
\begin{ytableau}
\none & \none & \none & \none & 2 & 8  \\
\none & \none & \none & 3 & 5 \\
\none & \none & \none & 7  \\
\none & \none &  6  & 9\\
\none &  1 \\
4 
\end{ytableau}
\end{array}
\qquad
\text{if}
\begin{array}{l}
\hskip 1mm
T =
\end{array}
\begin{array}{l}
\begin{ytableau}
\none & \none & \none & 2 & 8 \\
\none & \none & 3 & 5 \\
\none & \none & 7 \\
1 & 6 & 9  \\
4
\end{ytableau}
\end{array}
\in \SRT(\beta).
\end{displaymath}

From the construction it follows that $T^\bal \in \SRT(\bal)$ and the map 
$$
\kappa: \bigcup_{\beta \in [\bal]} \SRT(\beta) \ra  \SRT(\bal), \quad T \mapsto T^\bal
$$
is a bijection.

Let $\alpha$ and $\beta$ be compositions.
We say that $\alpha$ is a \emph{coarsening} of $\beta$, denoted by $\alpha \succeq \beta$, if we can obtain the parts of $\alpha$ in order by adding together adjacent parts of $\beta$ in order. 
We also denote by $\alpha \succdot \beta$ 
if $\alpha \succ \beta$ but there is no $\gamma$ such that $\alpha \succ \gamma$ and $\gamma \succ \beta$.
Then $([\bal], \preceq)$ is a partially order set for any generalized composition $\bal$
and it has a unique maximal element 
$
\bal^{\max} := \alpha^{(1)} \odot \alpha^{(2)} \odot \cdots \odot \alpha^{(k)}
$
and a unique minimal element
$
\bal^{\min} := \alpha^{(1)} \cdot \alpha^{(2)} \cdot \hskip1pt \cdots \hskip1pt \cdot \alpha^{(k)}.
$
It should be noticed that
$\bal^{\max}$ and $\bal^{\min}$ are not generalized compositions but compositions.

Let $T \in \SRT(\bal)$.
The descent set of $T$ is defined by
$$
\Des(T) := \{i \in [n-1] \mid i+1 \text{ appears strictly right of } i  \}.
$$
Recall that  $T_{0}^{(\beta)}$ is the SRT obtained by filling $\trd(\beta)$
with entries $1, 2, \ldots, n$ from top to bottom and from left to right
(see Subsection~\ref{subsec: PIM}).
Then one sees that
$$
\Des(T_{0}^{(\beta)}) = \{\beta^\rmc_1,\beta^\rmc_1 + \beta^\rmc_2, \ldots, \beta^\rmc_1 + \cdots +  \beta^\rmc_{m-1}\},
$$
where $m = \ell(\beta^\rmc)$.

\begin{lemma}\label{Lem: To be contained in Rad}
Let $T$ be an SRT of shape $\bal_E$. 
Suppose that $T = \pi_\sigma \cdot \rtE$ for some $\sigma \in \SG_n$
and $s_{i_l} \cdots s_{i_1}$ is a reduced expression for $\sigma$.
If $i_j \in \Des(T_0^{(\bal_E^{\min})})$ for some $1 \leq j \leq l$, then $T$ is contained in $\rad(\bbfP_{\bal_E})$.
\end{lemma}
\begin{proof}
By Theorem~\ref{thm: bfP isom to calP}(b) together with $\bfP_\alpha \cong \bbfP_{\alpha^\rmc}$,
we can choose an isomorphism 
$$
f:\bbfP_{\bal_E} \ra \bigoplus_{\beta \in [\bal_E]} \bbfP_{\beta}.
$$
Since $f(\rtE)$ is a generator of $\bigoplus_{\beta \in [\bal_E]} \bbfP_{\beta}$,
it should be written as 
$$
\sum_{\beta \in [\bal_E]} 
\left(c_\beta T^{(\beta)}_0 + \text{ lower terms in } \SRT(\beta)  \right).
$$
Here $c_\beta$ is nonzero for all $\beta \in [\bal_E]$.
Note that $\Des(T_{0}^{(\alpha)}) \supsetneq \Des(T_0^{(\alpha')})$ if $\alpha \succdot \alpha'$,
thus $\Des(T_{0}^{(\beta)}) \supsetneq \Des(T_0^{(\bal_E^{\min})})$
for all $\beta \in [\bal_E]$.
This implies that 
if $\pi_i \cdot (T_0^{(\bal_E^{\min})}) \neq (T_0^{(\bal_E^{\min})})$ for some $1 \leq i \leq n-1$,
then $\pi_i\cdot (T_0^{(\beta)}) \neq (T_0^{(\beta)})$ for all $\beta \in [\bal_E]$.
Due to this phenomenon, our assumption $i_j \in \Des(T_0^{\bal_E^{\min}})$ implies that  
none of $T_0^{(\beta)}$ terms appear in $f(T)$.
Since 
$$
f(\rad(\bbfP_{\bal_E})) =  \bigoplus_{\beta \in [\bal_E]} \rad(\bbfP_{\beta})
$$
and $\rad(\bbfP_{\beta})$ is the $\C$-span of $\SRT(\beta)\setminus\{T_0^{(\beta)}\}$,
we conclude that $T \in \rad(\bbfP_{\bal_E})$.
\end{proof}

Next, we provide special tableaux in $\SRT(\bal_E)$ which are outside of $\ker(\eta)$.

\begin{lemma}\label{Lem:T0 are SPCTs}
For all $\beta \in [\bal_E]$,
we have that
$\kappa(T_0^{(\beta)}) \notin \ker(\eta)$.
\end{lemma}
\begin{proof}
Let $\bal_E = \alpha^{(1)} \oplus \alpha^{(2)} \oplus \cdots \oplus \alpha^{(k)}$.
If $k=1$, then $[\bal_E] = \{\alpha^{(1)}\}$ and 
the image of $\kappa(T_0^{(\alpha^{(1)})})$ under $\eta$ is the source tableau $\tauE$,
thus the assertion is obviously true.
From now on, we assume that $k \geq 2$.
For any $\beta \in [\bal_E]$, it can be written as
$$
\alpha^{(1)} \  \square \  \alpha^{(2)} \ \square \ \cdots \  \square  \ \alpha^{(k)}
\quad 
\text{where }
\square \in \{\odot, \cdot\}.
$$
Let $n_\beta$ be the number of $\cdot$'s in this expression.
We will prove the assertion by induction on $n_\beta$.

When $n_\beta = 0$, equivalently, $\beta = \bal_E^{\max}$,
the image of $\kappa(T_0^{(\beta)})$ under $\eta$ is the source tableau $\tauE$.
So we are done.

Let $l$ be an arbitrary positive integer less than $k-1$.
Assume that the assertion is true for all $\beta \in [\bal_E]$ with $n_\beta<l$.
Pick up any $\gamma \in [\bal_E]$ with $n_\gamma = l$.
Then there exists $\beta \in [\bal_E]$ such that $\gamma \precdot \beta$, that is, $n_\gamma = n_\beta + 1$.
We claim that $\kappa(T_0^{(\gamma)})$ has no quadruples satisfying the conditions {\bf L1}, {\bf L2}, {\bf L3} and {\bf L4}
in~\eqref{Eq: Description of kernel}.
Let $p$ be the smallest positive integer such that $c_p(T_0^{(\gamma)}) \neq c_p(T_0^{(\beta)})$.
Then $T_0^{(\gamma)}$ and $T_0^{(\beta)}$ are identical if 
we ignore the $p$th column of $T_0^{(\gamma)}$ and the $p$th and $(p+1)$st column of $T_0^{(\beta)}$.
And the $p$th column of $T_0^{(\gamma)}$ is obtained by merging 
the $p$th and $(p+1)$st column of $T_0^{(\beta)}$
in such a way that the entries are increasing from top to bottom.
Let the $p$th column of $T^{(\gamma)}_0$ be decomposed into column $p'$ throughout column $p''$ in $\kappa(T_0^{(\gamma)})$.

Let us investigate the horizontal strips $\tH_r$ in $\tau_{\kappa(T_0^{(\beta)})}$ and $\tau_{\kappa(T_0^{(\gamma)})}$.
By the definition of $p'$ and $p''$, one can easily observe the following:
\begin{enumerate}[label = {\rm (O\arabic*)}]
\item
For $r < p'$ or $r>p''$, 
$\tH_r$ in $\tau_{\kappa(T_0^{(\gamma)})}$ is identical to that in $\tau_{\kappa(T_0^{(\beta)})}$, that is,
$$
\tau_{\kappa(T_0^{(\gamma)})}(\tH_r) = \tau_{\kappa(T_0^{(\beta)})}(\tH_r).
$$

\item For $p' \leq r < p''$,
$\tH_r$ and $\tH_{r+1}$ 
do not share a column.
In addition, they are 
nonattacking in $\tcd(\alpha)$,
where $\alpha$ is the shape of $\tauE$.
Since $\tauE$ is a source tableau,
one can easily see that 
$\tH_{r+1}$ is strictly right 
of $\tH_{r}$, 
in other words,
the leftmost box of $\tH_{r+1}$ is strictly right of the rightmost box of $\tH_r$.
This tells us that all entries in column $p$ and column $p+1$ of $T_0^{(\beta)}$ appear in mutually distinct columns of $\tau_{\kappa(T_0^{(\beta)})}$.
By (O1), the same phenomenon also happens for column $p$ of $T_0^{(\gamma)}$.
\end{enumerate}

Now let us show that there are no quadruples $(i,j,s,t)$ in $\kappa(T_0^{(\gamma)})$ satisfying {\bf L1}.
To do this, we have only to consider $(i,j)$'s with $p' \leq i < j \leq p''$.
The observation (O2) shows that $y_{i,s} \neq y_{j,t}$ for all $s$ and $t$,
and thus our assertion is verified. 

Next, let us show that there are no quadruples $(i,j,s,t)$ in $\kappa(T_0^{(\gamma)})$ satisfying {\bf L2}, {\bf L3} or {\bf L4}.
We have three cases.
If $i, j \notin [p',p'']$, then our assertion follows from the induction hypothesis and the observation (O1).
If $i, j \in [p',p'']$, then our assertion follows from the observation (O2),
which is that $y_{i,s} \neq y_{j,t}$ for all $s$ and $t$.
Finally, let us deal with the remaining case. 
When $s < c_i(\kappa(T_0^{(\gamma)}))$, we have either 
$$
(\kappa(T_0^{(\gamma)}))_i^{-s} < (\kappa(T_0^{(\gamma)}))_j^{-t} 
\quad \text{for all $s$ and $t$,}
$$
or 
$$
(\kappa(T_0^{(\gamma)}))_i^{-s} > (\kappa(T_0^{(\gamma)}))_j^{-t}
\quad \text{for all $s$ and $t$},
$$
which implies that there are no quadruples $(i,j,s,t)$ satisfying {\bf L4}.

In the following, let $s = c_i(\kappa(T_0^{(\gamma)}))$
and 
suppose that 
there exists a positive integer $t$ satisfying that 
$$
(\kappa(T_0^{(\gamma)}))_{i}^{-s} > (\kappa(T_0^{(\gamma)}))_{j}^{-t},~ 
x_i < x_j,
\text{ and } y_{i,s} = y_{j,t}-1.
$$
The observation (O1)
implies that
\begin{equation}\label{Eq:A1}
(\kappa(T_0^{(\beta)}))_{i}^{-s} > (\kappa(T_0^{(\beta)}))_{j}^{-t}.
\end{equation}
Hence, from the induction hypothesis it follows that
$\tau_{\kappa(T_0^{(\beta)})}$ is an SPCT,
which again implies that the box $(x_i,y_{i,s}+1)$ is  in $\tcd(\alpha)$.
This means that there are no quadruples $(i,j,s,t)$ in $\kappa(T_0^{(\gamma)})$ satisfying {\bf L2}.
To deal with {\bf L3},  
recall that 
$i^-$ is the column index $<i$ of $T$
uniquely determined by the conditions $x_{i} = x_{i^-}$ and $y_{i^-,1} =y_{i,s}+1$.
We already showed that  $(x_i,y_{i,s}+1) \in \tcd(\bal_E)$,
which appears in the horizontal strip $\tH_{i^-}$ in $\tcd(\alpha)$.
Combining the induction 
hypothesis with~\eqref{Eq:A1},
we can derive the inequality
\begin{displaymath}
(\kappa(T_0^{(\beta)}))_{i^-}^{-s} > (\kappa(T_0^{(\beta)}))_{j}^{-t}.
\end{displaymath}
From the observation (O1)
we have 
\begin{displaymath}
(\kappa(T_0^{(\gamma)}))_{i^-}^{-s} > (\kappa(T_0^{(\gamma)}))_{j}^{-t},~
\end{displaymath}
this means that 
there are no quadruples $(i,j,s,t)$ satisfying {\bf L3}.
\end{proof}

With this preparation, the main result of this section can be stated as follows.
\begin{theorem}\label{Thm:projective cover S}
For an arbitrary class $E \in \mathcal{E}^\sigma(\alpha)$,
$\bbfP_{\bal_E}$ is the projective cover of $\bfS^\sigma_{\alpha,E}$. 
\end{theorem}
\begin{proof}
For the assertion we must show that $\eta: \bbfP_{\bal_E} \to \bfS^\sigma_{\alpha,E}$ is an essential epimorphism, equivalently, $\ker(\eta) \subset \rad(\bbfP_{\bal_E})$ by Lemma~\ref{lem: ess epi}.
Set
\[
\Theta^{\rm gen}_{\alpha,E} :=
\left\{
T \in \SRT(\bal_E) \cap \ker(\eta) \; \middle| \;
s_q \cdot T \in \SRT(\bal_E) \setminus \ker(\eta)  \text{ for some $1 \le q \le n-1$}
\right\}.
\]
In view of Theorem~\ref{Thm:Surjective homo} and Lemma~\ref{lem: kernel of eta}, 
one can see that $\Theta^{\rm gen}_{\alpha,E}$ generates $\ker(\eta)$.
Hence we have only to show that $\Theta^{\rm gen}_{\alpha,E} \subseteq \rad(\bbfP_{\bal_E})$.

Suppose that $T \in \Theta^{\rm gen}_{\alpha,E}$ and $T' := s_q \cdot T \in \SRT(\bal_E) \setminus \ker(\eta)$ for some $1\le q \le n-1$.
Let $\rho$ be a permutation with $T' = \pi_\rho \cdot \rtE$.
Choose a reduced expression $s_{q_{l-1}} s_{q_{l-2}} \cdots s_{q_1}$ for $\rho$ and set $s_{q_l}:= s_q$.
To see that $T \in \rad(\bbfP_{\bal_E})$,
by Lemma~\ref{Lem: To be contained in Rad},
it suffices to show that 
there exists $1 \leq r \leq l$ such that $q_r \in \Des(T_0^{(\bal^{\min}_E)})$.

If there exists $1 \leq r \leq l-1$ such that $q_r \in \Des(T_0^{(\bal_E^{\min})})$, 
then we are done.
We now assume that none of $q_r$ is a descent of $T_0^{(\bal_E^{\min})}$ for all $1 \leq r \leq l-1$. 
For simplicity, let $(\bal_E^{\min})^\rmc = (a_1, a_2, \ldots, a_k)$ and set 
\[
X := \kappa(T_0^{(\bal_E^{\min})}).
\]

Let $q \in [a_{0}+\cdots+a_{v} +1, a_{0}+\cdots+a_{v+1}]$ for some $0 \leq v \leq k-1$.
Since $\tau_X \in E$ by Lemma~\ref{Lem:T0 are SPCTs},
there exist $v'$ and $v''$ such that
\begin{equation}\label{eq: v max min}
\begin{aligned}
[a_{0}+\cdots+a_{v} +1, a_{0}+\cdots+a_{v+1}]
& = \tau_X(\tH_{v'}) \cup \tau_X(\tH_{v'+1}) \cup \cdots \cup  \tau_X(\tH_{v''})\\
& = \tau_{\rtE}(\tH_{v'}) \cup \tau_{\rtE}(\tH_{v'+1}) \cup \cdots \cup  \tau_{\rtE}(\tH_{v''})
\end{aligned}
\end{equation}
as a set. 
Here the second equality follows from (i) in the proof of Lemma~\ref{Lem:T0 are SPCTs} and the equality $\rtE = \kappa(T_0^{(\bal_E^{\max})})$.
On the other hand,
the assumption that $q_r \notin \Des(T_0^{(\bal_E^{\min})})$ for all $1 \leq r \leq l-1$ says that
\begin{equation*}
s_{q_r} \in \SG_{[1,a_1]} \times 
\SG_{[a_1 + 1, a_1 + a_2]} \times 
\cdots \times 
\SG_{[a_1+\cdots+a_{k-1} +1, n]} \subset \SG_n 
\quad \text{for all $1 \leq r \leq l-1$}.
\end{equation*}
Combining this property with~\eqref{eq: v max min}, we can also derive that   
\begin{align}\label{eq: interval and T'}
[a_{0}+\cdots+a_{v} +1, a_{0}+\cdots+a_{v+1}] =
\tau_{T'}(\tH_{v'}) \cup \tau_{T'}(\tH_{v'+1}) \cup \cdots \cup \tau_{T'}(\tH_{v''})
\end{align}
as a set.

Since $T \in \ker(\eta)$ and $T' \in \SRT(\bal_E) \setminus \ker(\eta)$, by Lemma~\ref{lem: kernel of eta}, we have the following three cases:
\begin{enumerate}[label = {\rm (\roman*)}]
\item 
there is a quadruple $(i,j,s,t)$ satisfying {\bf L1},
\[
T_i^{-s} = q+1, ~\text{and}~ T_j^{-t} = q, \quad \text{or}
\]

\item
there is a quadruple $(i,j,s,t)$ satisfying {\bf L2} or {\bf L3} such that
\[
T_{i}^{-s} = q+1 ~\text{and}~ T_j^{-t} = q,  \quad \text{or}
\]

\item
there is a quadruple $(i,j,s,t)$ satisfying {\bf L4} such that
\[
T_i^{-s} = q+1 ~ \text{ and } ~ T_j^{-t} = q.
\]
\end{enumerate}
Let $(x_{u}, y_{u,w})$ be the box in $\tau_{T'}$ corresponding to $T'^{-w}_u$ for $1 \leq w \leq c_u(T')$ via $\eta$ (see~\eqref{eq: x and y def}).
In (i), we have
\[
i<j,~(T')_i^{-s} = q,~(T')_j^{-t} = q+1,~y_{i,s} =  y_{j,t}, \text{ and } x_i \neq x_j.
\]
The following figure shows how $\tH_i$ and $\tH_j$ appear in $\tau_{T'}$.
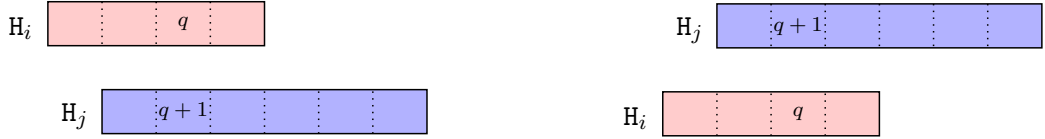
\begin{displaymath}
\begin{tikzpicture}[scale=0.8]
\def\hhh{8mm}
\def\vvv{6mm}
\def\www{0.20mm}
\draw[line width=\www, fill=red!20] (-\hhh*3,\vvv*0) rectangle (\hhh*1,\vvv*1);
\draw[line width=\www,dotted] (-\hhh*2,\vvv*1) -- (-\hhh*2,\vvv*0);
\draw[line width=\www,dotted] (-\hhh*1,\vvv*1) -- (-\hhh*1,\vvv*0);
\draw[line width=\www,dotted] (\hhh*0,\vvv*1) -- (\hhh*0,\vvv*0);
\draw[line width=\www, fill=blue!30] (-\hhh*2,-\vvv*2) rectangle (\hhh*4,-\vvv*1);
\draw[line width=\www,dotted] (-\hhh*1,-\vvv*2) -- (-\hhh*1,-\vvv*1);
\draw[line width=\www,dotted] (-\hhh*0,-\vvv*2) -- (-\hhh*0,-\vvv*1);
\draw[line width=\www,dotted] (\hhh*1,-\vvv*2) -- (\hhh*1,-\vvv*1);
\draw[line width=\www,dotted] (\hhh*2,-\vvv*2) -- (\hhh*2,-\vvv*1);
\draw[line width=\www,dotted] (\hhh*3,-\vvv*2) -- (\hhh*3,-\vvv*1);
\node[] at (-\hhh*3.5,\vvv*0.4) {\small $\tH_i$};
\node[] at (-\hhh*2.5,-\vvv*1.6) {\small $\tH_j$};
\node[] at (-\hhh*0.5,\vvv*0.5) {\tiny $q$};
\node[] at (-\hhh*0.5,-\vvv*1.5) {\tiny $q+1$};
\end{tikzpicture}
\begin{tikzpicture}[scale=0.8]
\def\hhh{8mm}
\def\vvv{-6mm}
\def\www{0.20mm}
\draw[line width=\www, fill=red!20] (-\hhh*3,\vvv*0) rectangle (\hhh*1,\vvv*1);
\draw[line width=\www,dotted] (-\hhh*2,\vvv*1) -- (-\hhh*2,\vvv*0);
\draw[line width=\www,dotted] (-\hhh*1,\vvv*1) -- (-\hhh*1,\vvv*0);
\draw[line width=\www,dotted] (\hhh*0,\vvv*1) -- (\hhh*0,\vvv*0);
\draw[line width=\www, fill=blue!30] (-\hhh*2,-\vvv*2) rectangle (\hhh*4,-\vvv*1);
\draw[line width=\www,dotted] (-\hhh*1,-\vvv*2) -- (-\hhh*1,-\vvv*1);
\draw[line width=\www,dotted] (-\hhh*0,-\vvv*2) -- (-\hhh*0,-\vvv*1);
\draw[line width=\www,dotted] (\hhh*1,-\vvv*2) -- (\hhh*1,-\vvv*1);
\draw[line width=\www,dotted] (\hhh*2,-\vvv*2) -- (\hhh*2,-\vvv*1);
\draw[line width=\www,dotted] (\hhh*3,-\vvv*2) -- (\hhh*3,-\vvv*1);
\node[] at (-\hhh*3.5,\vvv*0.6) {\small $\tH_i$};
\node[] at (-\hhh*2.5,-\vvv*1.4) {\small $\tH_j$};
\node[] at (-\hhh*0.5,\vvv*0.5) {\tiny $q$};
\node[] at (-\hhh*0.5,-\vvv*1.5) {\tiny $q+1$};
\node[] at (-\hhh*6,-\vvv*0) {};
\end{tikzpicture}
\end{displaymath}
In (ii), we have 
$$
(T')_i^{-s} = q, ~ (T')_j^{-t} = q+1, ~ 
x_i < x_j, ~ \text{ and } ~ \alpha_{x_i} = y_{i,s} = y_{j,t} -1
$$
or 
$$
(T')_i^{-s} = q, ~  (T')_j^{-t} = q+1, ~ (T')_{i^-}^{-1} < q, 
x_i < x_j, ~  \text{ and } ~
\alpha_{x_i} > y_{i,s} = y_{j,t}-1.
$$
The following figures show how $\tH_i$ and $\tH_j$ appear in $\tau_{T'}$. 
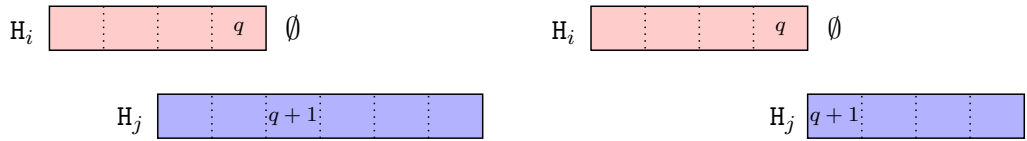
\begin{displaymath}
\begin{tikzpicture}[scale=0.8]
\def\hhh{8mm}
\def\vvv{6mm}
\def\www{0.20mm}
\def\hhhh{80mm}
\draw[line width=\www, fill=red!20] (-\hhh*3,\vvv*0) rectangle (\hhh*1,\vvv*1);
\draw[line width=\www,dotted] (-\hhh*2,\vvv*1) -- (-\hhh*2,\vvv*0);
\draw[line width=\www,dotted] (-\hhh*1,\vvv*1) -- (-\hhh*1,\vvv*0);
\draw[line width=\www,dotted] (\hhh*0,\vvv*1) -- (\hhh*0,\vvv*0);
\draw[line width=\www, fill=blue!30] (-\hhh*1,-\vvv*2) rectangle (\hhh*5,-\vvv*1);
\draw[line width=\www,dotted] (-\hhh*0,-\vvv*2) -- (-\hhh*0,-\vvv*1);
\draw[line width=\www,dotted] (\hhh*1,-\vvv*2) -- (\hhh*1,-\vvv*1);
\draw[line width=\www,dotted] (\hhh*2,-\vvv*2) -- (\hhh*2,-\vvv*1);
\draw[line width=\www,dotted] (\hhh*3,-\vvv*2) -- (\hhh*3,-\vvv*1);
\draw[line width=\www,dotted] (\hhh*4,-\vvv*2) -- (\hhh*4,-\vvv*1);
\node[] at (-\hhh*3.5,\vvv*0.4) {\small $\tH_i$};
\node[] at (-\hhh*1.5,-\vvv*1.6) {\small $\tH_j$};
\node[] at (\hhh*0.5,\vvv*0.5) {\tiny $q$};
\node[] at (\hhh*1.5,\vvv*0.5) {$\emptyset$};
\node[] at (\hhh*1.5,-\vvv*1.5) {\tiny $q+1$};
\draw[line width=\www, fill=red!20] (-\hhh*3+\hhhh,\vvv*0) rectangle (\hhh*1+\hhhh,\vvv*1);
\draw[line width=\www,dotted] (-\hhh*2+\hhhh,\vvv*1) -- (-\hhh*2+\hhhh,\vvv*0);
\draw[line width=\www,dotted] (-\hhh*1+\hhhh,\vvv*1) -- (-\hhh*1+\hhhh,\vvv*0);
\draw[line width=\www,dotted] (\hhh*0+\hhhh,\vvv*1) -- (\hhh*0+\hhhh,\vvv*0);
\draw[line width=\www, fill=blue!30] (\hhh*1+\hhhh,-\vvv*2) rectangle (\hhh*5+\hhhh,-\vvv*1);
\draw[line width=\www,dotted] (\hhh*2+\hhhh,-\vvv*2) -- (\hhh*2+\hhhh,-\vvv*1);
\draw[line width=\www,dotted] (\hhh*3+\hhhh,-\vvv*2) -- (\hhh*3+\hhhh,-\vvv*1);
\draw[line width=\www,dotted] (\hhh*4+\hhhh,-\vvv*2) -- (\hhh*4+\hhhh,-\vvv*1);
\node[] at (-\hhh*3.5+\hhhh,\vvv*0.4) {\small $\tH_i$};
\node[left] at (\hhh*1+\hhhh,-\vvv*1.6) {\small $\tH_j$};
\node[] at (\hhh*0.5+\hhhh,\vvv*0.5) {\tiny $q$};
\node[] at (\hhh*1.5+\hhhh,\vvv*0.5) {\small $\emptyset$};
\node[] at (\hhh*1.5+\hhhh,-\vvv*1.5) {\tiny $q+1$};
\end{tikzpicture}
\end{displaymath}
\vskip 3mm
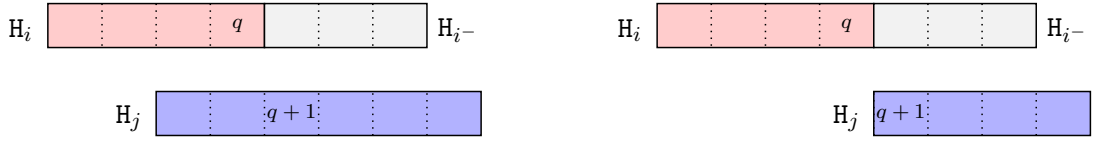
\begin{displaymath}
\begin{tikzpicture}[scale=0.8]
\def\hhh{8mm}
\def\vvv{6mm}
\def\www{0.20mm}
\def\hhhh{80mm}
\draw[line width=\www, fill=red!20] (-\hhh*3,\vvv*0) rectangle (\hhh*1,\vvv*1);
\draw[line width=\www,dotted] (-\hhh*2,\vvv*1) -- (-\hhh*2,\vvv*0);
\draw[line width=\www,dotted] (-\hhh*1,\vvv*1) -- (-\hhh*1,\vvv*0);
\draw[line width=\www,dotted] (\hhh*0,\vvv*1) -- (\hhh*0,\vvv*0);
\draw[line width=\www, fill=gray!10] (\hhh*1,\vvv*0) rectangle (\hhh*4,\vvv*1);
\draw[line width=\www,dotted] (\hhh*2,\vvv*1) -- (\hhh*2,\vvv*0);
\draw[line width=\www,dotted] (\hhh*3,\vvv*1) -- (\hhh*3,\vvv*0);
\draw[line width=\www, fill=blue!30] (-\hhh*1,-\vvv*2) rectangle (\hhh*5,-\vvv*1);
\draw[line width=\www,dotted] (-\hhh*0,-\vvv*2) -- (-\hhh*0,-\vvv*1);
\draw[line width=\www,dotted] (\hhh*1,-\vvv*2) -- (\hhh*1,-\vvv*1);
\draw[line width=\www,dotted] (\hhh*2,-\vvv*2) -- (\hhh*2,-\vvv*1);
\draw[line width=\www,dotted] (\hhh*3,-\vvv*2) -- (\hhh*3,-\vvv*1);
\draw[line width=\www,dotted] (\hhh*4,-\vvv*2) -- (\hhh*4,-\vvv*1);
\node[] at (-\hhh*3.5,\vvv*0.4) {\small $\tH_i$};
\node[right] at (\hhh*4,\vvv*0.4) {\small $\tH_{i^-}$};
\node[] at (-\hhh*1.5,-\vvv*1.6) {\small $\tH_j$};
\node[] at (\hhh*0.5,\vvv*0.5) {\tiny $q$};
\node[] at (\hhh*1.5,-\vvv*1.5) {\tiny $q+1$};
\draw[line width=\www, fill=red!20] (-\hhh*3+\hhhh,\vvv*0) rectangle (\hhh*1+\hhhh,\vvv*1);
\draw[line width=\www,dotted] (-\hhh*2+\hhhh,\vvv*1) -- (-\hhh*2+\hhhh,\vvv*0);
\draw[line width=\www,dotted] (-\hhh*1+\hhhh,\vvv*1) -- (-\hhh*1+\hhhh,\vvv*0);
\draw[line width=\www,dotted] (\hhh*0+\hhhh,\vvv*1) -- (\hhh*0+\hhhh,\vvv*0);
\draw[line width=\www, fill=gray!10] (\hhh*1+\hhhh,\vvv*0) rectangle (\hhh*4+\hhhh,\vvv*1);
\draw[line width=\www,dotted] (\hhh*2+\hhhh,\vvv*1) -- (\hhh*2+\hhhh,\vvv*0);
\draw[line width=\www,dotted] (\hhh*3+\hhhh,\vvv*1) -- (\hhh*3+\hhhh,\vvv*0);
\draw[line width=\www, fill=blue!30] (\hhh*1+\hhhh,-\vvv*2) rectangle (\hhh*5+\hhhh,-\vvv*1);
\draw[line width=\www,dotted] (\hhh*1+\hhhh,-\vvv*2) -- (\hhh*1+\hhhh,-\vvv*1);
\draw[line width=\www,dotted] (\hhh*2+\hhhh,-\vvv*2) -- (\hhh*2+\hhhh,-\vvv*1);
\draw[line width=\www,dotted] (\hhh*3+\hhhh,-\vvv*2) -- (\hhh*3+\hhhh,-\vvv*1);
\draw[line width=\www,dotted] (\hhh*4+\hhhh,-\vvv*2) -- (\hhh*4+\hhhh,-\vvv*1);
\node[] at (-\hhh*3.5+\hhhh,\vvv*0.4) {\small $\tH_i$};
\node[right] at (\hhh*4+\hhhh,\vvv*0.4) {\small $\tH_{i^-}$};
\node[] at (\hhh*0.5+\hhhh,-\vvv*1.6) {\small $\tH_j$};
\node[] at (\hhh*0.5+\hhhh,\vvv*0.5) {\tiny $q$};
\node[] at (\hhh*1.5+\hhhh,-\vvv*1.5) {\tiny $q+1$};
\end{tikzpicture}
\end{displaymath}
In (iii), we have
\[
(T')_i^{-s} = q,~ (T')_j^{-t} = q+1,  \text{ and } y_{i,s} = y_{j,t}-1.
\]
The following figure shows how $\tH_i$ and $\tH_j$ appear in $\tau_{T'}$. 
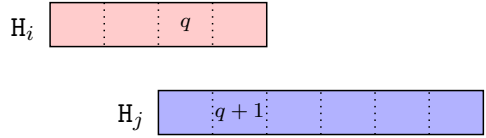
\begin{displaymath}
\begin{tikzpicture}[scale=0.8]
\def\hhh{8mm}
\def\vvv{6mm}
\def\www{0.20mm}
\draw[line width=\www, fill=red!20] (-\hhh*3,\vvv*0) rectangle (\hhh*1,\vvv*1);
\draw[line width=\www,dotted] (-\hhh*2,\vvv*1) -- (-\hhh*2,\vvv*0);
\draw[line width=\www,dotted] (-\hhh*1,\vvv*1) -- (-\hhh*1,\vvv*0);
\draw[line width=\www,dotted] (\hhh*0,\vvv*1) -- (\hhh*0,\vvv*0);
\draw[line width=\www, fill=blue!30] (-\hhh*1,-\vvv*2) rectangle (\hhh*5,-\vvv*1);
\draw[line width=\www,dotted] (-\hhh*0,-\vvv*2) -- (-\hhh*0,-\vvv*1);
\draw[line width=\www,dotted] (\hhh*1,-\vvv*2) -- (\hhh*1,-\vvv*1);
\draw[line width=\www,dotted] (\hhh*2,-\vvv*2) -- (\hhh*2,-\vvv*1);
\draw[line width=\www,dotted] (\hhh*3,-\vvv*2) -- (\hhh*3,-\vvv*1);
\draw[line width=\www,dotted] (\hhh*4,-\vvv*2) -- (\hhh*4,-\vvv*1);
\node[] at (-\hhh*3.5,\vvv*0.4) {\small $\tH_i$};
\node[] at (-\hhh*1.5,-\vvv*1.6) {\small $\tH_j$};
\node[] at (-\hhh*0.5,\vvv*0.5) {\tiny $q$};
\node[] at (\hhh*0.5,-\vvv*1.5) {\tiny $q+1$};
\end{tikzpicture}
\end{displaymath}
In all cases, we can observe that
\begin{align}\label{eq: observation q q+1}
\text{$\tH_i$ and $\tH_j$ share a column
or 
$\tH_i, \tH_j$ are attacking in $\tcd(\alpha)$.}
\end{align}

Recall that $q \in [a_{0}+\cdots+a_{v} +1, a_{0}+\cdots+a_{v+1}]$
and the index $i$ lies in $[v',v'']$
by~\eqref{eq: interval and T'}.
We now claim that $q+1$ is not contained in $[a_{0}+\cdots+a_{v} +1, a_{0}+\cdots+a_{v+1}]$.
Otherwise, $q+1$ is contained in $\tau_{T'}(\tH_{v'}) \cup \tau_{T'}(\tH_{v'+1}) \cup \cdots \cup \tau_{T'}(\tH_{v''})$ by~\eqref{eq: interval and T'},
so $j \in [v',v'']$.
Applying the observation (O2) in the proof of Lemma~\ref{Lem:T0 are SPCTs} repeatedly,
we deduce that
$\tH_{\mathbf{i}}$ and $\tH_{j}$ do not share any column 
and 
$\tH_i, \tH_j$ are not attacking.
But this gives a contradiction to~\eqref{eq: observation q q+1}, which verifies the claim.
This implies that $q = a_{0}+\cdots+a_{v+1}$.
Thus we conclude that $q$ is a descent of $T_0^{(\bal_E^{\min})}$, as required.
\end{proof}

The following corollary is an immediate consequence of Theorem~\ref{Thm:projective cover S}.
\begin{corollary}{\rm (cf.~\cite[Theorem 5.5]{20CKNO})}
\label{Cor: Classification covers}
Let $E$ be a class in $\mathcal{E}^\sigma(\alpha)$
whose source tableau has $m$ descents.
Then the projective cover of $\bfS^\sigma_{\alpha,E}$ is indecomposable
if and only if
for all $1 \leq i \leq m$,    
$\tH_i$ and $\tH_{i+1}$ share a column 
or $d_{i-1}+1$ and $d_{i+1}$ are attacking in $\tauE$.
In particular, 
every canonical class satisfies this condition.
\end{corollary}

\ytableausetup{mathmode, boxsize=1em}

The following figure shows how $\eta:\bbfP_{(1)\oplus(2,2)} \ra \bfS^{231}_{(2,2,1),E}$ works on the level of bases, where $E$ is the class having
\[
\tau_E =
\begin{array}{l}
\begin{ytableau}
4 & 3 \\
5 & 2 \\
1
\end{ytableau}
\end{array}
\]
as the source tableau.
The shaded part indicates the basis elements whose image is nonzero 
and the unshaded part indicates the basis elements whose image is zero.
We are here omitting the arrows acting as the identity as well as the zero map.
\begin{figure}[h]
\begin{displaymath}
\begin{tikzpicture}
\def \hhh{33mm}
\def \vvv{21mm}
\def \shh{13mm}
\def \svv{7mm}
\node at (0,0) (A1) {};
\node at (-\hhh,-\vvv*1) (B1) {};
\node at (-\hhh*0,-\vvv*1) (B2) {};
\node at (\hhh*1,-\vvv*1) (B3) {};
\node at (-\hhh*1.5,-\vvv*2) (C1) {};
\node at (-\hhh*0.5,-\vvv*2) (C2) {};
\node at (\hhh*0.5,-\vvv*2) (C3) {};
\node at (\hhh*1.5,-\vvv*2) (C4) {};
\node at (-\hhh*2,-\vvv*3) (O1) {};
\node at (-\hhh*1,-\vvv*3) (O2) {};
\node at (\hhh*0,-\vvv*3) (D3) {};
\node at (\hhh*1,-\vvv*3) (D4) {};
\node at (\hhh*2,-\vvv*3) (D5) {};
\node at (-\hhh*2,-\vvv*4) (E1) {};
\node at (-\hhh*1,-\vvv*4) (E2) {};
\node at (\hhh*0,-\vvv*4) (E3) {};
\node at (\hhh*1,-\vvv*4) (E4) {};
\node at (\hhh*2,-\vvv*4) (E5) {};
\node at (-\hhh*1.5,-\vvv*5) (F1) {};
\node at (-\hhh*0.5,-\vvv*5) (F2) {};
\node at (\hhh*0.5,-\vvv*5) (F3) {};
\node at (\hhh*1.5,-\vvv*5) (F4) {};
\node at (-\hhh*0.5,-\vvv*6) (G1) {};
\node at (\hhh*0.5,-\vvv*6) (G2) {};
\node at (-\hhh*0,-\vvv*7) (H1) {};
\node at (A1) {
\begin{ytableau}
\none & \none & *(blue!20) 3 & *(blue!20) 5 \\
\none & *(blue!20) 2 & *(blue!20) 4 \\
*(blue!20) 1
\end{ytableau}};
\node at (B1) {
\begin{ytableau}
\none & \none & *(blue!20) 3 & *(blue!20) 5 \\
\none & *(blue!20) 1 & *(blue!20) 4 \\
*(blue!20) 2
\end{ytableau}};
\node at (B2) {
\begin{ytableau}
\none & \none & 2 & 5 \\
\none & 3 & 4 \\
1
\end{ytableau}};
\node at (B3) {
\begin{ytableau}
\none & \none & 3 & 4 \\
\none & 2 & 5 \\
1
\end{ytableau}};
\node at (C1) {
\begin{ytableau}
\none & \none & *(blue!20) 2 & *(blue!20) 5 \\
\none & *(blue!20) 1 & *(blue!20) 4 \\
*(blue!20) 3
\end{ytableau}};
\node at (C2) {
\begin{ytableau}
\none & \none & 3 & 4 \\
\none & 1 & 5 \\
2
\end{ytableau}};
\node at (C3) {
\begin{ytableau}
\none & \none & 1 & 5 \\
\none & 3 & 4 \\
2
\end{ytableau}};
\node at (C4) {
\begin{ytableau}
\none & \none & 2 & 4 \\
\none & 3 & 5 \\
1
\end{ytableau}};
\node at (O1) {
\begin{ytableau}
\none & \none & 2 & 5 \\
\none & 1 & 3 \\
4
\end{ytableau}};
\node at (O2) {
\begin{ytableau}
\none & \none & 2 & 4 \\
\none & 1 & 5 \\
3
\end{ytableau}};
\node at (D3) {
\begin{ytableau}
\none & \none & 1 & 5 \\
\none & 2 & 4 \\
3
\end{ytableau}};
\node at (D4) {
\begin{ytableau}
\none & \none & 1 & 4 \\
\none & 3 & 5 \\
2
\end{ytableau}};
\node at (D5) {
\begin{ytableau}
\none & \none & 2 & 3 \\
\none & 4 & 5 \\
1
\end{ytableau}};
\node at (E1) {
\begin{ytableau}
\none & \none & 2 & 4 \\
\none & 1 & 3 \\
5
\end{ytableau}};
\node at (E2) {
\begin{ytableau}
\none & \none & 1 & 5 \\
\none & 2 & 3 \\
4
\end{ytableau}};
\node at (E3) {
\begin{ytableau}
\none & \none & 2 & 3 \\
\none & 1 & 5 \\
4
\end{ytableau}};
\node at (E4) {
\begin{ytableau}
\none & \none & 1 & 4 \\
\none & 2 & 5 \\
3
\end{ytableau}};
\node at (E5) {
\begin{ytableau}
\none & \none & 1 & 3 \\
\none & 4 & 5 \\
2
\end{ytableau}};
\node at (F1) {
\begin{ytableau}
\none & \none & 2 & 3 \\
\none & 1 & 4 \\
5
\end{ytableau}};
\node at (F2) {
\begin{ytableau}
\none & \none & 1 & 4 \\
\none & 2 & 3 \\
5
\end{ytableau}};
\node at (F3) {
\begin{ytableau}
\none & \none & 1 & 3 \\
\none & 2 & 5 \\
4
\end{ytableau}};
\node at (F4) {
\begin{ytableau}
\none & \none & 1 & 2 \\
\none & 4 & 5 \\
3
\end{ytableau}};
\node at (G1) {
\begin{ytableau}
\none & \none & 1 & 3 \\
\none & 2 & 4 \\
5
\end{ytableau}};
\node at (G2) {
\begin{ytableau}
\none & \none & 1 & 2 \\
\none & 3 & 5 \\
4
\end{ytableau}};
\node at (H1) {
\begin{ytableau}
\none & \none & 1 & 2 \\
\none & 3 & 4 \\
5
\end{ytableau}};
\draw[->] ([xshift=-\shh,yshift=-\svv]A1.south) -- ([xshift=+\shh,yshift=+\svv]B1.north) node[above,midway,sloped] {\footnotesize $\pi_1$};
\draw[->, shorten <=8mm] ([yshift=\svv*0.1]A1.south) -- ([yshift=\svv]B2.north) node[right,yshift=\svv*0.2] {\footnotesize $\pi_2$};
\draw[->, shorten <=8mm] ([xshift=\shh*0.3,yshift=-\svv*0.2]A1.south) -- ([xshift=-\shh*0.3,yshift=\svv]B3.north) node[above,midway,sloped] {\footnotesize $\pi_4$};
\draw[->] ([xshift=-\shh*0.5,yshift=-\svv]B1.south) -- ([xshift=8mm,yshift=+\svv]C1.north) node[left,midway] {\footnotesize $\pi_2$};
\draw[->, shorten <=8mm] ([yshift=\svv*0.1]B1.south) -- ([xshift=-5mm,yshift=\svv]C2.north) node[pos=0.8,left] {\footnotesize $\pi_4$};
\draw[->] ([xshift=0,yshift=-\svv]B2.south) -- ([xshift=0,yshift=+\svv]C3.north) node[pos=0.6, below] {\footnotesize $\pi_1$};
\draw[->, shorten <=8mm] ([xshift=-\shh*0.4,yshift=-\svv*0.5]B2.south) -- ([xshift=-\shh*0.5,yshift=\svv]C4.north) node[below,midway,yshift=0] {\footnotesize $\pi_4$};
\draw[->] ([xshift=-\shh,yshift=-\svv*0.7]B3.south) -- ([xshift=+\shh,yshift=+\svv]C2.north) node[above,sloped,pos=0.2] {\footnotesize $\pi_1$};
\draw[->] ([xshift=\shh*0.4,yshift=-\svv*0.4]B3.south) -- ([xshift=-\shh*0.2,yshift=+\svv]C4.north) node[right,midway] {\footnotesize $\pi_2$};
\draw[->] ([xshift=-\shh*0.5,yshift=-\svv]C1.south) -- ([xshift=8mm,yshift=+\svv]O1.north) node[left,midway] {\footnotesize $\pi_3$};
\draw[->, shorten <=8mm] ([yshift=\svv*0.1]C1.south) -- ([xshift=-5mm,yshift=\svv]O2.north) node[pos=0.8,left] {\footnotesize $\pi_4$};
\draw[->, shorten <=8mm] ([yshift=-\svv*0.1]C1.south) -- ([xshift=-7mm,yshift=\svv*0.9]D3.north) node[pos=0.8,below] {\footnotesize $\pi_1$};
\draw[->] ([xshift=-\shh*0.6,yshift=-\svv*0.9]C2.south) -- ([xshift=\shh*0.4,yshift=\svv*0.9]O2.north) node[pos=0.7,right] {\footnotesize $\pi_2$};
\draw[->, shorten <=8mm] ([xshift=-\shh*0.1,yshift=-\svv*0.1]C3.south) -- ([xshift=\shh*0.7,yshift=\svv*0.9]D3.north) node[pos=0.8,right] {\footnotesize $\pi_2$};
\draw[->] ([xshift=-\shh*0.1,yshift=-\svv*0.5]C3.south) -- ([xshift=-\shh*0.3,yshift=\svv*0.8]D4.north) node[above,sloped,pos=0.4] {\footnotesize $\pi_4$};
\draw[->] ([xshift=-\shh*0.5,yshift=-\svv*0.9]C4.south) -- ([xshift=\shh*0.5,yshift=\svv]D4.north) node[right,midway] {\footnotesize $\pi_1$};
\draw[->] ([xshift=-\shh*0.2,yshift=-\svv*0.5]C4.south) -- ([xshift=-\shh*0.8,yshift=+\svv]D5.north) node[right,midway] {\footnotesize $\pi_3$};
\draw[->] ([xshift=\shh*0,yshift=-\svv]O1.south) -- ([xshift=\shh*0,yshift=+\svv]E1.north) node[left,midway] {\footnotesize $\pi_4$};
\draw[->, shorten <=8mm] ([yshift=\svv*0.1]O1.south) -- ([xshift=-5mm,yshift=\svv]E2.north) node[pos=0.8,left] {\footnotesize $\pi_1$};
\draw[->, shorten <=8mm] ([xshift=-\shh*0.2,yshift=-\svv*0.1]O2.south) -- ([xshift=-\shh*0.5,yshift=\svv*1]E3.north) node[pos=0.7,below] {\footnotesize $\pi_3$};
\draw[->] ([xshift=\shh*0.5,yshift=-\svv*0.5]O2.south) -- ([xshift=-\shh*0.5,yshift=\svv*0.7]E4.north) node[pos=0.5,below] {\footnotesize $\pi_1$};
\draw[->] ([xshift=-\shh*0.9,yshift=-\svv*0.8]D3.south) -- ([xshift=\shh*0.8,yshift=\svv*1]E2.north) node[pos=0.9,above] {\footnotesize $\pi_3$};
\draw[->] ([xshift=-\shh*0.1,yshift=-\svv*0.5]D3.south) -- ([xshift=-\shh*0.3,yshift=\svv*0.8]E4.north) node[above,sloped,pos=0.4] {\footnotesize $\pi_4$};
\draw[->] ([xshift=-\shh*0.0,yshift=-\svv*0.7]D4.south) -- ([xshift=\shh*0.2,yshift=\svv]E4.north) node[right,midway] {\footnotesize $\pi_2$};
\draw[->] ([xshift=\shh*0.2,yshift=-\svv*0.5]D4.south) -- ([xshift=-\shh*0.3,yshift=+\svv]E5.north) node[right,midway] {\footnotesize $\pi_3$};
\draw[->] ([xshift=\shh*0.2,yshift=-\svv*0.5]D5.south) -- ([xshift=\shh*0.2,yshift=+\svv]E5.north) node[right,midway] {\footnotesize $\pi_1$};
\draw[->] ([xshift=\shh*0,yshift=-\svv]E1.south) -- ([xshift=\shh*0,yshift=+\svv]F1.north) node[below,midway] {\footnotesize $\pi_3$};
\draw[->, shorten <=8mm] ([yshift=-\svv*0.1]E1.south) -- ([xshift=-5mm,yshift=\svv*0.8]F2.north) node[pos=0.4,below] {\footnotesize $\pi_1$};
\draw[->, shorten <=8mm] ([xshift=-\shh*0.2,yshift=-\svv*0.1]E2.south) -- ([xshift=-\shh*0,yshift=\svv*0.9]F2.north) node[pos=0.5,right] {\footnotesize $\pi_4$};
\draw[->] ([xshift=-\shh*1,yshift=-\svv*1]E3.south) -- ([xshift=\shh*1,yshift=\svv*0.9]F1.north) node[pos=0.1,below] {\footnotesize $\pi_4$};
\draw[->] ([xshift=\shh*0.2,yshift=-\svv*0.8]E3.south) -- ([xshift=-\shh*0.2,yshift=\svv*1]F3.north) node[pos=0.8,above] {\footnotesize $\pi_1$};
\draw[->] ([xshift=-\shh*0.1,yshift=-\svv*0.5]E4.south) -- ([xshift=\shh*0.7,yshift=\svv*1]F3.north) node[right,pos=0.6] {\footnotesize $\pi_3$};
\draw[->] ([xshift=-\shh*0.0,yshift=-\svv*0.7]E5.south) -- ([xshift=\shh*0.6,yshift=\svv*1]F4.north) node[right,midway] {\footnotesize $\pi_2$};
\draw[->] ([xshift=\shh*0.5,yshift=-\svv*0.7]F1.south) -- ([xshift=-\shh*0.3,yshift=+\svv]G1.north) node[below,midway] {\footnotesize $\pi_1$};
\draw[->, shorten <=8mm] ([yshift=\svv*0.3]F2.south) -- ([xshift=\shh*0.1,yshift=\svv*0.9]G1.north) node[pos=0.7,right] {\footnotesize $\pi_3$};
\draw[->, shorten <=8mm] ([xshift=-\shh*0.2,yshift=-\svv*0.5]F3.south) -- ([xshift=\shh*0.9,yshift=\svv*0.9]G1.north) node[pos=0.7,above] {\footnotesize $\pi_4$};
\draw[->] ([xshift=\shh*0.5,yshift=-\svv*0.7]F3.south) -- ([xshift=\shh*0.5,yshift=\svv*0.9]G2.north) node[pos=0.5,right] {\footnotesize $\pi_2$};
\draw[->] ([xshift=-\shh*0.9,yshift=-\svv*0.9]F4.south) -- ([xshift=\shh*1,yshift=\svv*0.9]G2.north) node[pos=0.7,right] {\footnotesize $\pi_3$};
\draw[->] ([xshift=\shh*0.1,yshift=-\svv*0.5]
G1.south) -- ([xshift=\shh*0,yshift=\svv*1]H1.north) node[left,pos=0.6] {\footnotesize $\pi_2$};
\draw[->] ([xshift=-\shh*0.0,yshift=-\svv*0.7]G2.south) -- ([xshift=\shh*0.8,yshift=\svv*1]H1.north) node[right,midway] {\footnotesize $\pi_4$};
\draw[-,dotted,red,line width=0.30mm] (-\shh*1.2,\svv*1.2) -- (\shh*6.2,\svv*1.2) -- (\shh*6.2,-\svv*16.2) -- (\shh*2.5,-\svv*16.2) -- (\shh*2.5,-\svv*14.2) -- (\shh*4,-\svv*10.5) -- (\shh*1.5,-\svv*10.5) -- (\shh*1.5,-\svv*7.7) -- (\shh*0,-\svv*7.7) -- (\shh*0,-\svv*4.4) -- (-\shh*1.2,-\svv*4.4) -- (-\shh*1.2,\svv*1.2); 
\draw[-,dotted,red,line width=0.30mm] (-\shh*1.2,-\svv*1.6) -- (-\shh*6,-\svv*1.6) -- (-\shh*6,-\svv*22.3) -- (\shh*2.3,-\svv*22.3) -- (\shh*2.3,-\svv*16.2) -- (\shh*2.5,-\svv*16.2); 
\node[left] at (\shh*3.5,-\svv*0) {\small $\cong \bbfP_{(3,2)}$};
\node[left] at (-\shh*4,-\svv*2.5) {\small $\bbfP_{(1,2,2)} \cong$};
\end{tikzpicture}
\end{displaymath}
\caption{$\eta:\bbfP_{(1)\oplus(2,2)} \ra \bfS^{231}_{(2,2,1),E}$}
\end{figure}
\newpage

We close this section with the remark concerned with future research.
\begin{remark} \hfill
\begin{enumerate}[label = {\rm (\roman*)}]
\item Given a generalized composition $\bgam$,
it would be very nice to characterize the triples $(\alpha,\sigma,E)$
such that $\bgam = \bal_E$.
It enables us to classify all $\bfS^\sigma_{\alpha,E}$'s having $\bbfP_\bgam$ as the projective cover.

\item Let $M$ be a finitely generated module. 
It is well known that the projective presentation of $M$ contains important information on the structure of $M$.
We expect that Lemma~\ref{lem: kernel of eta} and Theorem~\ref{Thm:projective cover S} play a substantial role in 
constructing the projective presentation of $\bfS^\sigma_{\alpha,E}$. 
\end{enumerate}
\end{remark}

\vspace*{10mm}

\bibliographystyle{amsplain}

\end{document}